\documentclass[12pt,a4paper]{amsart}
\pdfoutput=1
\usepackage{amscd,amsmath,amssymb,amsfonts,times}
\usepackage{mathrsfs}
\usepackage[dvips]{color} 
\usepackage[all]{xy}
\usepackage{amsthm,miyamath,tikz-cd}
\mathchardef\hyphen="2D

\numberwithin{equation}{section}
\theoremstyle{slanted}
    \newtheorem{thm}{Theorem}[section]
    \newtheorem{lem}[thm]{Lemma}
    
    \newtheorem{prop}[thm]{Proposition}
    
    \newtheorem{cor}[thm]{Corollary}

\theoremstyle{normal}
    \newtheorem{cond}[thm]{Condition}
    \newtheorem{assump}[thm]{Assumption}
    \newtheorem{sett}[thm]{Setting}
    \newtheorem{defn}[thm]{Definition}
    
    \newtheorem{rem}[thm]{Remark}
\DeclareMathAlphabet{\mathpzc}{OT1}{pzc}{m}{it}
\newenvironment{pf}{\par\smallskip\noindent\emph{Proof.}}{\hfill\qed\par\smallskip}
\newenvironment{pf*}[1]{\par\smallskip\noindent\emph{#1.}}{\hfill\qed\par\smallskip}
\begin{document}
\title{Milnor $K$-theory, $F$-isocrystals and syntomic regulators}
\author{Masanori Asakura}
\address{Department of Mathematics, Hokkaido University, Sapporo, 060-0810 Japan}
\email{asakura@math.sci.hokudai.ac.jp}
\author{Kazuaki Miyatani}
\address{School of Science and Technology for Future Life, Tokyo Denki University, Tokyo, 120-8551 Japan}
\email{miyatani@mail.dendai.ac.jp}
\date\today
\maketitle

\def\can{\mathrm{can}}
\def\ch{{\mathrm{ch}}}
\def\Coker{\mathrm{Coker}}
\def\crys{\mathrm{crys}}
\def\dlog{d{\mathrm{log}}}
\def\dR{{\mathrm{d\hspace{-0.2pt}R}}}            
\def\et{{\mathrm{\acute{e}t}}}  
\def\Frac{{\mathrm{Frac}}}
\def\phami{\phantom{-}}
\def\id{{\mathrm{id}}}              
\def\Image{{\mathrm{Im}}}        
\def\Hom{{\mathrm{Hom}}}  
\def\Ext{{\mathrm{Ext}}}
\def\MHS{{\mathrm{MHS}}}  
  
\def\Ker{{\mathrm{Ker}}}          
\def\rig{{\mathrm{rig}}}
\def\Pic{{\mathrm{Pic}}}
\def\CH{{\mathrm{CH}}}
\def\NS{{\mathrm{NS}}}
\def\Fil{{\mathrm{Fil}}}
\def\End{{\mathrm{End}}}
\def\pr{{\mathrm{pr}}}
\def\gp{{\mathrm{gp}}}
\def\Proj{{\mathrm{Proj}}}
\def\ord{{\mathrm{ord}}}
\def\reg{{\mathrm{reg}}}          %
\def\res{{\mathrm{res}}}          %
\def\Res{\mathrm{Res}}
\def\Spec{\operatorname{Spec}}     
\def\syn{{\mathrm{syn}}}    
\def\sym{{\mathrm{sym}}}
\def\cont{{\mathrm{cont}}}
\def\zar{{\mathrm{zar}}}
\def\red{{\mathrm{red}}}
\def\bA{{\mathbb A}}
\def\bC{{\mathbb C}}
\def\C{{\mathbb C}}
\def\G{{\mathbb G}}
\def\bE{{\mathbf E}}
\def\bF{{\mathbb F}}
\def\F{{\mathbb F}}
\def\bG{{\mathbb G}}
\def\bH{{\mathbb H}}
\def\bJ{{\mathbf J}}
\def\bL{{\mathbf L}}
\def\bO{{\mathbf O}}
\def\cL{{\mathscr L}}
\def\bN{{\mathbb N}}
\def\bP{{\mathbb P}}
\def\P{{\mathbb P}}
\def\bQ{{\mathbb Q}}
\def\Q{{\mathbb Q}}
\def\bR{{\mathbb R}}
\def\R{{\mathbb R}}
\def\bZ{{\mathbb Z}}
\def\Z{{\mathbb Z}}
\def\cH{{\mathscr H}}
\def\cD{{\mathscr D}}
\def\cE{{\mathscr E}}
\def\cO{{\mathscr O}}
\def\O{{\mathscr O}}
\def\cJ{{\mathscr J}}
\def\cK{{\mathscr K}}
\def\cM{{\mathscr M}}
\def\cU{{\mathscr U}}
\def\cR{{\mathscr R}}
\def\cS{{\mathscr S}}
\def\cX{{\mathscr X}}
\def\cY{{\mathscr Y}}
\def\cZ{{\mathscr Z}}

%
\def\ve{\varepsilon}
\def\vG{\varGamma}
\def\vg{\varGamma}
%
%
%
%
\def\lra{\longrightarrow}
\def\lla{\longleftarrow}
\def\Lra{\Longrightarrow}
\def\hra{\hookrightarrow}
\def\lmt{\longmapsto}
\def\ot{\otimes}
\def\op{\oplus}
\def\l{\lambda}
\def\Isoc{{\mathrm{Isoc}}}
\def\FIsoc{{F\text{-}\mathrm{Isoc}^\dag}}
\def\FMIC{{F\text{-}\mathrm{MIC}}}
\def\FilMIC{{\mathrm{Fil}\text{-}\mathrm{MIC}}}
\def\rigsyn{{\mathrm{rig}\text{-}\mathrm{syn}}}
\def\FilFMIC{{\mathrm{Fil}\text{-}F\text{-}\mathrm{MIC}}}
\def\Log{{\mathscr{L}{og}}}
\newcommand{\sPol}{\mathop{\mathscr{P}\!\mathit{ol}}\nolimits}
\def\uq{\underline{q}}
\def\wt#1{\widetilde{#1}}
\def\wh#1{\widehat{#1}}
\def\spt{\sptilde}
\def\ol#1{\overline{#1}}
\def\ul#1{\underline{#1}}
\def\us#1#2{\underset{#1}{#2}}
\def\os#1#2{\overset{#1}{#2}}

\begin{abstract}
We introduce a category of filtered $F$-isocrystals and construct a 
symbol map from Milnor $K$-theory to the group of 1-extensions of 
filtered $F$-isocrystals. We show that our symbol map is compatible with
the syntomic symbol map to the log syntomic cohomology by Kato and Tsuji.
These are fundamental materials in our computations of syntomic regulators
which we work in other papers.
\end{abstract}

\section{Introduction}
The purpose of this paper is to provide fundamental materials
for computing the syntomic regulators on Milnor $K$-theory, which is based
on the theory of $F$-isocrystals.

\medskip

Let $V$ be a complete discrete valuation ring 
such that the residue field $k$ is a perfect field of characteristic $p>0$ and
the fractional field $K$ is of characteristic $0$.
For a smooth affine scheme $S$ over $V$,
we introduce a {\it category of filtered $F$-isocrystals}, which is denoted by 
$\FilFMIC(S)$ (see \S \ref{filFMIC-sect} for the details).
Roughly speaking, an isocrystal is a crystalline sheaf
which corresponds to a smooth $\Q_l$-sheaf, and ``$F$"
means Frobenius action. 
We refer the book \cite{LS} for the general terminology of $F$-isocrystals.
As is well-known, it is equivalent to a notion of an integrable connection
with Frobenius action.
According to this, we shall define the category $\FilFMIC(S)$ without the terminology of 
$F$-isocrystals. Namely we define it to be a category of de Rham cohomology
endowed with Hodge filtration, integrable connection and Frobenius action,
so that the objects are described by familiar and elementary notation.
However, the theory of $F$-isocrystals plays an essential role in verifying several
functorial properties.
The purpose of this paper is to introduce a symbol map on the Milnor $K$-group
to the group of 1-extensions of filtered $F$-isocrystals. To be precise, 
let $S$ be a smooth affine scheme over $V$ and
$U\to S$ a smooth $V$-morphism having a good compactification, 
which means that $U\to S$ extends to 
a projective smooth morphism
$X\to S$ such that $X\setminus U$ is a relative simple normal crossing divisor 
({\it abbreviated to} NCD) over $S$.
Suppose that the comparison isomorphism
\[
\O(S)^\dag_K\ot_{\O(S)_K}H^i_\dR(U_K/S_K)\os{\cong}{\lra}H^i_\rig(U_k/S_k)
\]
holds for each $i\geq0$ where $U_K=U\times_V\Spec K$ and $U_k=U\times _V\Spec k$
etc.
Then, for an integer $n\geq 0$ such that $\Fil^{n+1}H^{n+1}_\dR(U_K/S_K)=0$
where
$\Fil^\bullet$ is the Hodge filtration, we construct
a homomorphism
\[
[-]_{U/S}:K^M_{n+1}(\O(U))\lra\Ext^1_{\FilFMIC(S)}(\O_S,H^n(U/S)(n+1))
\]
from the Milnor $K$-group of the affine ring $\O(U)$ to the group of $1$-extensions in 
the category of filtered $F$-isocrystals (Theorem \ref{mot-map-prop}).
We call this the {\it symbol map for $U/S$}.
We provide an explicit formula of our symbol map (Theorem \ref{comp-thm}).
Moreover we shall give the comparison of our symbol map with the syntomic 
symbol map 
\[
[-]_\syn:K^M_{n+1}(\O(U))\lra H^{n+1}_\syn(U,\Z_p(n+1))
\]
to the syntomic cohomology of Fontaine-Messing, or more generally
the log syntomic cohomology (cf. \cite[Chapter I \S3]{Ka2}, \cite[\S 2.2]{Ts1}).
See Theorem \ref{ext-logsyn-thm} for the details. 

\medskip

Thanks to the recent work by Nekovar-Niziol \cite{NN}, there are the syntomic regulator maps
\[
\reg_\syn^{i,j}:K_i(X)\ot\Q\lra H^{2j-i}_\syn(X,\Q_p(j))
\]
in very general setting,
which includes the syntomic symbol maps (up to torsion)
and the rigid syntomic regulator maps by Besser \cite{Be1}.
They play the central role in the Bloch-Kato conjecture \cite{BK}, and 
in the $p$-adic Beilinson conjecture by Perrin-Riou
\cite[4.2.2]{Perrin-Riou} (see also \cite[Conj. 2.7]{Colmez}).
On the other hand, the authors do not know how to construct
$\reg_\syn^{i,j}$ without ``$\ot\Q$".
We focus on the log syntomic cohomology with $\Z_p$-coefficients
since the integral structure is important in our ongoing applications (e.g. \cite{AC}), namely
a {\it deformation method} for computing syntomic regulators.

It is a notorious fact that, 
it is never easy to compute the syntomic regulator maps.
Indeed it is non-trivial
even for showing non-vanishing of $\reg_\syn^{i,j}$ in a general situation.
The deformation method is a method to employ differential equations,
which is motivated by Lauder \cite{Lauder}
who
provided the method for computing the Frobenius eigenvalues
of a smooth projective variety over a finite field.
The overview is as follows.
Suppose that a variety $X$ extends to a projective smooth family $f:Y\to S$ 
with $X=f^{-1}(a)$ 
and suppose that an element $\xi_X\in K_i(X)$ extends to an element $\xi\in K_i(Y)$.
We deduce a differential equation such that a ``function" 
$F(t)=\reg_\syn(\xi|_{f^{-1}(t)})$ is a solution.
Solve the differential equation.
Then we get the $\reg_\syn(\xi_X)$ by evaluating $F(t)$ at the point $a\in S$.
Of course, this method works only in a good situation, for example, it is powerless
if $f$ is a constant family. However once it works, it has a big advantage in explicit
computation of the syntomic regulators.
We demonstrate it by a particular example, namely an elliptic curve
with 3-torison points.
\begin{thm}[Corollary \ref{ell-cor}]
Let $p\geq 5$ be a prime. Let $W=W(\ol\F_p)$ be the Witt ring, and $K:=\Frac(W)$.
Let $a\in W$ satisfy $a\not\equiv 0,1$ mod $p$.
Let $E_a$ be the elliptic curve over $W$ defined by a Weierstrass equation
$y^2=x^3+(3x+4-4a)^2$.
Let
\[
\xi_a=\{h_1,h_2\}=\left\{
\frac{y-3x-4+4a}{-8(1-a)},
\frac{y+3x+4-4a}{8(1-a)}
\right\}\in K_2(E_a),
\]
where we note that the divisors $\mathrm{div}(h_i)$ have supports in $3$-torsion points.
Then there are overconvergent functions $\ve_1(t),\ve_2(t)\in K\ot W[t,(1-t)^{-1}]^\dag$
which are explicitly given as in Theorem \ref{ell-thm}
together with \eqref{ell-thm-eq3} and \eqref{ell-thm-eq4}, and we have
\[
\reg_\syn(\xi_a)=\ve_1(a)\frac{dx}{y}+\ve_2(a)\frac{xdx}{y}
\in H^2_\syn(E_a,\Q_p(2))\cong H^1_\dR(E_a/K).
\]
\end{thm}
We note that the function $\ve_i(t)$ are defined in terms of the hypergeometric series
\[
{}_2F_1\left({\frac13,\frac23\atop 1};t\right)=\sum_{n=0}^\infty\frac{(\frac13)_n}{n!}
\frac{(\frac23)_n}{n!}t^n,
\quad (\alpha)_n:=\alpha(\alpha+1)\cdots(\alpha+n-1).
\]

Concerning hypergeometric functions and regulators, 
the first author obtains more examples
in \cite{New}. There
he introduces certain convergent functions which
satisfy Dwork type congruence relations \cite{Dwork-p-cycle} to describe the syntomic regulators.
Besides, Chida and the first author discuss $K_2$ of elliptic curves
in more general situation, and obtain a number of numerical verifications of
the $p$-adic Beilinson conjecture.
In both works, our category $\FilFMIC(S)$ and symbol maps play as fundamental materials.

\medskip

Finally we comment on the category of syntomic coefficients by Bannai
\cite[1.8]{Bannai00}.
His category is close to our $\FilFMIC(S)$. The difference is that,
Bannai takes account into the boundary condition
at $\ol S\setminus S$ on the Hodge filtration, while we do not.
In this sense our category is less polish than his.
On the other hand, he did not work on the symbol maps or regulator maps.
Our main interest is the syntomic regulators, especially the deformation method.
For this ours is sufficient.


\subsection*{Notation.}
For a integral domain $V$ and a $V$-algebra $R$ (resp. $V$-scheme $X$),
let $R_K$ (resp. $X_K$) denote the tensoring $R\otimes_VK$ (resp. $X\times_VK$)
with the fractional field $K$.

Suppose that $V$ is a complete valuation ring $V$ endowed with a non-archimedian 
valuation $|\cdot|$.
For a $V$-algebra $B$ of finite type,
let $B^{\dag}$ denote the weak completion of $B$.
Namely if $B=V[T_1,\cdots,T_n]/I$, then 
$B^\dag=V[T_1,\cdots,T_n]^\dag/I$ where $V[T_1,\cdots,T_n]^\dag$
is the ring of power series $\sum a_\alpha T^\alpha$ such that
for some $r>1$,
$|a_\alpha|r^{|\alpha|}\to0$ as $|\alpha|\to\infty$. 
We simply write $B^\dag_K=K\ot _V B^{\dag}$.

\subsection*{Acknowledgements.}
The authors are grateful to Atsushi Shiho for kindly answering their questions
and telling us about his previous results.

\section{Filtered $F$-isocrystals and Milnor $K$-theory}
\label{Milnor-Ext-sect}
In this section, we work over a complete discrete valuation ring $V$
of characteristic $0$
such that the residue field $k$ is a perfect field of characteristic $p>0$.
We suppose that $V$ has a $p$-th Frobenius $F_V$, namely an endomorphism
on $V$ such that $F_V(x)\equiv x^p$ mod $pV$, and fix it throughout this section. 
Let $K=\Frac(V)$ be the fractional field.
The extension of $F_V$ to $K$ is also denoted by $F_V$.

A scheme means a separated scheme which is 
of finite type over $V$ unless otherwise specified.
If $X$ is a $V$-scheme (separated and of finite type), then $\widehat{X}_K$ will denote
Raynaud's generic fiber of the formal completion $\widehat{X}$,
$X_K^{\an}$ will denote the analytification of the $K$-scheme $X_K$,
and $j_X\colon \widehat{X}_K\hookrightarrow X_K^{\an}$ will denote the canonical immersion \cite[(0.3.5)]{Berthelot96}.

\subsection{The category of Filtered $F$-isocrystals}\label{filFMIC-sect}
Let $S=\Spec(B)$ be an affine smooth variety over $V$.
Let $\sigma\colon B^{\dag}\to B^{\dag}$ be a $p$-th Frobenius compatible with
$F_V$ on $V$, which means that $\sigma$ is $F_V$-linear and satisfies
$\sigma(x)\equiv
x^p$ mod $pB$.
The induced endomorphism $\sigma\otimes_{\bZ}\bQ\colon B_K^{\dag}\to B_K^{\dag}$ 
is also denoted by $\sigma$. 
We define the category $\FilFMIC(S,\sigma)$ (which we call the category of
{\it filtered $F$-isocrystals} on $S$) as follows.

\begin{defn}
    An object of $\FilFMIC(S,\sigma)$ 
is a datum $H=(H_{\dR}, H_{\rig}, c, \Phi, \nabla, \Fil^{\bullet})$, where
    \begin{itemize}
            \setlength{\itemsep}{0pt}
        \item $H_{\dR}$ is a coherent $B_K$-module,
        \item $H_{\rig}$ is a coherent $B^{\dag}_K$-module,
        \item $c\colon H_{\dR}\otimes_{B_K}B^{\dag}_K\xrightarrow{\,\,\cong\,\,} H_{\rig}$ is a $B^{\dag}_K$-linear isomorphism,
        \item $\Phi\colon \sigma^{\ast}H_{\rig}\xrightarrow{\,\,\cong\,\,} H_{\rig}$ is an isomorphism of $B^{\dag}_K$-algebra,
        \item $\nabla\colon H_{\dR}\to \Omega_{B_K}^1\otimes H_{\dR}$ is an (algebraic) integrable connection and
        \item $\Fil^{\bullet}$ is a finite descending filtration on $H_{\dR}$ of locally free $B_K$-module (i.e. each graded piece is locally free),
    \end{itemize}
    that satisfies $\nabla(\Fil^i)\subset \Omega^1_{B_K}\ot \Fil^{i-1}$ and
    the compatibility of $\Phi$ and $\nabla$ in the following sense.
    Note first that $\nabla$ induces an integrable connection $\nabla_{\rig}\colon H_{\rig}\to\Omega^1_{B_K^{\dag}}\otimes H_{\rig}$,
    where $\Omega^1_{B_K^{\dag}}$ denotes the sheaf of continuous differentials.
    In fact, firstly regard $H_{\dR}$ as a coherent $\sO_{S_K}$-module.
        Then, by (rigid) analytification, we get an integrable connection $\nabla^{\an}$
        on the coherent $\sO_{S_K^{\an}}$-module $(H_{\dR})^{\an}$.
        Then, apply the functor $j_S^{\dag}$ to $\nabla^{\an}$ to obtain
    an integrable connection on $\Gamma\left(S_K^{\an}, j_S^{\dag}\big( (H_{\dR})^{\an}\big)\right)=H_{\dR}\otimes_{B_K}B_K^{\dag}$.
    This gives an integrable connection $\nabla_{\rig}$ on $H_{\rig}$
    via the isomorphism $c$.
Then the compatibility of $\Phi$ and $\nabla$ means
that $\Phi$ is horizontal with respect to $\nabla_{\rig}$, namely $\Phi\nabla_\rig=\nabla_\rig\Phi$.
We usually write $\nabla_\rig=\nabla$ to simplify the notation.

    A morphism $H'\to H$ in $\FilFMIC(S,\sigma)$ 
    is a pair of homomorphisms $(h_{\dR}\colon H'_{\dR}\to H_{\dR}, h_{\rig}\colon H'_{\rig}\to H_{\rig})$,
    such that $h_{\rig}$ is compatible with $\Phi$'s, 
    $h_{\dR}$ is compatible with $\nabla$'s and $\Fil^{\bullet}$'s,
    and moreover they agree under the isomorphism $c$.
\end{defn}

\begin{rem}

(1) The category $\FilFMIC(S,\sigma)$ can also be described by using simpler categories as follows.
Let $\FilMIC(S_K)$ denote the category of 
filtered $S_K$-modules with integrable connection,
that is, the category of data $(M_\dR,\nabla,\Fil^{\bullet})$
with $M_{\dR}$ a coherent $B_K$-module, $\nabla$ an integrable connection on $M_{\dR}$, and
$\Fil^{\bullet}$ a finite descending filtration on $M_{\dR}$ of locally free $B_K$-module.
Let $\MIC(B^\dag_K)$ denote the category of 
coherent $B_K^{\dag}$-modules with integrable connections $(M_\rig,\nabla)$ on $B^{\dag}_K$, and
let $\FMIC(B^\dag_K,\sigma)$ denote the category of 
the coherent $B^{\dag}_K$-modules with integrable connections $(M_\rig,\nabla,\Phi)$ with 
$\sigma$-linear Frobenius endomorphisms.
Then $\FilFMIC(S,\sigma)$ is identified with the fiber product
\[
\FMIC(B_K^\dag,\sigma)\times_{\MIC(B_K^\dag)}\FilMIC(S_K).
\]

(2) Let $\FIsoc(B_k)$ denote the category of overconvergent $F$-isocrystals on $S_k$.
Then there is the equivalence of categories (\cite[Theorem 8.3.10]{LS})
\[
\FIsoc(B_k)\os{\cong}{\lra}\FMIC(B_K^\dag,\sigma).
\]
Therefore, by combining with the description in (1), 
we see that our category $\FilFMIC(S,\sigma)$ does not depend on $\sigma$,
which means that there is the natural equivalence
$\FilFMIC(S,\sigma)\cong \FilFMIC(S,\sigma')$ (see also Lemma \ref{EK-lem} in Appendix).
By virtue of this fact, we often drop ``$\sigma$'' in the notation.
\end{rem}

\medskip

For two objects $H$, $H'$ in the category $\FilFMIC(S,\sigma)$, we have a direct sum $H\oplus H'$ and the tensor product $H\otimes H'$ in a customary manner.
The unit object for the tensor product, denoted by $B$ or $\mathscr{O}_S$, is $(B_K, B_K^{\dag}, c, \sigma_B, d, \Fil^{\bullet})$,
where $c$ is the natural isomorphism, $d$ is the usual differential and $\Fil^{\bullet}$ is defined by $\Fil^0B_K=B_K$ and $\Fil^1B_K=0$.
The category $\FilFMIC(S,\sigma)$ forms a tensor category with this tensor product and the unit object $B$ or $\O_S$.

The unit object can also be described as $B=B(0)$ or $\O_S=\O_S(0)$ by using the following notion of \textit{Tate object}.

\begin{defn}
    Let $n$ be an integer.

(1) The \textit{Tate object} in $\FilFMIC(S)$, which
we denote by $B(n)$ or $\sO_{S}(n)$, is defined 
to be $(B_K, B_K^{\dag}, c, p^{-n}\sigma_B, d, \Fil^{\bullet})$, where
$c$ is the natural isomorphism, $d:B_K\to\Omega^1_{B_K}$ is the usual differential,
 and $\Fil^{\bullet}$ is defined by
$\Fil^{-n}B_K=B_K$ and $\Fil^{-n+1}B_K=0$.

(2) For an object $H$ of $\Fil$-$F$-$\MIC(S)$, we write $H(n)\defeq H\otimes B(n)$.
\end{defn}

\medskip

Now, we discuss the Yoneda extension groups in the category $\FilFMIC(S)$.
A sequence
\[
 H_1\lra H_2\lra H_3
\]
in $\FilFMIC(S)$ (or in $\FilMIC(S_K)$) is called {\it exact} if
\[
 \Fil^i  H_{1,\dR}\lra \Fil^i H_{2,\dR}\lra \Fil^i H_{3,\dR}
\]
are exact for all $i$.
Then the category $\FilFMIC(S)$ forms an exact category which has
kernel and cokernel objects for any morphisms.
Thus, the Yoneda extension groups
\[
\Ext^\bullet(H,H')=\Ext^\bullet_{\FilFMIC(S)}(H,H')
\]
in $\FilFMIC(S)$
are defined in the canonical way (or one can further define the derived category of
$\FilFMIC(S)$ \cite[1.1]{BBDG}).
An element of $\Ext^j(H,H')$ is represented by an exact sequence
\[
0\lra H'\lra M_1\lra\cdots\lra M_j\lra H\lra 0
\]
and is subject to the equivalence relation generated by commutative diagrams
\[
\xymatrix{
0\ar[r]&H' \ar@{=}[d]\ar[r]&M_1\ar[d]\ar[r]&\cdots \ar[r]&M_j\ar[d]\ar[r]
&H\ar[r]\ar@{=}[d]&0\\
0\ar[r]&H' \ar[r]&N_1\ar[r]&\cdots \ar[r]&N_j\ar[r]&H\ar[r]&0.
}
\]
Note that $\Ext^\bullet(H,H')$ is uniquely divisible (i.e. a $\Q$-module) as 
the multiplication by $m\in\Z_{>0}$ on $H$ or $H'$ is bijective.

\medskip

Nextly, we discuss the functoriality of the category $\FilFMIC(S)$ with respect to $S$.
Let $S'=\Spec(B')$ be another affine smooth variety
 and
$i\colon S'\to S$ a morphism of $V$-schemes.
Then, $i$ induces a pull-back functor
\begin{equation}\label{pull-back-functor}
    i^{\ast}\colon \Fil\hyphen F\hyphen\MIC(S)\to \Fil\hyphen F\hyphen\MIC(S')
\end{equation}
in a natural way.
In fact, if $H=(H_{\dR}, H_{\rig}, c, \Phi, \nabla, \Fil^{\bullet})$ is an object
of $\Fil$-$F$-$\MIC(S)$, then we define
$i^{\ast}H=(H_{\dR}\otimes_{B_K}B'_K, H_{\rig}\otimes_{B^{\dag}_K}B'^{\dag}_K, c\otimes_{B^{\dag}_K}B'^{\dag}_K,
\Phi', \nabla', \Fil'^{\bullet})$,
where $\nabla'$ and $\Fil'^{\bullet}$ are natural pull-backs of $\nabla$ and $\Fil^{\bullet}$ respectively,
and where $\Phi'$ is the natural Frobenius structure obtained as follows.
We may regard $(H_{\rig}, \nabla_{\rig}, \Phi)$ as an overconvergent $F$-isocrystal on $S$
via the equivalence $F\hyphen\MIC(B_K^{\dag})\simeq F\hyphen\mathrm{Isoc}^{\dag}(B_k)$ \cite[Theorem 8.3.10]{LS}.
Then, its pull-back along $i_k\colon S'_k\to S_k$
is an overconvergent $F$-isocrystal on $S'_k$,
which is again identified with an object of $F\hyphen\MIC(B'^{\dag}_K)$.
Thus it is of the form $(H'_{\rig}, \nabla'_{\rig}, \Phi')$,
and $H'_\rig$ is naturally isomorphic to $H_{\rig}\otimes_{B^{\dag}_K}B'^{\dag}_K$ (\cite[Prop 8.1.15]{LS}).
Now, $\Phi'$ gives the desired Frobenius structure. 

\medskip

\subsection{The complex $\cS(M)$}
In this subsection, we introduce a complex $\cS(M)$ for each object $M$ in $\FilFMIC(S,\sigma)$ which is,
in the case where $M=\O_S(r)$, close to the syntomic complex $\cS_n(R)_{S,\sigma}$ of Fontaine--Messing.

Before the definition, we prepare a morphism attached to each object of $\FilMIC(S_K)$.
Let $H=(H_\dR,\nabla,\Fil^\bullet)$ be an object of $\FilMIC(S_K)$.
Let $\Omega^\bullet_{B_K}\ot\Fil^{i-\bullet}H_\dR$ denote the de Rham complex
\[
\Fil^iH_\dR\to\Omega^1_{B_K}\ot \Fil^{i-1}H_\dR
\to\cdots \to\Omega^n_{B_K}\ot \Fil^{i-n}H_\dR\to
\cdots,
\]
where $\ot$ denotes $\ot_{B_K}$ and the differentials are given by
\begin{equation}\label{filFMIC-eq1}
\omega\ot x\longmapsto d\omega^j\ot x+(-1)^j\omega\wedge\nabla(x),
\quad (\omega\ot x\in \Omega^j_{B_K}\ot H_\dR).
\end{equation}
Now, we define a natural map
\begin{equation}\label{filFMIC-eq2}
\Ext^i_{\FilMIC(S_K)}(\O_{S_K},H)\lra H^i(\Omega^\bullet_{B_K}\ot\Fil^{-\bullet}H_\dR)
\end{equation}
in the following way.
Let
\begin{equation}\label{filMIC-element}
0\lra H\lra M_i\lra M_{i-1}\lra\cdots\lra M_0\lra \O_{S_K}\lra0
\end{equation}
be an exact sequence in $\FilMIC(S_K)$.
This induces an exact sequence
\[
0\lra \Omega^\bullet_{B_K}\ot\Fil^{-\bullet}H\lra 
\Omega^\bullet_{B_K}\ot\Fil^{-\bullet}M_i\lra \cdots\lra 
\Omega^\bullet_{B_K}\ot\Fil^{-\bullet}M_0\lra \Omega^\bullet_{B_K}\lra0
\]
of complexes, and hence a connecting homomorphism
$\delta\colon H^0(\Omega^\bullet_{B_K})\to H^i(\Omega^\bullet_{B_K}\ot\Fil^{-\bullet}H)$.
Then the map \eqref{filFMIC-eq2} is defined by sending the sequence \eqref{filMIC-element} to $\delta(1)$.

By the forgetful functor $\FilFMIC(S,\sigma)\to\FilMIC(S)$,
the morphism \eqref{filFMIC-eq2} clearly induces
a canonical morphism
\begin{equation}\label{filFMIC-eq7}
\Ext^i_{\FilFMIC(S,\sigma)}(\O_{S},H)\lra
H^i(\Omega^\bullet_{B_K}\ot\Fil^{-\bullet}H_\dR).
\end{equation}

Let $F\text{-mod}=\FMIC(\Spec K)$, namely the category of $K$-modules endowed with
$F_V$-linear endomorphisms.
Then, we have a functor
\[
\FMIC(B_K^\dag,\sigma)\lra D^b(F\text{-mod}),\quad
(M_\rig,\Phi)\longmapsto (\Omega^\bullet_{B^\dag_K}\ot M_\rig,\Phi)
\]
to the derived category of complexes in $F\text{-mod}$,
where $\Phi$ in the target is defined to be $\sigma\ot\Phi$
(we use the same notation because we always extend the Frobenius action on
the de Rham complex $\Omega^\bullet_{B^\dag_K}\ot M_\rig$
by this rule).
We note that the above functor does not depend on $\sigma$.
Indeed, the composition
\[
 F\hyphen\mathrm{Isoc}^{\dag}(B_k)\os{\cong}{\lra}\FMIC(B_K^\dag,\sigma)
 \lra D^b(F\text{-mod})
\]
is the functor $(E,\Phi)\mapsto (R\vg_\rig(S_k,E),\Phi_\rig)$
where $\Phi_\rig$ denotes the Frobenius action on the rigid cohomology (\cite[Proposition 8.3.12]{LS}),
and this does not depend on $\sigma$.

\begin{defn}
For an object $M\in \FilFMIC(S,\sigma)$, we define $\cS(M)$ to be
the mapping fiber of the morphism
\[
1-\Phi:\Omega^\bullet_{B_K}\ot \Fil^{-\bullet}M_\dR\lra 
\Omega^\bullet_{B^\dag_K}\ot M_\rig,
\]
where $1$ denotes the inclusion $
\Omega^\bullet_{B_K}\ot \Fil^{-\bullet}M_\dR\hra
\Omega^\bullet_{B^\dag_K}\ot M_\rig$ via the comparison and $\Phi$ is the composition
of it with $\Phi$ on $\Omega^\bullet_{B^\dag_K}\ot M_\rig$.
\end{defn}

Note that, in more down-to-earth manner, each term of $\cS(M)$ is given by
\begin{equation}\label{filFMIC-eq4}
\cS(M)^i=
\Omega^i_{B_K}\ot \Fil^{-i}M_\dR\op
\Omega^{i-1}_{B^\dag_K}\ot M_\rig
\end{equation}
and the differential $\cS(M)^i\to\cS(M)^{i+1}$ is given by
\[
(\omega,\xi)\longmapsto(d_M\omega, (1-\Phi)\omega-d_M\xi),\quad
(\omega\in \Omega^i_{B_K}\ot \Fil^{-i}M_\dR,\,\xi\in
\Omega^{i-1}_{B^\dag_K}\ot M_\rig)
\]
where $d_M$ is the differential \eqref{filFMIC-eq1}.

An exact sequence
\[
0\lra H\lra M_i\lra M_{i-1}\lra\cdots\lra M_0\lra \O_{S}\lra0
\]
in $\FilFMIC(S,\sigma)$ gives rise to an exact sequence
\[
0\lra \cS(H)\lra \cS(M_i)\lra \cS(M_{i-1})\lra\cdots\lra \cS(M_0)\lra \cS(\O_{S})\lra0
\]
of complexes of $\Q_p$-modules.
Let
$\delta\colon \Q_p\cong H^0(\cS(\O_S))\to H^i(\cS(H))$
be the connecting homomorphism.
We define a homomorphism
\begin{equation}\label{filFMIC-eq6}
u\colon \Ext^i_{\FilFMIC(S,\sigma)}(\O_{S},H)\lra
H^i(\cS(H))
\end{equation}
by associating $\delta(1)$ to the above extension.
The composition of $u$ with the natural map
$H^i(\cS(H))\to H^i(\Omega^\bullet_{B_K}\ot\Fil^{-\bullet}H_\dR)$
agrees with \eqref{filFMIC-eq7}.

The complex $\cS(\O_S(r))$ is close to the syntomic complex of Fontaine-Messing.
More precisely, let $S_n:=S\times_W\Spec W/p^nW$ and $B_n:=B/p^nB$.
The syntomic complex $\cS_n(r)_{S,\sigma}$ is the mapping fiber of the morphism
\[
1-p^{-r}\sigma\colon \Omega_{B_n}^{\bullet\geq r}\lra \Omega^\bullet_{B_n}
\]
of complexes where we note that $p^{-r}\sigma$ is well-defined
(see \cite[p.410--411]{Ka1}).
Hence there is the natural map
\begin{equation}\label{filFMIC-eq8}
H^i(\cS(\O_S(r)))\lra \Q\ot\varprojlim_n H^i_\zar(S_1,\cS_n(r)_{S,\sigma})=:H^i_\syn(S,\Q_p(r)).
\end{equation}
Let
\begin{equation}\label{filFMIC-eq9}
\Ext^i_{\FilFMIC(S,\sigma)}(\O_{S},\O_S(r))\lra
H^i_\syn(S,\Q_p(r))
\end{equation}
be the composition morphism.
Apparently, the both side of \eqref{filFMIC-eq9} depend on $\sigma$.
However, if we replace $\sigma$ with $\sigma'$, there is the natural transformation
between
$\FilFMIC(S,\sigma)$ and $\FilFMIC(S,\sigma')$ thanks to the theory of $F$-isocrystals,
and there is also a natural transformation on the syntomic cohomology.
The map \eqref{filFMIC-eq9} is compatible under these transformations.
In this sense, \eqref{filFMIC-eq9} does not depend on the Frobenius $\sigma$.

\begin{lem}\label{filFMIC-lem}
Suppose $\Fil^0H_\dR=0$.
Then the map $u$ in \eqref{filFMIC-eq6} is injective when $i=1$.
Moreover, the map \eqref{filFMIC-eq8} is injective when $i=1$ and $r\geq 0$.
\end{lem}
\begin{proof}
Let
\[
0\lra H\lra M\lra\O_S\lra0
\]
be an exact sequence in $\FilFMIC(S,\sigma)$.
Since $\Fil^0H_\dR=0$, there is the unique lifting $e\in \Fil^0M_\rig$.
Then
\begin{equation}\label{filFMIC-lem-eq1}
u(M)=(\nabla(e),(1-\Phi)e)\in 
H^1(\cS(H))\subset \Omega^1_{B^\dag_K}\ot \Fil^{-1}H_\rig\op H_\rig
\end{equation}
by definition of $u$.
If $u(M)=0$, then the datum $(B_Ke,B^\dag_Ke, c,\Phi,\nabla,\Fil^\bullet)$ forms
a subobject of $M$ which is isomorphic to
the unit object $\O_S$. This gives a splitting of the above exact sequence.
The latter follows from the fact that
\[
H^1_\syn(S,\Q_p(1))\subset \Q\ot\varprojlim_n(\Omega^1_{B_n}\op B_n),\quad
H^1_\syn(S,\Q_p(r))\subset \Q\ot\varprojlim_n(\{0\}\op B_n),\,(r\geq2)
\]
and 
\[
H^1(\cS(\O_S(1)))\subset \Omega^1_{B^\dag_K}\op B^\dag_K,\quad
H^1(\cS(\O_S(r)))\subset \{0\}\op B^\dag_K,\,(r\geq2).
\]
\end{proof}

\subsection{Log objects}\label{log-obj-sect}
In this subsection, we introduce the ``log object'' in $\FilFMIC(S)$
concerning a $p$-adic logarithmic function.
In the next subsection, it will be generalized to a notion of ``polylog object''.

For $f\in B^\times$ let
\[
\log^{(\sigma)}(f):=p^{-1}\log\left(\frac{f^p}{f^\sigma}\right)
=-p^{-1}\sum_{n=1}^\infty\frac{(1-f^p/f^\sigma)^n}{n}\in B^\dag.
\]
An elementary computation yields $\log^{(\sigma)}(f)+\log^{(\sigma)}(g)=\log^{(\sigma)}(fg)$
for $f,g\in B^\times$.
\begin{defn}\label{log-obj-def}
For $f\in B^\times$, we define the {\it log object} $\sLog(f)$ in $\FilFMIC(S,\sigma)$ 
as follows.
\begin{itemize}
    \setlength{\itemsep}{0pt}
    \item $\sLog(f)_{\dR}$ is a free $B_K$-module of rank two; $\sLog(f)_{\dR}=B_Ke_{-2}\oplus B_Ke_0$.
    \item $\sLog(f)_{\rig}=
B_K^{\dag}e_{-2}\oplus B_K^{\dag}e_0$.
    \item $c$ is the natural isomorphism.
    \item $\Phi$ is the $\sigma$-linear morphism defined by
        \[
            \Phi(e_{-2})=p^{-1}e_{-2},\quad \Phi(e_0)=
            e_0-\log^{(\sigma)}(f)e_{-2}.
            \]
    \item $\nabla$ is the connection defined by $\nabla(e_{-2})=0$ and $\nabla(e_0)=\frac{df}{f}e_{-2}$.
    \item $\Fil^{\bullet}$ is defined by
        \[
            \Fil^{i}\sLog(f)_{\dR}=\begin{cases} \sLog(f)_{\dR} & \text{ if } i\leq -1,\\
                B_Ke_0 & \text{ if } i= 0,\\
               0 & \text{ if } i> 0.\end{cases}
        \]
\end{itemize}
\end{defn}
This is fit into the exact sequence
\begin{equation}\label{log-obj-gp}
    0\lra \O_S(1)\os{\epsilon}{\lra} \sLog(f)\os{\pi}{\lra} \O_S\lra 0
\end{equation}
in $\FilFMIC(S)$ where the two arrows are defined by $\epsilon(1)=e_{-2}$ and $\pi(e_0)=1$.
This defines a class in $\Ext^1_{\FilFMIC(S,\sigma)}(\O_S,\O_S(1))$,
which we write by $[f]_S$.
It is easy to see that $f\mapsto [f]_S$ is additive.
We call the group homomorphism
\begin{equation}\label{symbol-map}
[-]_S:\O(S)^\times\lra \Ext^1_{\FilFMIC(S,\sigma)}(\O_S,\O_S(1))
\end{equation}
the symbol map.
\begin{lem}\label{log-obj-lem1}
The composition
\[
\O(S)^\times\os{[-]_S}{\lra}
\Ext^1_{\FilFMIC(S,\sigma)}(\O_{S},\O_{S}(1)) 
\lra
H^1_\syn(S,\Q_p(1))
\]
agrees with the symbol map by Kato \cite{Ka1} where the second arrow
is the map \eqref{filFMIC-eq9}.
Namely, it is explicitly described as follows,
\[
f\longmapsto
\left(\frac{df}{f},\log^{(\sigma)}(f)\right)\in H^1_\syn(S,\Q_p(1))\subset
\Omega^1_{\wh B_K}\op \wh B_K
\]
where $\wh B:=\varprojlim_n B/p^nB$ and $\wh B_K:=\Q\ot\wh B$.
\end{lem}
\begin{proof}
By definition of $\rho$ in \eqref{filFMIC-eq6}, one has
\[
\rho(\sLog(f))=(\nabla(e_0),(1-\Phi)e_0)=
\left(\frac{df}{f},\log^{(\sigma)}(f)\right)
\]
as desired.
\end{proof}
\begin{lem}\label{log-obj-lem2}
The symbol map \eqref{symbol-map} is functorial with respect to the pull-back.
Namely, for a morphism $i\colon S'\to S$, the diagram
\[
\xymatrix{
\O(S)^\times\ar[r]^-{[-]_S}\ar[d]_i&
\Ext^1_{\FilFMIC(S,\sigma)}(\O_{S},\O_{S}(1)) \ar[d]^{i^*}\\
\O(S')^\times\ar[r]^-{[-]_{S'}}&
\Ext^1_{\FilFMIC(S',\sigma)}(\O_{S'},\O_{S'}(1))
}
\]
is commutative where $i^*$ denotes the map induced from the pull-back functor
\eqref{pull-back-functor}.
\end{lem}
\begin{proof}
Because of the injectivity of \eqref{filFMIC-eq9} (Lemma \ref{filFMIC-lem}),
the assertion can be reduced to the compatibility of Kato's symbol maps
by Lemma \ref{log-obj-lem1}. 
\end{proof}


\subsection{Polylog objects}\label{poly-obj-sect}
In this subsection, we generalize the log object to polylog objects.
To define the polylog objects, we need the $p$-adic polylog function.

For an integer $r$, we denote
the {\it $p$-adic polylog function} by
\begin{equation}\label{main-sect-eq1}
\ln^{(p)}_r(z):=\sum_{n\geq 1,\,p\not{\hspace{2pt}|}\,n}\frac{z^n}{n^r}
=
\lim_{s\to\infty}\frac{1}{1-z^{p^s}}
\sum_{1\leq n<p^s,\,p\not{\hspace{2pt}|}\,n}\frac{z^n}{n^r}
\in 
\Z_p\left[ z,\frac{1}{1-z}\right]^\wedge
\end{equation}
where $A^\wedge$ denotes the $p$-adic completion of
a ring $A$.
As is easily seen, one has 
\[
\ln^{(p)}_r(z)=(-1)^{r+1}\ln^{(p)}_r(z^{-1}),\quad 
z\frac{d}{dz}\ln_{r+1}^{(p)}(z)=\mathrm{ln}_r^{(p)}(z).
\]
If $r\leq 0$, this is a rational function. Indeed
\[
\ln^{(p)}_0(z)=\frac{1}{1-z}-\frac{1}{1-z^p},\quad 
\ln^{(p)}_{-r}(z)=\left(z\frac{d}{dz}\right)^r\ln^{(p)}_0(z).
\]
If $r\geq 1$, it is no longer a rational function but an overconvergent function.
\begin{prop}\label{polylog-oc}
Let $r\geq 1$. Put $x:=(1-z)^{-1}$. Then $\ln^{(p)}_r(z)\in (x-x^2)\Z_p[x]^\dag$.
\end{prop}
\begin{pf}
Since $\ln^{(p)}_r(z)$ has $\Z_p$-coefficients, it is enough to check
$\ln^{(p)}_r(z)\in (x-x^2)\Q_p[x]^\dag$.
We first note
\begin{equation}\label{lnk-eq1}
(x^2-x)\frac{d}{dx}\mathrm{ln}_{k+1}^{(p)}(z)=\mathrm{ln}_k^{(p)}(z).
\end{equation}
The limit in \eqref{main-sect-eq1} can be rewritten as
\[
\lim_{s\to\infty}\frac{1}{x^{p^s}-(x-1)^{p^s}}
\sum_{1\leq n<p^s,\,p\not{\hspace{2pt}|}\,n}\frac{x^{p^s}(1-x^{-1})^n}{n^r}.
\]
This shows that $\ln^{(p)}_r(z)$ is divided by $x-x^2$.
Let $w(x)\in \Z_p[x]$ be defined by
\[
\frac{1-z^p}{(1-z)^p}=x^p-(x-1)^p=1-pw(x).
\]
Then
\[
\ln_1^{(p)}(z)=p^{-1}\log\left(\frac{1-z^p}{(1-z)^p}\right)
    =-p^{-1}\sum_{n=1}^\infty\frac{p^nw(x)^n}{n}\in \Z_p[x]^\dag.
\]
This shows $\ln^{(p)}_1(z)\in(x-x^2)\Q_p[x]^\dag$,
as required in case $r=1$. 
Let
\[
-(x-x^2)^{-1}\mathrm{ln}_1^{(p)}(z)
    =a_0+a_1x+a_2x^2+\cdots+a_nx^n+\cdots\in\Z_p[x]^\dag.
\]
By \eqref{lnk-eq1} one has
\[
\ln_2^{(p)}(z)=c+x+\frac{a_1}{2}x^2+\cdots
    +\frac{a_n}{n+1}x^{n+1}+\cdots\in \Q_p[x]^\dag
\]
and hence $\ln_2^{(p)}(z)\in(x-x^2)\Q_p[x]^\dag$,
as required in case $r=2$.
Continuing the same argument,
    we obtain $\ln_r^{(p)}(z)\in (x-x^2)\Q_p[x]^\dag$ for every $r$.
\end{pf}
\begin{rem}
The proof shows that $\ln_r^{(p)}(z)$ converges on an open disk $|x|<|1-\zeta_p|$.  
\end{rem}

\begin{defn}\label{poly-obj-def}
Let $C=V[T, T^{-1}, (1-T)^{-1}]$, and $\sigma_C$ a $p$-th Frobenius such that $\sigma_C(T)=T^p$.
Let $n\geq 1$ be an integer.
    We define the $n$-th polylog object $\sPol_n(T)$ of $\Fil\hyphen F\hyphen\MIC(\Spec C)$ as follows.
    \begin{itemize}
        \item $\sPol_n(T)_{\dR}$ is a free $C_K$-module of rank $n+1$;
            $\sPol_n(T)_\dR=\bigoplus_{j=0}^n C_K\, e_{-2j}.$
                 \item 
                 $\sPol_n(T)_{\rig}\defeq
                 \sPol_n(T)_{\dR}\otimes_{C_K}C_K^{\dag}$.
        \item $c$ is the natural isomorphism.
        \item $\Phi$ is the $C_K^{\dag}$-linear morphism defined by
            \[
                \Phi(e_0)=e_0-\sum_{j=1}^n(-1)^j\ln_j^{(p)}(T)e_{-2j},\quad
                \Phi(e_{-2j})=p^{-j}e_0,\quad(j\geq1).
                    \]
    \item $\nabla$ is the connection defined by 
    \[\nabla(e_0)=\frac{\rd T}{T-1}e_{-2},\quad
    \nabla(e_{-2j})=\frac{\rd T}{T}e_{-2j-2},\quad(j\geq1)
    \]
    where $e_{-2n-2}:=0$.
    \item $\Fil^{\bullet}$ is defined by 
 $\Fil^m\sPol_n(T)_{\dR}=\bigoplus_{0\leq j\leq -m} C_K\, e_{-2j}.$
 In particular, $\Fil^m\sPol_n(T)_{\dR}=\sPol_n(T)_{\dR}$ if $m\leq -n$ and $=0$ if $m\geq 1$.
    \end{itemize}
When $n=2$, we also write $\sDilog(T)=\sPol_2(T)$, and call it the {\it dilog object}.
\end{defn}
For a general $S=\Spec(B)$ and $f\in B$ satisfying $f,1-f\in B^\times$, 
we define
the polylog object $\sPol_n(f)$ to be the pull-back $u^*\sPol_n(T)$ where $u\colon \Spec(B)\to 
\Spec V[T,T^{-1},(1-T)^{-1}]$ is given by $u(T)=f$.
When $n=1$, $\sPol_1(T)$ coincides with 
the log object $\sLog(1-T)$ in $\FilFMIC(C)$.

\subsection{Relative cohomologies.}
Let $S=\Spec(B)$ be a smooth affine $V$-scheme and let $\sigma$ be a $p$-th Frobenius on $B^{\dag}$.
In this subsection, we discuss objects in $\Fil$-$F$-$\MIC(S)$
arising as a relative cohomology of smooth morphisms.

Let $u\colon U\to S$ be a quasi-projective smooth morphism.
We firstly describe a datum which we discuss in this subsection.

\begin{defn}
    \label{defn:relativecoh}
    We define a datum
\[
    H^i(U/S)=(H^i_{\dR}(U_K/S_K), H^i_{\rig}(U_k/S_k), c, \nabla, \Phi, \Fil^{\bullet})
\]
as follows.

\begin{itemize}
    \setlength{\itemsep}{0pt}
    \item  $H^i_{\dR}(U_K/S_K)$ is the $i$-th relative algebraic de Rham cohomology of $u_K$,
    namely the module of global sections of the $i$-th cohomology sheaf $R^i(u_K)_{\ast}\Omega^{\bullet}_{U_K/S_K}$,
        and $\nabla$ \resp{$\Fil^{\bullet}$} is the Gauss--Manin connection \resp{the Hodge filtration} on $H^i_{\dR}(U_K/S_K)$.
\item 
    $(H^i_{\rig}(U_k/S_k), \Phi)$ is the $B_K^{\dag}$-module with $\sigma$-linear Frobenius structure
    obtained as the $i$-th relative rigid cohomology of $u_k$.
    In particular,
    \[
        H^i_{\rig}(U_k/S_k) = \Gamma(S_K^{\an}, R^iu_{\rig}j^{\dag}_U\sO_{U_K^{\an}}),
    \]
    where $R^iu_{\rig}j^{\dag}_U\sO_{U_K^{\an}}:=R^i(u_K^{\an})_{\ast}j^{\dag}_U\Omega^{\bullet}_{U_K^{\an}/S_K^{\an}}$
    is the $i$-th relative rigid cohomology sheaf.
    (We justify this notation and definition in Remark \ref{rem:defofrigidcoh} below.)
    \item $c\colon H^i_{\dR}(U_K/S_K)\otimes_{B_K}B_K^{\dag}\to H^i_{\rig}(U_k/S_k)$ is the natural morphism between the algebraic de Rham cohomology and the rigid cohomology.
\end{itemize}

    Let us give a construction of the comparison morphism $c$ in this datum.
    We basically follow the construction in \cite[5.8.2]{Solomon}.
    Let $\iota\colon S_K^{\an}\to S_K$ be the natural morphism of ringed topoi \cite[0.3]{Berthelot96}.
    Then, by adjunction, we have a natural morphism
    $R^iu_{K,\ast}\Omega^{\bullet}_{U_K/S_K}\to \iota_{\ast}\big(R^iu_{K,\ast}\Omega_{U_K/S_K}^{\bullet}\big)^{\an}$.
    Now, together with the natural morphisms
    \[
        \big(R^iu_{K,\ast}\Omega^{\bullet}_{U_K/S_K}\big)^{\an}
    \to R^i(u_K^{\an})_{\ast}\Omega^{\bullet}_{U_K^{\an}/S_K^{\an}}
    \to R^i(u^{\an}_K)_{\ast}j^{\dag}_{U}\Omega^{\bullet}_{U_K^{\an}/S_K^{\an}},
    \]
    we get a morphism $R^iu_{K,\ast}\Omega^{\bullet}_{U_K/S_K}\to \iota_{\ast}R^i(u^{\an}_K)_{\ast}j^{\dag}_{U}\Omega^{\bullet}_{U_K^{\an}/S_K^{\an}}$.
    By taking the module of global sections,
    we get a morphism $H^i_{\dR}(U_K/S_K)\to H^i_{\rig}(U_k/S_k)$,
    and therefore the desired morphism $c\colon H^i_{\dR}(U_K/S_K)\otimes_{B_K}B_K^{\dag}\to H^i_{\rig}(U_k/S_k)$
    because $H^i_{\rig}(U_k/S_k)$ is a $B_K^{\dag}$-module.
\end{defn}

\begin{rem}
    \label{rem:defofrigidcoh}
    (1)
    Let us justify our description of the rigid cohomology $H^i_{\rig}(U_k/S_k)$
    in Definition \ref{defn:relativecoh}.

    We begin by recalling a usual definition of the rigid cohomology of $u_k\colon U_k\to S_k$
    over a frame $(S_k, \overline{S}_k, \widehat{\overline{S}})$,
    where $\overline{S}$ is a closure of $S$ in a projective space over $V$
    and $\widehat{\overline{S}}$ is its completion.
    Let $\overline{X}$ be the closure of $U$ in a projective space over $\overline{S}$
    and let $\overline{f}\colon \overline{X}\to\overline{S}$ be the extension of $u$.
    (We choose this notation because, in this article, $X$ usually denotes $\overline{f}^{-1}(S)$ and $f$ denotes $\overline{f}|_X$.)
    Then, by definition, the $i$-th relative rigid cohomology sheaf is 
    $R^i(\overline{f}_K^{\an})_{\ast}\overline{j}^{\dag}_{U}\Omega^{\bullet}_{\overline{X}_K^{\an}/\overline{S}_K^{\an}}$,
    where $\overline{j}_U\colon \widehat{U}_K\hookrightarrow \overline{X}_K^{\an}$ is the natural inclusion,
    and $H^i_{\rig}(U_k/S_k)$ is the module of global sections of this sheaf on $\overline{S}_K$.

    However, by using the fact that $u\colon U\to S$
    is a lift of $u_k\colon U_k\to S_k$ to a smooth morphism of smooth \emph{algebraic} $V$-schemes,
    we may obtain the same module without referring to compactifications.
    In fact, firstly, since $S_K^{\an}$ is a strict neighborhood of $\widehat{S}_K$ \cite[(1.2.4)(ii)]{Berthelot96},
    $H^i_{\rig}(U_k/S_k)$ is also the module of global sections of $R^i(\overline{f}_K^{\an})_{\ast}j^{\dag}_{U}\Omega^{\bullet}_{\overline{X}_K^{\an}/\overline{S}_K^{\an}}$ on $S_K^{\an}$
    by overconvergence.
    Moreover, the restriction of this sheaf on $S_K^{\an}$ is isomorphic to
    $R^i(u_K^{\an})_{\ast}j^{\dag}_{U}\Omega^{\bullet}_{U_K^{\an}/S_K^{\an}}$
    by \cite[6.2.2]{LS}
    because again $U_K^{\an}$ is a strict neighborhood of $\widehat{U}_K$.
    Our definition of $R^iu_{\rig}j^{\dag}_U\sO_{U_K^{\an}}$
    and the description of $H^i_{\rig}(U_k/S_k)$ in Definition \ref{defn:relativecoh} are thus justified.

    (2) 
    We will use the datum $H^i(U/S)$ only in the case where
    the rigid cohomology sheaf $R^iu_{\rig}j^{\dag}_U\sO_{U_K^{\an}}$ is known to be a coherent $j^{\dag}_{S}\sO_{S_K^{\an}}$-module
    for all $i\geq 0$.
    In this case, we also have
    \[
        H^i_{\rig}(U_k/S_k)=\Gamma(S_K^{\an},R^iu_{\rig}j^{\dag}_U\sO_{U_K^{\an}})
    \]
    by the vanishing of higher sheaf cohomologies for coherent $j^{\dag}_S\sO_{S_K^{\an}}$-modules.
    (This vanishing is perhaps well-known, but we included it as Lemma \ref{lem:vanishing} at the end of this subsection
    because we could not find an appropriate reference.)
\end{rem}

\bigskip

Now, we have defined the datum $H^i(U/S)$.
This, however, does not immediately mean that it is an object of $\FilFMIC(S)$.
For this datum to be an object in $\FilFMIC(S)$,
we need the $i$-th relative cohomology
$H^i_{\rig}(U_k/S_k)$ to be a coherent $B_K^{\dag}$-module and 
we need the morphism $c$ to be an isomorphism.
In the rest of this subsection, we discuss two settings
under which these conditions hold.
Briefly said, these two settings are: the case of proper smooth morphisms (Setting \ref{set:projsm})
and the case of smooth families of general dimension with ``good compactification'' of both the source and the target (Setting \ref{set:generaldimension}).
\medskip

We start with the first setting.

\begin{sett}
    \label{set:projsm}
$u\colon U\to S$ is a projective smooth morphism of smooth $V$-schemes with $S=\Spec(B)$.
\end{sett}
    
\begin{prop} \label{prop:comparisonisom}
    Under Setting \ref{set:projsm}, the relative rigid cohomology sheaf
    $R^iu_{\rig}j^{\dag}_U\sO_{U_K^{\an}}$
    is a coherent $j^{\dag}_S\sO_{S_K^{\an}}$-module for each $i\geq 0$.
    Consequently, $H^i_{\rig}(U_k/S_k)$ is a coherent $B_K^{\dag}$-module for each $i\geq 0$.
    Moreover, the comparison morphism
 \begin{equation}
            \label{eq:comparisonisom2}
c\colon H^i_{\dR}(U_K/S_K)\otimes_{B_K}B_K^{\dag}\to H^i_{\rig}(U_k/S_k)
\end{equation}
is bijective for each $i\geq 0$.

    In particular, the datum
    \[
        H^i(U/S)=(H^i_{\dR}(U_K/S_K),H^i_{\rig}(U_k/S_k),c,\nabla,\Phi,\Fil^{\bullet})
    \]
    is an object of $\FilFMIC(S)$.
\end{prop}
\begin{pf}
    Firstly, the coherence of the rigid cohomology follows from
    \cite[Th\'eor\`eme 5]{Berthelot81}.
        Now, since $c$ is a morphism of coherent modules over the noetherian ring $B_K^{\dag}$,
        it suffices to prove that it is an isomorphism on the reduction by each maximal ideal of $B_K^{\dag}$,
        which is the extension of a maximal ideal of $B_K$.
        Therefore we may assume that $S=\overline{S}=\Spec(k)$ (after a possible extension of $k$), and
        then the claim follows from comparison of the (absolute) algebraic de Rham cohomology
        and the rigid (or, in this case, crystalline) cohomology (e.g. \cite[4.2]{AndreBaldassarri}, \cite[(7)]{Gerkmann08}).
\end{pf}

The second sufficient condition for $H^i(U/S)$ to be an object of $\FilFMIC(S)$
is, briefly said, that $u$ has a ``good compactification''.

\begin{sett}
    Let $S=\Spec(B)$ be a smooth affine $V$-scheme, let $\overline{S}$ be a projective smooth $V$-scheme
    with an open immersion $S\hookrightarrow\overline{S}$ such that
    the complement $T$ is a relative simple NCD on $\overline{S}$ over $V$.
Let $\overline{X}$ be a projective smooth $V$-scheme,
and let $\overline{f}\colon \overline{X}\to\overline{S}$ be a projective morphism.
Let $\overline{D}^{\mathrm{h}}$  be a relative NCD  on $\overline{X}$  over $V$,
and put $\overline{D}^{\mathrm{v}}=\overline{f}^{-1}(T)$ and $\overline{D}=\overline{D}^{\mathrm{h}}\cup\overline{D}^{\mathrm{v}}$.
We put $X=\overline{X}\setminus \overline{D}^{\mathrm{v}}$, $f=\overline{f}|_X$ and $U=\overline{X}\setminus\overline{D}$.
We then assume that the following conditions hold:
\begin{enumerate}
    \item[(1)] $\overline{D}=\overline{D}^{\mathrm{h}}\cup\overline{D}^{\mathrm{v}}$ is also a relative NCD over $V$.
    \item[(2)] $D:=\overline{D}^{\mathrm{h}}\cap X\hookrightarrow X$ is a relative NCD over $S$. 
    \item[(3)] The morphism $\overline{f}\colon (\overline{X},\overline{D})\to(\overline{S},T)$ is log smooth and integral,
    and $(\overline{S}, T)$ is of Zariski type.
\end{enumerate}
\label{set:generaldimension}
\end{sett}
The notation in Setting \ref{set:generaldimension} can be summarized by the diagram
\[
\begin{tikzcd}
    X \arrow[r, "\overline{D}^{\mathrm{v}}", hook] \arrow[d, "f"'] & \overline{X} \arrow[d, "\overline{f}"] &\arrow[l, "\overline D"',hook']U\arrow[ld,"u"]\\
    S \arrow[r, "T", hook]                                         & \overline{S}                          
\end{tikzcd}
\]
where the notation above $\hookrightarrow$ shows the complement of the subscheme.

\begin{prop}  \label{prop:generaldimension}
    Under Setting \ref{set:generaldimension},
    the relative rigid cohomology sheaf
    $R^iu_{\rig}j^{\dag}_U\sO_{U_K^{\an}}$
    is a coherent $j^{\dag}_S\sO_{S_K^{\an}}$-module for each $i\geq 0$.
    Consequently, $H^i_{\rig}(U_k/S_k)$ is a coherent $B_K^{\dag}$-module for each $i\geq 0$.
    Moreover, the comparison morphism
 \[
c\colon H^i_{\dR}(U_K/S_K)\otimes_{B_K}B_K^{\dag}\to H^i_{\rig}(U_k/S_k)
\]
is bijective for each $i\geq 0$.

    In particular, the datum
    \[
        H^i(U/S)=(H^i_{\dR}(U_K/S_K),H^i_{\rig}(U_k/S_k),c,\nabla,\Phi,\Fil^{\bullet})
    \]
    is an object of $\FilFMIC(S)$.
\end{prop}

\begin{pf}
Our setting assures us that we are in the situation of \cite[Section 2]{Shiho3},
i.e. the assumptions before \cite[Theorem 2.1]{Shiho3} are satisfied.
Moreover, since $f$ is integral, the assumption ($\star$) in \cite[Thereom 2.1]{Shiho3} is also satisfied \cite[Corollary 4.7]{Shiho1}.
Therefore, the coherence of the rigid cohomology sheaves follows from \cite[Theorem 2.2]{Shiho3}.

 Now that the coherence is proved for all $i$,
 the proof reduces to the absolute case as in the proof of Proposition \ref{prop:comparisonisom}.
 Then, since the algebraic de Rham cohomology (resp. the rigid cohomology)
 is isomorphic to the algebraic log de Rham cohomology
 (resp. log rigid cohomology by e.g. \cite[Theorem 3.5.1]{TsuzukiGysin}),
 the claim follows from 
 the comparison theorem between algebraic log de Rham cohomology and log rigid cohomology
 \cite[Corollary 2.6]{BaldassarriChiarellotto}.
\end{pf}

We also have a Gysin exact sequence in $\Fil$-$F$-$\MIC(S)$ for curves under this setting.

\begin{prop}[Gysin exact sequence]
    \label{prop:Gysin}
    Let $U=\Spec(A)$ and $S=\Spec(B)$ be smooth affine $V$-schemes
    and let $u\colon U\to S$ be a smooth morphism of relative dimension one
with connected fibers.
Assume that there exists a projective smooth curve $f\colon X\to S$ with an open immersion $U\hookrightarrow X$
such that $f|_U=u$ and that the complementary divisor $D:=X\setminus U$ is finite \'etale over $S$.
Moreover, assume that $u$ satisfies the conclusions of
Proposition \ref{prop:generaldimension}
\textup{(}namely, the coherence of the rigid cohomology and the bijectivity of the comparison morphism\textup{)},
e.g. that we are in Setting \ref{set:generaldimension}.

    Then, we have an exact sequence
    \[
        0\to H^1(X/S)\to  H^1(U/S)\to H^0(D/S)(-1)\to H^2(X/S)\to 0.
    \]
    in $\FilFMIC(S)$.
\end{prop}
\begin{pf}
    Firstly, $H^i(X/S)$ and $H^i(D/S)$ are objects of $\FilFMIC(S)$ by Proposition \ref{prop:comparisonisom},
    and so is $H^1(U/S)$ by assumption.
    Nextly, it is a standard fact about de Rham cohomology that we have an exact sequence 
    \[
    0\to H^1_{\dR}(X_K/S_K)\to H^1_{\dR}(U_K/S_K)\to H^0_{\dR}(D_K/S_K)(-1)\to H^2_{\dR}(X_K/S_K)\to 0
    \]
    whose morphisms are horizontal and compatible with respect to $\Fil$.
    Therefore, by the comparison isomorphism on each term
    and by the flatness of $B_K^{\dag}$ over $B_K$, we get a corresponding exact sequence
    \[
    0\to H^1_{\rig}(X_k/S_k)\to H^1_{\rig}(U_k/S_k)\to H^0_{\rig}(D_k/S_k)(-1)\to H^2_{\rig}(X_k/S_k)\to 0
    \]
    for rigid cohomologies.
    The compatibility of this sequence with Frobenius structures on each term can be checked on
    each closed point of $\widehat{S}_K$, therefore reduced to the absolute case \cite[Theorem 2.19]{Bannai02}.
\end{pf}

The following lemma is the promised statement in Remark \ref{rem:defofrigidcoh} (2).

\begin{lem}
    \label{lem:vanishing}
    Let $X$ be a smooth $V$-scheme and let $\sM$ be a coherent $j^{\dag}_X\sO_{X_K^{\an}}$-module.
    Then, for any $j\geq 1$, we have $H^j(X_K^{\an},\sM)=0$.
\end{lem}

\begin{proof}
    This lemma is essentially given in the proof of \cite[6.2.12]{LS}.
    We recall the argument for the convenience for the reader.
    Choose a closed immersion $X\hookrightarrow\bA_V^N$ to an affine space and
    let $Y$ be the closure of $X$ in $X\hookrightarrow\bA_V^N\hookrightarrow\bP_V^N$.
    Then, $V_{\rho}:=X_K^{\an}\cap \mathbb{B}^N(0,\rho^+)$ for $\rho>1$ form a cofinal family of strict neighborhoods of $\widehat{X}_K$.

    By the coherence of $\sM$,
    we can take a coherent $\sO_{V_{\rho_0}}$-module $M$ for some $\rho_0>1$
    such that $\sM|_{V_{\rho_0}}=j^{\dag}_{X}M$ \cite[5.4.4]{LS}.
    Then, if $j_{\rho}\colon V_{\rho}\hookrightarrow V_{\rho_0}$ denotes the inclusion for $1<\rho<\rho_0$,
    we have isomorphisms
    \[
    R\Gamma(X_K^{\an}, \sM)=R\Gamma(V_{\rho_0}, \sM|_{V_{\rho_0}})=\varinjlim_{\rho}R\Gamma(V_{\rho_0},(j_{\rho})_{\ast}M)=\varinjlim_{\rho}R\Gamma(V_{\rho},M|_{V_{\rho}}).
    \]
    In fact, the second identification follows from quasi-compactness and separatedness of $V_{\rho_0}$,
    and the third one holds because $j_{\rho}$ is affinoid.
    Now, the claim follows because each $V_{\rho}$ is affinoid and $M|_{V_\rho}$ is coherent.
\end{proof}

\subsection{Extensions associated to Milnor symbols}\label{Ext-Mil-sect}
In this subsection, we discuss how we associate an extension in $\FilFMIC(U)$ to a Milnor symbol.

Let $u\colon U=\Spec(A)\to S=\Spec(B)$ be a smooth morphism of smooth affine $V$-scheme.
We assume that $u$ satisfies the consequences of Proposition \ref{prop:generaldimension}, namely:

\begin{assump}\label{as:isomorphism}
For each $i\geq 0$, the $i$-th relative rigid cohomology sheaf $R^iu_{\rig}j^{\dag}_U\sO_{U_K^{\an}}$
is a coherent $j^{\dag}_{S}\sO_{S_K^{\an}}$-module
and the natural morphism
\begin{equation}
    \label{as:isom}
H^i_{\dR}(U_K/S_K)\otimes_{B_K} B_K^\dag\lra H^i_{\rig}(U_k/S_k)
\end{equation}
is bijective.
\end{assump}

\begin{rem}
    \label{rem:onassumption}
    (1) As we have discussed in Proposition \ref{prop:generaldimension},
    Assumption \ref{as:isomorphism}
    is satisfied if we are in Setting \ref{set:generaldimension}.

    (2) 
    If there is a projective smooth morphism $f\colon X\to S$
    with an open immersion $U\hookrightarrow X$
    such that $X\setminus U$ is a relative simple normal crossing divisor over $S$,
    then the bijectivity of the morphism (\ref{as:isom})
    is deduced from the former half of the Assumption \ref{as:isom} (the coherence of the $i$-th relative rigid cohomology sheaf for each $i\geq 0$)
    as in the proof of Proposition \ref{prop:generaldimension}.

    (3) Note that (a part of) Assumption \ref{as:isomorphism} 
    allows us to interpret the relative rigid cohomology 
    as a cohomology of Monsky--Washnitzer type.

    More precisely, for a smooth morphism $u\colon U=\Spec(A)\to S=\Spec(B)$ of affine smooth $V$-schemes,
    assume that, for each $i\geq 0$, the $i$-th rigid cohomology sheaf $R^iu_{\rig}j^{\dag}_U\sO_{U_K^{\an}}$ satisfies
        $H^j(S_K^{\an}, R^iu_{\rig}j^{\dag}_U\sO_{U_K^{\an}})=0$ for all $j\geq 1$
    (e.g. if all $R^iu_{\rig}j^{\dag}_U\sO_{U_K^{\an}}$ are coherent).
    Then, the $B_K^{\dag}$-module $H^i_{\rig}(U_k/S_k)=\Gamma(S_K^{\an},R^iu_{\rig}j^{\dag}_U\sO_{S_K^{\an}})$
    is isomorphic to the cohomology
    $H^i_{\mathrm{MW}}(U_k/A_k)=H^i(\Omega^{\bullet}_{A_K^{\dag}/B_K^{\dag}})$
    of the complex of continuous differentials $\Omega^{\bullet}_{A_K^{\dag}/B_K^{\dag}}=\Gamma(U_K^{\an}, j^{\dag}_U\Omega^{\bullet}_{U_K^{\an}/S_K^{\an}})$.
    This follows from our assumption
    and the vanishing
    of the cohomologies $H^j(U_K^{\an},j^{\dag}_U\Omega^k_{U_K^{\an}/S_K^{\an}})=0$
    for all $j\geq 1$ and $k\geq 0$ (which also follows from Lemma \ref{lem:vanishing}).
\end{rem}

\bigskip

Recall that, for a commutative ring $R$, the $r$-th Milnor $K$-group $K_r^M(R)$ is defined to be
the quotient of $(R^\times)^{\ot r}$ by the subgroup generated by
\[
a_1\ot\cdots\ot b\ot\cdots\ot(-b)\ot\cdots\ot a_r,\quad
a_1\ot\cdots\ot b\ot\cdots\ot(1-b)\ot\cdots\ot a_r.
\]
Recall from \S \ref{log-obj-sect} the log object $\sLog(f)$ in 
$\FilFMIC(U)$ for $f\in\O(U)^\times$, and the extension \eqref{log-obj-gp}
which represents the class 
\[
[f]_U\in \Ext^1_{\FilFMIC(U)}(\O_U,\O_U(1)).
\]
For $h_0,h_1,\ldots,h_n\in \sO(U)^\times$, we associate an $(n+1)$-extension
\[
    0\lra \sO_{U}(n+1)\lra \sLog(h_n)(n)\lra\cdots\lra\sLog(h_1)(1)
\lra\sLog(h_0)\lra \sO_{U}\lra 0
\]
which represents the class
\[
[h_0]_U\cup\cdots\cup[h_n]_U\in \Ext^{n+1}_{\FilFMIC(U)}(\O_U,\O_U(n+1)).
\]
It is a standard argument to show that the above cup-product
is additive with respect to each $h_i$, so that we have an additive map
\[
(\O(U)^\times)^{\ot n+1}\lra
\Ext^{n+1}_{\FilFMIC(U)}(\O_U,\O_U(n+1)).
\]

\begin{prop}   \label{2-ext-welldef}
$[f]_U\cup[f]_U=0$ for $f\in \O(U)^\times$
and $[f]_U\cup[1-f]_U=0$ for $f\in\O(U)^\times$ such that $1-f\in\O(U)^\times$.
Hence the homomorphism
\[
K^M_{n+1}(\O(U))\lra\Ext^{n+1}_{\FilFMIC(U)}(\sO_{U},\sO_U(n+1)),\quad
\{h_0,\ldots,h_n\}\mapsto [h_0]_U\cup\cdots\cup[h_n]_U
\]
is well-defined (note $\sLog(-f)=\sLog(f)$ by definition).
\end{prop}
\begin{pf}
To prove this, it follows from Lemma \ref{log-obj-lem2}
that we may assume $U=\Spec V[T,T^{-1}]$ and $f=T$ 
for the vanishing $[f]_U\cup[f]_U=0$
and $U=\Spec V[T,(T-T^2)^{-1}]$ and $f=T$
for the vanishing $[f]_U\cup[1-f]_U=0$.

Here, we show the latter vanishing. Let $C=V[T,(T-T^2)^{-1}]$ and $U=\Spec(C)$.
Recall the dilog object
    $D:=\sDilog(T)$ which has a unique increasing filtration $W_{\bullet}$ (as an object of $\Fil$-$F$-$\MIC(U)$)
    that satisfies
    \[
        W_jD_\dR=\begin{cases}
            0 & \text{if } j\leq -5,\\
            C_Ke_{-4} & \text{if } j=-4, -3,\\
            C_Ke_{-4}\oplus C_Ke_{-2} & \text{if } j=-2, -1,\\
            \sDilog(T)_{\dR} & \text{if } j\geq 0
        \end{cases}
    \]
and the filtration $\Fil^{\bullet}$ on $W_jD_\dR$ is given to be
    $\Fil^iW_jD_\dR=W_jD_\dR\cap\Fil^iD_{\dR}$.
    Then, it is straightforward to check that 
    $W_{-4}\cong \O_U(2), W_{-2}\cong\sLog(T)(1), W_0/W_{-4}\cong\sLog(1-T)$ and 
    $W_0/W_{-2}\cong \O_U$.
    Thus, $[T]_U\cup[1-T]_U$ is realized by the extension
    \[0\lra W_{-4}\lra W_{-2}\lra W_0/W_{-4}\lra W_0/W_{-2}\lra 0.\]
    Consider a commutative diagram
   \begin{equation}\label{2-ext-welldef-eq1}
        \xymatrix{
0\ar[r]&\O_U(2)\ar@{=}[d]\ar[r]^\id&\O_U(2)\ar[r]^0&\O_U\ar[r]^\id&\O_U\ar[r]&0\\
            0 \ar[r] & \O_U(2) \ar@{=}[d]\ar[r]^{\iota\qquad} 
& \O_U(2)\oplus W_{-2} \ar[r]^{\qquad\pi_2} \ar[u]_{\pi_1}
            \ar[d]^{\mathrm{add}}& \ar[r] W_0 \ar[d]\ar[u]_{\pi_3}
& \O_U\ar@{=}[d]\ar@{=}[u] \ar[r] & 0\\
            0 \ar[r] & W_{-4} \ar[r] & W_{-2} \ar[r] & W_0/W_{-4} \ar[r] & \O_U \ar[r] & 0,
        }
    \end{equation}
    where $\iota$ is the first inclusion,
$\pi_1$ is the first projection, $\pi_2$ is the composition with the second projection and the inclusion $W_{-2}\hookrightarrow W_0$,
$\pi_3$ is the quotient $W_0\to W_0/W_{-2}\cong\O_U$,
  and $\mathrm{add}\colon (x,y)\mapsto x+y$.
The above diagram shows
 the vanishing $[T]_U\cup[1-T]_U=0$.

The proof of the vanishing $[T]_U\cup[T]_U=0$ goes in a similar way by replacing
$\sDilog(T)$ with $\mathrm{Sym}^2\sLog(T)$.
\end{pf}

Let $n$ be a non-negative integer, let $h_0,\ldots,h_{n}\in\O(U)^\times$
and suppose that $\Fil^{n+1}H^{n+1}(U/S)=0$.
Under this setting, we define an object
\[
M(U/S)_{h_0,\ldots,h_{n}}
\]
in $\FilFMIC(U)$ in the following way.

Let $\cM_{h_0,\ldots,h_{n}}$ be the complex
\[
\sLog(h_n)(n)\lra \cdots\lra\sLog(h_1)(1)\lra\sLog(h_0)
\]
in $\FilFMIC(U)$
where the first term is placed in degree $0$, which fits into a distinguished triangle
\begin{equation}
0\lra \O_U(n)\lra \cM_{h_0,\ldots,h_{n}}\lra \O_U[-n]\lra0
\label{eq:triangleforM}
\end{equation}
in the derived category of $\FilFMIC(U)$.

Firstly, we define the de Rham part of $M(U/S)$.
Let $\cM_{h_0,\ldots,h_n,\dR}$ denote the de Rham realization
of $\cM_{h_0,\ldots,h_n}$.
This can be seen as a complex of mixed Hodge modules by M. Saito
\cite{msaito2}. Therefore the de Rham cohomology
\[
M(U/S)_{h_0,\ldots,h_n,\dR}\defeq
H^n_\dR(U_K/S_K,\cM_{h_0,\ldots,h_n,\dR})=H^n(U_K,
\Omega^\bullet_{U_K/S_K}\ot\cM_{h_0,\ldots,h_n,\dR})
\]
carries the Hodge filtration which we write by $\Fil^\bullet$.
Now, consider the exact sequence
\[
0\lra H^n_\dR(U_K/S_K)\lra M(U/S)_{h_0,\ldots,h_n,\dR}\lra \O_{S_K}\os{\delta}{\lra} H^{n+1}_\dR(U_K/S_K).
\]
It follows from the Hodge theory that $\Image(\delta)\subset \Fil^{n+1}H^{n+1}_\dR(U_K/S_K)=0$. Hence we have an exact sequence
\[
0\lra H^n_\dR(U_K/S_K)\lra M(U/S)_{h_0,\ldots,h_n,\dR}\lra \O_{S_K}\lra0.
\]
In particular, $M(U/S)_{h_0,\ldots,h_n,\dR}$ is locally free.
Therefore, together with the connection $\nabla$ 
as relative de Rham cohomology, 
$(M(U/S)_{h_0,\ldots,h_n,\dR},\nabla,\Fil^{\bullet})$ is an object of $\FilMIC(S_K)$.
Here, The strictness with respect to $\Fil^\bullet$ follows from
the fact that the above can be seen as an exact sequence of variations of
mixed Hodge structures, again by the theory of Hodge modules
\cite{msaito1}, \cite{msaito2}. 

Let $\cM_{h_0,\ldots,h_n,\rig}$ be the corresponding complex
in $\FMIC(A^\dag_K)$ to $\cM_{h_0,\ldots,h_n}$
which can be seen as a complex of overconvergent $F$-isocrystals.
In particular, this can also be seen a complex of coherent $j^{\dag}_U\sO_{U_K^{\an}}$-modules
with Frobenius structure.
We define 
\[
    M(U/S)_{h_0,\ldots,h_n,\rig}\defeq H^n_{\rig}(U_k/S_k, \cM_{h_0,\ldots,h_n,\rig})
    :=\Gamma(S_K^{\an}, R^nu_{\rig}\sM_{h_0,\dots,h_n,\rig}),
\]
where $R^nu_{\rig}\sM_{h_0,\dots,h_n,\rig}
=R^n(u_K^{\an})_{\ast}\left(j^{\dag}_U\Omega^{\bullet}_{U_K^{\an}/S_K^{\an}}\otimes\sM_{h_0,\dots,h_n,\rig}\right)$.

Now, we explain how we get a comparison isomorphism
\[
    c\colon M(U/S)_{h_0,\ldots,h_n,\dR}
\otimes_{B_K}{B_K^{\dag}}\xrightarrow{\quad\cong\quad} 
M(U/S)_{h_0,\ldots,h_n,\rig}.
\]
Firstly, we have a canonical homomorphism
\[
   H^i_\dR(U_K/S_K,\cM_{h_0,\ldots,h_n,\dR})
\otimes_{B_K}{B_K^{\dag}}\lra
   H^i_\rig(U_k/S_k,\cM_{h_0,\ldots,h_n,\rig})
\]
for each $i$.
(The construction is the same as in the case of trivial coefficients in the beginning of the previous subsection.)
To prove that this is an isomorphism,
as in the de Rham part, note that we have an exact sequence
\[
 0\lra H^n_{\rig}(U_k/S_k)\lra M(U/S)_{h_0,\dots,h_n,\rig}\lra B_K^{\dag} \lra H^{n+1}_{\rig}(U_k/S_k)
\]
because $H^1(S_K^{\an}, R^nu_{\rig}j^{\dag}_U\sO_{U_K^{\an}/S_K^{\an}})=0$
by Assumption \ref{as:isomorphism}.
Thus, the comparison homomorphism is an isomorphism by the flatness of $B_K^{\dag}$
over $B_K$, by Assumption \ref{as:isomorphism} and by five lemma.

\medskip

Now, we have constructed an object $M_{h_0,\ldots,h_n}(U/S)$ in $\FilFMIC(S)$
which fits into the exact sequence
\begin{equation}\label{motive-obj}
    0\lra H^n(U/S)(n+1)\lra M_{h_0,\ldots,h_n}(U/S)\lra \O_S\lra0.
\end{equation}
The extension class \eqref{motive-obj}
is additive with respect to each $h_i$
(one can show this in the same way as the proof of bi-additivity of
$[h_0]_U\cup\cdots\cup[h_n]_U$, but based on the theory of
mixed Hodge modules concerning the strictness of the filtrations;
for the rigid part, we use the functoriality of the relative rigid cohomology
and the comparison to algebraic de Rham cohomology),
so that we have a homomorphism
\begin{equation}\label{motive-obj-eq1}
(\O(U)^\times)^{\ot n+1}\lra\Ext^1_{\FilFMIC(S)}(\O_S,H^n(U/S)(n+1)).
\end{equation}
\begin{thm}\label{mot-map-prop}
Suppose $\Fil^{n+1}H^{n+1}_\dR(U_K/S_K)=0$.
Then, under the assumption \ref{as:isomorphism}, the homomorphism \eqref{motive-obj-eq1} factors through the Milnor $K$-group,
so that we have a map
\begin{equation}\label{motive-map}
[-]_{U/S}:K^M_{n+1}(\O(U))\lra\Ext^1_{\FilFMIC(S)}(\O_S,H^n(U/S)(n+1))
\end{equation}
which we call the {\it symbol map for $U/S$}.
\end{thm}

\begin{rem}
If it were possible to define a natural projection
\[
\Ext^{n+1}(\O_U,\O_U(n+1))\lra 
\Ext^1(\O_S,H^n(U/S)(n+1))
\]
under the assumption $H^{n+1}(U/S)=0$, then 
the object $M_{h_0,\ldots,h_n}(U/S)$ should correspond to the class 
$[h_0]_U\cup
\cdots\cup[h_n]_U$, and hence Theorem \ref{mot-map-prop} would be immediate
from Proposition \ref{2-ext-welldef}.
However this is impossible since we do not take into consideration
the boundary conditions such as admissibility etc. for constructing our category.
We need to prove Theorem \ref{mot-map-prop} independently while
almost same argument works as well.
\end{rem}
\begin{proof}
Write the homomorphism \eqref{motive-obj-eq1} by $\wt\rho$.
It is enough to show the following.
\[
\wt\rho(h_0\ot\cdots\ot f\ot f\ot\cdots\ot h_n)=0,\quad
\wt\rho(h_0\ot\cdots\ot f\ot (1-f)\ot\cdots\ot h_n)=0.
\]
We show the latter. 
Let $f\in \O(U)^\times$ such that $1-f\in\O(U)^\times$.
Let $u:U\to \Spec V[T,(T-T^2)^{-1}]$ be the morphism given by $u^*T=f$.
Recall the diagram \eqref{2-ext-welldef-eq1}, and take the pull-back by $u$.
It follows that
\[
0\to\O_U(n+1)\to
\overbrace{\sLog(h_n)\to\cdots\to\sLog(f)(i+1)\to\sLog(f)(i)\to\cdots
\to\sLog(h_0)}^{\cM}\to\O_U\to0
\]
is equivalent to
\[
0\to\O_U(n+1)\to
\overbrace{\sLog(h_n)\to\cdots\to\O_U(i+2)\os{0}{\to}
\underbrace{\O_U(i)\to\cdots\to\sLog(h_0)}_{\mathscr N}}^{\cM'}\to\O_U\to0
\]
as $(n+1)$-extension in $\FilFMIC(U)$.
This implies that $\cM$ is quasi-isomorphic to $\cM'$ as mixed Hodge modules
or $F$-isocrystals. Hence $\mathscr N$ gives rise to a splitting of
\[
    0\lra H^n(U/S)(n+1)\lra M_{h_0,\ldots,f,1-f,\ldots,h_n}(U/S)\lra \O_S\lra0.
\]
This completes the proof of 
$\wt\rho(h_0\ot\cdots\ot f\ot (1-f)\ot\cdots\ot h_n)=0$.

To see $\wt\rho(h_0\ot\cdots\ot f\ot f\ot\cdots\ot h_n)=0$,
we consider the similar diagram to \eqref{2-ext-welldef-eq1}
obtained from $\mathrm{Sym}^2\sLog(T)$.
Then the rest is the same.
\end{proof}

\subsection{Explicit formula}\label{comp-sect}
In this subsection, we continue to use the setting of the previous subsection.
In particular, $U=\Spec(A)\to S=\Spec(B)$ is a smooth morphism of smooth $V$-schemes
that satisfies
Assumption \ref{as:isomorphism}.
We fix a $p$-th Frobenius endomorphism $\varphi$ (resp. $\sigma$) on $A^\dag$ (resp. $B^{\dag}$).

Let $n\geq 0$ be an integer. Suppose that $\Fil^{n+1}H^{n+1}_\dR(U_K/S_K)=0$. 
Recall the maps
\[K_{n+1}^M(A)\os{[-]_{U/S}}{\lra} 
\Ext^1_{\FilFMIC(S)}(\O_S,H^n(U/S)(n+1))
\os{\eqref{filFMIC-eq6}}{\lra} H^1(\cS(H^n(U/S)(n+1))).
\]
Note that the last cohomology group is
\[
\left\{(\omega,\xi)\in \left(\Omega^1_{B_K}\ot
\Fil^nH^n_\dR(U_K/S_K)\right)\op\left( B^\dag_K\ot H^n_\dR(U_K/S_K)\right)
\bigg| (1-\varphi_{n+1})\omega=d\xi,\,d(\omega)=0
\right\}
\]
where $\varphi_i:=p^{-i}\varphi$ and
$d$ is the differential map induced from the Gauss-Manin connection 
as in \eqref{filFMIC-eq1}.
Let
\begin{equation}\label{phi.sigma}
\xymatrix{
&\Ext^1_{\FilFMIC(S)}(\O_S,H^n(U/S)(n+1))\ar[rd]^{R_\sigma}\ar[ld]_\delta\\
\Omega^1_{B_K}\ot
\Fil^nH^n_\dR(U_K/S_K)&&B^\dag_K\ot H^n_\dR(U_K/S_K)
}
\end{equation}
be the compositions of \eqref{filFMIC-eq6} and the projections $(\omega,\xi)\mapsto\omega$ and
$(\omega,\xi)\mapsto\xi$ respectively.
The map $\delta$ agrees with the map \eqref{filFMIC-eq7}, and hence
it does not depend on $\sigma$,
while $R_\sigma$ does.
We put
\begin{equation}\label{comp-eq2}
D_{U/S}:=\delta\circ[-]_{U/S},\quad
\reg^{(\sigma)}_{U/S}:=R_\sigma\circ[-]_{U/S}.
\end{equation}
The purpose of this section is to 
give an explicit description of these maps.

\medskip

\begin{thm}\label{comp-thm}
Suppose $\Fil^{n+1}H^{n+1}_\dR(U_K/S_K)=0$.
Let $\xi=\{h_0,\ldots,h_n\}\in K^M_{n+1}(A)$.
Then
\begin{equation}\label{comp-thm-eq1}
D_{U/S}(\xi)=(-1)^n\frac{dh_0}{h_0}\wedge\frac{dh_1}{h_1}\wedge\cdots \wedge\frac{dh_n}{h_n},
\end{equation}
\begin{equation}\label{comp-thm-eq2}
\reg^{(\sigma)}_{U/S}(\xi)
=(-1)^n\sum_{i=0}^n(-1)^ip^{-1}\log\left(\frac{h_i^p}{h_i^{\varphi}}\right)
\left(\frac{dh_0}{h_0}\right)^{\varphi_1}\wedge\cdots\wedge\left(\frac{dh_{i-1}}{h_{i-1}}\right)^{\varphi_1}\wedge
\frac{dh_{i+1}}{h_{i+1}}\wedge\cdots\wedge\frac{dh_n}{h_n}.
\end{equation}
Here one can think of \eqref{comp-thm-eq1} as an element of
\[
\Omega^1_{B_K}
\ot \Fil^nH^n_\dR(U_K/S_K)
=\Omega_{B_K}^1\ot\vg(X_K,\Omega^n_{X_K/S_K}(\log D_K))
\]
in the following way.
Let $\wt\Omega^\bullet_{X_K}(\log D_K):=\Omega^\bullet_{X_K}(\log D_K)/\Image(\Omega^2_{S_K}
\ot  \Omega^{\bullet-2}_{X_K}(\log D_K))$ which fits into the exact sequence
\[
0\lra \Omega_{S_K}^1\ot\Omega^{i-1}_{X_K/S_K}(\log D_K)
\lra \wt\Omega^i_{X_K}(\log D_K)\lra \Omega^i_{X_K/S_K}(\log D_K)\lra0.
\]
We think \eqref{comp-thm-eq1} of being an element of 
$\vg(X_K,\wt\Omega^{n+1}_{X_K}(\log D_K))$.
However, since $\Fil^{n+1}H^{n+1}_\dR(U_K/S_K)=0$ by the assumption,
it turns out to be an element of
\[
\vg(X_K,\Omega_{S_K}^1\ot\Omega^n_{X_K/S_K}(\log D_K))
=
\Omega^1_{B_K}\ot\vg(X_K,\Omega^n_{X_K/S_K}(\log D_K)).
\]
\end{thm}
\begin{pf}
In this proof, we omit to write the symbol ``$\wedge$".

Firstly, we describe the extension $[\xi]_{U/S}$.
Let
\[
0\lra H^n_\dR(U/S)(n+1)\os{\iota}{\lra} M_\xi(U/S)\lra \O_S\lra0
\]
be the extension $[\xi]_{U/S}$ in $\FilFMIC(S,\sigma)$.
Let $e_\xi\in \Fil^0M_\xi(U/S)_\dR$ be the unique lifting of $1\in \O(S_K)$.
Then
\begin{equation}\label{comp-thm-pf-eq0}
\iota (D_{U/S}(\xi))=\nabla(e_\xi),\quad
\iota(\reg_{U/S}^{(\sigma)}(\xi))=(1-\Phi)e_\xi
\end{equation}
by definition (see also \eqref{filFMIC-lem-eq1})
where $\nabla$ and $\Phi$ are the data in $M_\xi(U/S)\in\FilFMIC(S,\sigma)$.

We first write down the term $M_\xi(U/S)$ explicitly.
Write $l_i:=p^{-1}\log(h_i^p/h_i^{\varphi})$. 
Let $\{e_{i,0},e_{i,-2}\}$ be the basis of $\Log(h_i)(i)$ such that 
$\Fil^{-i}\sLog(h_i)(i)_\dR=A_Ke_{i,0}$ and
\[
\nabla(e_{i,0})=\frac{dh_i}{h_i} e_{i,-2},\quad
\Phi(e_{i,0})=p^{-i}e_{i,0}-p^{-i}l_ie_{i,-2},\quad
\Phi(e_{i,-2})=p^{-i-1}e_{i,-2}.\]
Recall the $(n+1)$-extension
\[
0\to \O_{U}(n+1)\to \Log(h_n)(n)\to\cdots\to\Log(h_1)(1)\to\Log(h_0)\to\O_{U}\to0.
\]
Let $(T^\bullet_{A_K/B_K},D)$ be the total complex of the double complex
$\Omega^\bullet_{A_K/B_K}\ot \Log(h_\star)_\dR$.
In more down-to-earth manner, we have
\[
T^q_{A_K/B_K}:=\bigoplus_{i=0}^d\Omega^i_{A_K/B_K}\ot\Log(h_{i-q})(i-q)_\dR,\quad q\in\Z
\]
where we denote $\Log(h_j)(j):=0$ if $j<0$ or $j>n$, and the differential $D:T^q\to T^{q+1}$ is defined
by
\[
D(\omega^i\ot x_j)=d\omega^i\ot x_j+(-1)^i\omega^i\wedge \nabla(x_j)+(-1)^i\omega^i\ot
\pi(x_j)
\]
for $\omega^i\ot x_j\in\Omega^i_{A_K/B_K}\ot\Log(h_j)(j)_\dR$, where
$\pi:\Log(h_i)(i)\to A_Ke_{i,0}$ is the projection.
We have the exact sequence 
\begin{equation}
    \label{eq:triangleforT}
\xymatrix{
0\ar[r]&\Omega^{\bullet+n}_{A_K/B_K}
\ar[r]&T^\bullet_{A_K/B_K}\ar[r]&\Omega^\bullet_{A_K/B_K}\ar[r]&0
}
\end{equation}
where the first arrow is induced from 
$\O_{U_K}\cong \O_{U_K}e_{d,-2}\hra \Log(h_d)_\dR$,
the second arrow induced from the projection $\Log(h_0)_\dR\to\O_{U_K}e_{0,0}\cong \O_{U_K}$,
and the differential on $\Omega^{\bullet+n}_{A_K/B_K}$ is the usual differential operator
$d$ (not $(-1)^nd$).
For the rigid part, let $(T_{A^\dag_K/B^\dag_K}^{\bullet},D)$ be defined in the same way by replacing 
$\Omega^\bullet_{A_K/B_K}$ with $\Omega^\bullet_{A^\dag_K/B^\dag_K}$
(see Remark \ref{rem:onassumption} (3)).
Then, we also have an exact sequence corresponding to (\ref{eq:triangleforT}).
Now, we have a description
\[
M_\xi(U/S)_\dR=H^0(T^\bullet_{A_K/B_K}),\quad
M_\xi(U/S)_\rig=H^0(T^\bullet_{A^\dag_K/B^\dag_K})
\]
(the description of the rigid part follows from the exact sequence (\ref{eq:triangleforT})
on two sides with Remark \ref{rem:onassumption} (3)),
and we have an exact sequence
\[
\xymatrix{
0\ar[r]&H^n(U/S)(n+1)\ar[r]&M_\xi(U/S)\ar[r]&\O_{S}\ar[r]&0
}
\]
in $\FilFMIC(S,\sigma)$.

\medskip

Before going to the proof of Theorem \ref{comp-thm}, 
we give an explicit description of $e_\xi$ in \eqref{comp-thm-pf-eq0}.
Put
\[
\omega^0:=1,\qquad
\omega^i:=\frac{dh_0}{h_0}\frac{dh_1}{h_1}\cdots \frac{dh_{i-1}}{h_{i-1}}\in \Omega^i_{A_K/B_K}\quad (1\leq i\leq n+1).
\]
However we note that $\omega^{n+1}=\omega^n\wedge\frac{dh_n}{h_n}=0$ as $\omega^n
\in \Fil^{n+1}H^{n+1}_\dR(U_K/S_K)=0$.
We claim 
\begin{equation}\label{comp-thm-pf-eq1}
e_\xi=\sum_{i=0}^{n}\omega^ie_{i,0}\in T^0_{A_K/B_K}.
\end{equation}
Indeed it is a direct computation to show $D(e_\xi)=0$.
We see $e_\xi\in \Fil^0M_\xi(U/S)_\dR$ in the following way.
Let $j:U_K\hra X_K$. We think $T^\bullet_{A_K/B_K}$ of being a complex of 
$\O_{U_K}$-modules.
Let $\Fil^0\sLog(h_i)_{X_K}=\O_{X_K} e_{i,0}$,
$\Fil^k\sLog(h_i)_{X_K}=\O_{X_K} e_{i,0}+\O_{X_K} e_{i,-2}$ for $k<0$,
$\Fil^k\sLog(h_i)_{X_K}=0$ for $k>0$, and put
\[\Fil^k{\mathscr T}^q_{U/S}=\bigoplus_{i=0}^d\Omega^i_{X_K/S_K}(\log D_K)\ot
\Fil^{k-q}\Log(h_{i-q})_{X_K}\subset j_*T^q_{A_K/B_K}
\]
locally free $\O_{X_K}$-modules.
Then 
\[
F^kM_\xi(U/S)_\dR=\vg(X_K,\Fil^k{\mathscr T}_{U/S}^\bullet).
\]
Therefore $e_\xi\in \Fil^0M_\xi(U/S)_\dR$.

\medskip

We prove \eqref{comp-thm-eq2}. 
Apply $1-\Phi$ on \eqref{comp-thm-pf-eq1}. We have
\begin{equation}
(1-\Phi)e_\xi=\sum_{i=0}^n
\omega^ie_{i,0}-\varphi_i(\omega^i)(e_{i,0}-l_ie_{i,-2})
=\sum_{i=0}^n
(1-\varphi_i)(\omega^i)e_{i,0}+l_i\varphi_i(\omega^i)e_{i,-2}.
\label{comp-thm-pf-eq2}
\end{equation}
Put
\[
Q_0:=l_0,\qquad 
Q_k:=\sum_{i=0}^k(-1)^il_i
\left(\frac{dh_0}{h_0}\right)^{\varphi_1}\cdots\left(\frac{dh_{i-1}}{h_{i-1}}\right)^{\varphi_1}
\frac{dh_{i+1}}{h_{i+1}}\cdots\frac{dh_k}{h_k}\quad (1\leq k\leq n).
\]
A direct calculation yields
\[
dQ_k=(1-\varphi_{k+1})\left(\frac{dh_0}{h_0}\cdots\frac{dh_k}{h_k}\right)
=(1-\varphi_{k+1})(\omega^{k+1}).
\]
It follows
\begin{align*}
&D(Q_ke_{k+1,0})\\
&=(1-\varphi_{k+1})(\omega^{k+1})
e_{k+1,0}+(-1)^kQ_k\frac{dh_{k+1}}{h_{k+1}}e_{k+1,-2}
+(-1)^kQ_ke_{k,-2}\\
&=(1-\varphi_{k+1})(\omega^{k+1})e_{k+1,0}
+(-1)^k(Q_{k+1}-(-1)^{k+1}l_{k+1}\varphi_{k+1}(\omega^{k+1}))
e_{k+1,-2}
+(-1)^kQ_ke_{k,-2}\\
&=
[(1-\varphi_{k+1})(\omega^{k+1})e_{k+1,0}+l_{k+1}\varphi_{k+1}(\omega^{k+1})e_{k+1,-2}]
+(-1)^kQ_ke_{k,-2}-
(-1)^{k+1}Q_{k+1}e_{k+1,-2}
\end{align*}
for $0\leq k\leq n-1$.
Hence
\begin{align*}
D\left(\sum_{k=0}^{n-1}Q_ke_{k+1,0}\right)&=
(1-\Phi)(e_\xi)
-l_0e_{0,-2}+(-1)^n(1-\varphi_{n+1})(\eta)e_{n,-2}
+Q_0e_{0,-2}-(-1)^nQ_ne_{n,-2}\\
&=(1-\Phi)(e_\xi)-(-1)^nQ_ne_{n,-2}
\end{align*}
by \eqref{comp-thm-pf-eq2}.
This shows (cf. \eqref{comp-thm-pf-eq0})
\[
\reg^{(\sigma)}_{U/S}(\xi)=(-1)^nQ_n
=(-1)^n\sum_{i=0}^n(-1)^il_i
\left(\frac{dh_0}{h_0}\right)^{\varphi_1}\cdots\left(\frac{dh_{i-1}}{h_{i-1}}\right)^{\varphi_1}
\frac{dh_{i+1}}{h_{i+1}}\cdots\frac{dh_n}{h_n}
\]
as required.

We prove \eqref{comp-thm-eq1}. Note that $D_{U/S}(\xi)$ is characterized by
$D_{U/S}(\xi)e_{n,-2}=\nabla(e_\xi)$ where
$\nabla$ is the Gauss-Manin connection on $M_\xi(U/S)_\dR$
(cf. \eqref{comp-thm-pf-eq0}).
Let $T_{A_K}^\bullet$ be the complex defined in the same way as $T_{A_K/B_K}^\bullet$
by replacing $\Omega^\bullet_{A_K/B_K}$ with $\Omega^\bullet_{A_K}$.
Let
\[
\wt\omega^0:=1,\qquad
\wt\omega^i:=\frac{dh_0}{h_0}\frac{dh_1}{h_1}\cdots \frac{dh_{i-1}}{h_{i-1}}\in 
\Omega^i_{A_K}\quad (1\leq i\leq n+1),
\]
and $\wt e_\xi:=\sum_{i=0}^{n}\wt\omega^ie_{i,0}\in T^0_{A_K}$.
Then
\[
D(\wt e_\xi)=\sum_{i=0}^n(-1)^i\left(
\frac{dh_0}{h_0}\cdots\frac{dh_i}{h_i}e_{i,-2}+
\frac{dh_0}{h_0}\cdots\frac{dh_{i-1}}{h_{i-1}}e_{i-1,-2}\right)
=(-1)^n\frac{dh_0}{h_0}\cdots\frac{dh_n}{h_n}e_{n,-2}
\]
and this shows (cf. \eqref{comp-thm-eq1}),
\begin{equation}\label{comp-thm-pf-eq3}
\nabla(e_\xi)=(-1)^n\frac{dh_0}{h_0}\cdots\frac{dh_n}{h_n}e_{n,-2}
\in \Omega_{B_K}^1\ot\vg(X_K,\Omega^n_{X_K/S_K}(\log D_K))e_{n,-2}
\end{equation}
as required.
\end{pf}

\section{Comparison with the Syntomic Symbol maps}
In this section we compare the symbol map $[-]_{U/S}$ introduced in \S \ref{Ext-Mil-sect} with the symbol maps to the log syntomic cohomology.
We refer to Kato's article \cite{Ka3} for the formulation and terminology of log schemes.
Throughout this section,
we write $(X_n,M_{n}):=(X,M)\ot\Z/p^n\Z$ for a log scheme $(X,M)$.
\subsection{Log syntomic cohomology}\label{logsyn-sect}
We work over a fine log scheme $(S,L)$, flat over $\Z_{(p)}$.
We endow the DP-structure $\gamma$ on $I=p\O_S$ compatible with the canonical DP-structure on $p\Z_{(p)}$.
The log de Rham complex for a morphism
$f:(X,M_X)\to(S,L)$ of log schemes
is denoted by $\omega^\bullet_{X/S}$ (\cite[(1.7)]{Ka3}).
\begin{prop}[{\cite[Corollary 1.11]{T}}]\label{logsyn-prop1}
Let $(Y_n,M_n)$ be a fine log scheme over $(S_n,L_n)$.
Let $(Y_n,M_n)\hra (Z_n,N_n)$ be an $(S_n,L_n)$-closed immersion.
Assume that $(Z_n,N_n)$ has $p$-bases over $(S_n,L_n)$ locally in the sense of 
\cite[Definition 1.4]{T}.
Let $(D_n,M_{D_n})$ be the DP-envelope of $(Y_n,M_n)$ in $(Z_n,N_n)$.
Let $J^{[r]}_{D_n}\subset \O_{Z_n}$ be the $r$-th DP-ideal of $D_n$.
Then the complex
\[
J^{[r-\bullet]}_{D_n}\ot\omega^\bullet_{Z_n/S_n}
\]
of sheaves on $(D_n)_\et=(Y_n)_\et$ does not depend on the embedding $(Y_n,M_n)
\hra (Z_n,N_n)$.
In particular if $(Y_n,M_n)$ has $p$-bases over $(S_n,L_n)$ locally, then
the natural morphism
\begin{equation}\label{logsyn-prop1-eq1}
J^{[r-\bullet]}_{D_n}\ot\omega^\bullet_{Z_n/S_n}\lra \omega^{\bullet\geq r}_{Y_n/S_n}
\end{equation}
is a quasi-isomorphism.
\end{prop}
Concerning $p$-bases of log schemes, 
the following result is sufficient in most cases.
\begin{lem}[{\cite[Lemma 1.5]{T}}]\label{p-bases-lem}
If $f:(X,M_X)\to (Y,M_Y)$ is a smooth morphism of fine log schemes over
$\Z/p^n\Z$, then $f$ has $p$-bases locally.
\end{lem}

Following \cite[\S 2]{T}, 
we say a collection $\{(X_n,M_{X,n},i_n)\}_n$ (abbreviated to $\{(X_n,M_{X,n})\}_n$)
an adic inductive system of fine log schemes, where
$(X_n,M_{X,n})$ are fine log schemes over $\Z/p^n\Z$, and $i_n:
(X_n,M_{X,n})\to(X_{n+1},M_{X,n+1})\ot\Z/p^n\Z$ are isomorphisms.

Let 
$\{(Y_n,M_n)\}_n\hra \{(Z_n,N_n)\}_n$ be an exact closed immersion 
of adic inductive systems of fine log schemes over $\{(S_n,L_n)\}_n$.
We consider the following condition.
\begin{cond}\label{logsyn-cond}
\begin{enumerate}
\item[(1)]
For each $n\geq 1$, $(Z_n,N_n)$ has $p$-bases over $(S_n,L_n)$ locally,
\item[(2)]
Let $(D_n,M_{D_n})\to (Z_n,N_n)$ be the DP-envelope of $(Y_n,M_n)$
compatible with the DP-structure on $(S_n,L_n)$, and let $J^{[r]}_{D_n}\subset \O_{D_n}$
be the $r$-th DP-ideal.
Then
$J^{[r]}_{D_n}$ is flat over $\Z/p^n\Z$ and $J^{[r]}_{D_{n+1}}\ot\Z/p^n\Z\cong J^{[r]}_{D_n}$.
\end{enumerate}
\end{cond}
Suppose that there is a compatible system $\sigma=\{\sigma_n:(S_n,L_n)\to(S_n,L_n)\}_n$
of $p$-th Frobenius endomorphisms. 
\begin{cond}\label{logsyn-cond2}
There is a hypercovering $\{Y_n^\star\}_n\to 
\{Y_n\}_n$ in 
the etale topology which admits an embedding 
$\{(Y^\star_n,M_n^\star)\}_n\hra \{(Z^\star_n,N^\star_n)\}_n$ into an adic inductive system of simplicial log schemes over $\{(S_n,L_n)\}_n$, such that
each $\{(Y^\nu_n,M_n^\nu)\}_n\hra \{(Z^\nu_n,N^\nu_n)\}_n$ satisfies
Condition \ref{logsyn-cond},
and there is a $p$-th Frobenius
$\{\varphi^\nu_{Z_n}:(Z^\nu_n,N^\nu_n)\to(Z^\nu_n,N^\nu_n)\}_n$ compatible with $\sigma$.
\end{cond}
Let $(Y,M)=\{(Y_n,M_n)\}_n\to \{(S_n,L_n)\}$ together with the $p$-th Frobenius
$\sigma$ on $(S,L)$
satisfy Condition \ref{logsyn-cond2}.
We define the log syntomic complexes 
according to the construction in \cite[p.539--541]{T}\footnote{Tsuji \cite{T} defined the log syntomic complexes only in case $S=\Spec \Z_p$
with trivial log structure. However the same construction works in general
as long as $\sigma$ is fixed.}.
Let $r\geq0$ be an integer.
Let $(D^\nu_n,M_{D^\nu_n})\to (Z^\nu_n,N^\nu_n)$ be the DP-envelopes of $(Y^\nu_n,M^\nu_n)$
compatible with the DP-structure on $(S_n,L_n)$, and $J^{[r]}_{D^\nu_n}\subset \O_{D^\nu_n}$
be the $r$-th DP-ideal.
Let
\[
\bJ^{[r],\bullet}_{n,(Y^\nu,M^\nu),(Z^\nu,N^\nu)/(S,L)}:=J^{[r-\bullet]}_{D^\nu_n}\ot\omega^\bullet_{Z^\nu_n/S_n}
\]
be the complex of sheaves on $(D_1^\nu)_\et=(Y_1^\nu)_\et$. 
We also write 
\[\bO^\bullet_{n,(Y^\nu,M^\nu),(Z^\nu,N^\nu)/(S,L)}=\bJ^{[0],\bullet}_{n,(Y^\nu,M^\nu),(Z^\nu,N^\nu)/(S,L)}.\]
If $r<p$, then there is the well-defined morphism
\[
\varphi^\nu_r:
\bJ^{[r],\bullet}_{n,(Y^\nu,M^\nu),(Z^\nu,N^\nu)/(S,L)}\lra \bO^\bullet_{n,(Y^\nu,M^\nu),(Z^\nu,N^\nu)/(S,L)}
\]
satisfying $p^r\varphi^\nu_r=(\varphi^\nu_{Z_n})^*$ (cf. \cite[p.540]{T}).
It follows from Proposition \ref{logsyn-prop1} that one can ``glue''
the complexes $\bJ^{[r],\bullet}_{n,(Y^\nu,M^\nu),(Z^\nu,N^\nu)/(S,L)}$ so that we have
a complex
\begin{equation}\label{logsyn-eq1}
\bJ^{[r],\bullet}_{n,(Y,M),(Z^\star,N^\star)/(S,L)}
\end{equation}
in the derived category of sheaves on $(Y_1)_\et$ (\cite[p.541]{T}, 
see also \cite[Remark(1.8)]{Ka2}).
Moreover the Frobenius $\varphi_r^\nu$ are also glued as $\sigma$ is fixed
(this can be shown by the same argument as in \cite[p.212]{Ka2}),
we have
\[
\varphi_r:
\bJ^{[r],\bullet}_{n,(Y,M),(Z^\star,N^\star)/(S,L)}\lra \bO^\bullet_{n,(Y,M),(Z^\star,N^\star)/(S,L)}.
\]
Then we define the log syntomic complex
$\cS_n(r)_{(Y,M),(Z^\star,N^\star)/(S,L,\sigma)}$ to be the mapping fiber of
\[
1-\varphi_r:
\bJ^{[r],\bullet}_{n,(Y,M),(Z^\star,N^\star)/(S,L)}\lra
\bO^\bullet_{n,(Y,M),(Z^\star,N^\star)/(S,L)}.
\]
In more down-to-earth manner, the degree $q$-component of 
$\cS_n(r)_{(Y,M),(Z^\star,N^\star)/(S,L,\sigma)}$ is
\[
\bJ^{[r],q}_{n,(Y,M),(Z^\star,N^\star)/(S,L)}\op\bO^{q-1}_{n,(Y,M),(Z^\star,N^\star)/(S,L)}
\]
and the differential maps are given by
$(\alpha,\beta)\mapsto(d\alpha,(1-\varphi_{d+1})\alpha-d\beta)$.
We define the log syntomic cohomology groups by
\[
H^i_\syn((Y,M)/(S,L,\sigma),\Z/p^n(r))
:=H^i_\et(Y_1,\cS_n(r)_{(Y,M),(Z^\star,N^\star)/(S,L,\sigma)}),
\]
\[
H^i_\syn((Y,M)/(S,L,\sigma),\Z_p(r))
:=\varprojlim_n H^i_\syn((Y,M)/(S,L,\sigma),\Z/p^n(r))
\]
and $H^i_\syn((Y,M)/(S,L,\sigma),\Q_p(r)):=H^i_\syn((Y,M)/(S,L,\sigma),\Z_p(r))\ot\Q$.
Note that they depend on the choice of $\sigma$.

\medskip

Let $(Y',M')\to (S',L',\sigma')$ satisfy Condition \ref{logsyn-cond2}, and 
let\[
\xymatrix{
(Y',M')\ar[d]\ar[r]^f&(Y,M)\ar[d]\\
(S',L')\ar[r]&(S,L)
}
\]
be a commutative diagram of fine log schemes in which $\sigma$ and $\sigma'$ are compatible.
Then there is the pull-back
\[
f^*:H^i_\syn((Y,M)/(S,L,\sigma),\Z/p^n(r))
\lra H^i_\syn((Y',M')/(S',L',\sigma'),\Z/p^n(r)).
\]

\begin{prop}\label{logsyn-exact-prop}
Let $0\leq r<p$.
Suppose that $(Y,M)\to(S,L)$ is smooth. 
Let $\Phi_r:
H^{i}_\zar(Y_n,\omega^{\bullet\geq r}_{Y_n/S_n})
\to H^i_\zar(Y_n,\omega^\bullet_{Y_n/S_n})$ be the $\sigma$-linear map 
induced from $\varphi_r$.
Then there is an exact sequence
\begin{align*}
\cdots\lra 
&H^{i-1}_\zar(Y_n,\omega^{\bullet\geq r}_{Y_n/S_n})
\os{1-\Phi_r}{\lra} 
H^{i-1}_\zar(Y_n,\omega^\bullet_{Y_n/S_n})
\lra\\
H^i_\syn((Y,M)/(S,L,\sigma),\Z/p^n(r))
\lra 
&H^i_\zar(Y_n,\omega^{\bullet\geq r}_{Y_n/S_n})
\os{1-\Phi_r}{\lra} 
H^{i}_\zar(Y_n,\omega^\bullet_{Y_n/S_n})
\lra\cdots
\end{align*}
\end{prop}
\begin{pf}
Since $(Y,M)\to(S,L)$ is smooth, it has $p$-bases locally by Lemma \ref{p-bases-lem}.
Therefore the natural morphism \eqref{logsyn-prop1-eq1} is a quasi-isomorphism, and
the exact sequence follows.
\end{pf}
If $S=\Spec V$ with $V$ a $p$-adically complete discrete valuation ring
and $Y\to S$ is proper,
the exact sequence in Proposition \ref{logsyn-exact-prop} remains true
after talking the projective limit with respect to $n$.
Indeed, each term is a $V/p^nV$-module of finite length, so that the Mittag-Leffler
condition holds.
The author does not know whether it is true for general $(S,L)$.

\medskip

We attach the following lemma which we shall often use.

\begin{lem}\label{syn-cpx-hom-lem}
For a ring $A$, let 
$A^\wedge:=\varprojlim A/p^nA$ denote
the $p$-adic completion.
Let $(S,L)$ be a fine log scheme flat over $\Z_{(p)}$ such that $S$ is affine and noetherian.
Let $(Y,M)\to (S,L)$ be
a smooth morphism of fine log schemes such that $Y\to S$ is proper.
Let $(Y_n,M_n)=(Y,M)\ot\Z/p^n\Z$ and $(S_n,L_n)=(S,L)\ot\Z/p^n\Z$.
Then \[
\O(S)^\wedge\ot_{\O(S)}H^i_\zar(Y,\omega^{\bullet\geq k}_{Y/S})
\os{\cong}{\lra}
\varprojlim_n H^i_\zar(Y_n,\omega^{\bullet\geq k}_{Y_n/S_n})
\]
for any $k$.
Note that $\O(S)^\wedge$ is a noetherian ring (\cite[Theorem 10.26]{atiyah-mac}).
\end{lem}
\begin{pf}
For an abelian group $M$ and an integer $n$, 
we denote by $M[n]$ the kernel of the multiplication by $n$.
An exact sequence
\[
0\lra \omega^{\bullet\geq k}_{Y/S}\os{p^n}{\lra} \omega^{\bullet\geq k}_{Y/S}
\lra \omega^{\bullet\geq k}_{Y_n/S_n}\lra0
\]
gives rise to an exact sequence
\[
0\lra 
H^i_\zar(Y,\omega^{\bullet\geq k}_{Y/S})/p^n
\lra H^i_\zar(Y_n,\omega^{\bullet\geq k}_{Y_n/S_n})
\lra H^{i+1}_\zar(Y,\omega^{\bullet\geq k}_{Y/S})[p^n]\lra 0
\]
of finitely generated $\O(S)$-modules as $Y\to S$ is proper.
Therefore the assertion is reduced to show that, for any
$p$-torsion free noetherian ring $A$ 
and any
finitely generated $A$-module $M$, 
\begin{equation}\label{M-lem-1}
A^\wedge\ot_AM\os{\cong}{\lra}\varprojlim_n M/p^nM
\end{equation}
where the transition map $M/p^{n+1}\to M/p^n$ is the natural surjection, and
\begin{equation}\label{M-lem-2}
\varprojlim_n M[p^n]=0
\end{equation}
where the transition map $M[p^{n+1}]\to M[p^n]$ is multiplication by $p$.
The isomorphism \eqref{M-lem-1} 
is well-known (\cite[Prop. 10.13]{atiyah-mac}).
We show \eqref{M-lem-2}.
Let
\[
0\lra M_1\lra M_2\lra M_3\lra0
\]
be any exact sequence of finitely generated $A$-modules. Then
\[
0\to M_1[p^n]\to M_2[p^n]\to M_3[p^n]\to M_1/p^n\to M_2/p^n\to M_3/p^n\to0.
\]
Suppose that $M_2[p^n]=0$ for all $n$. Then, by taking the projective limit, we have
\[
\xymatrix{
0\ar[r]&\varprojlim_n M_3[p^n]\ar[r]&
\varprojlim_n(M_1/p^n)\ar[r]& \varprojlim_n(M_2/p^n)\\
&&A^\wedge\ot_AM_1\ar[r]\ar[u]^\cong&A^\wedge\ot_AM_2.\ar[u]_\cong}
\]
Since $A\to A^\wedge$ is flat (\cite[Prop. 10.14]{atiyah-mac}), the bottom arrow is injective and hence the vanishing
$\varprojlim_n M_3[p^n]=0$ follows.
For the proof of \eqref{M-lem-2}, apply this to 
$M_3=M$ and $M_2$ a free $B$-module of 
finite rank.

\end{pf}

\subsection{Syntomic symbol maps}\label{syn-sym-sect}
We review syntomic symbol maps (\cite[p.205]{FM}, 
\cite[Chapter I \S3]{Ka2}, \cite[\S 2.2]{Ts1}, \cite[p.542]{T}).
Let $(Y,M)$ be a fine log scheme which is flat over $\bZ_p$.
We write $(Y_n,M_n)=(Y,M)\ot\Z/p^n\Z$ as before.
Suppose that $\{(Y_n,M_n)\}_n\to\Spec \Z/p^n\Z$ satisfies 
Condition \ref{logsyn-cond2} where $\Spec \Z/p^n\Z$ is endowed with the trivial
log structure, the identity as the Frobenius and the canonical DP-structure on 
$p\Z/p^n\Z$.
For a sheaf $M$ of monoid, we denote by $M^\gp$ the associated sheaf of abelian group.

\medskip

For $0\leq r<p$, there is the natural map
\begin{equation}\label{syn-sym-eq1}
\vg(Y_n,M^\gp_{n+1})\lra H^1_\syn((Y,M),\Z/p^n(1))
\end{equation}
where we omit to write ``$/(\Spec\Z/p^n\Z,(\Z/p^n\Z)^\times,\id)$" in the notation of
the syntomic cohomology or syntomic complexes in below. 
This is defined in the following way.
We define the complex $C_n$ as
\[ C_n :=\big(1+J_{D_n} \lra M_{D_n}^\gp \big) \qquad \hbox{($1+J_{D_n}$ is placed in degree $0$)}, \]
where $M_{D_n}^\gp$ denotes the associated sheaf of abelian groups.
We define the map of complexes
\[ s : C_{n+1}  \lra  \cS_n(1)_{(Y,M),(Z^\star,N^\star)} \]
as the map
\[  s^0 : 1+J^{[1]}_{D_{n+1}} \lra J^{[1]}_{D_n}, \quad a \longmapsto \log(a) \]
in degree $0$, and the map
\[ \begin{CD}
 s^1 : M_{D_{n+1}}^\gp @. \ \lra \ @. \big(\cO_{D_n} \otimes_{\cO_{Z_n}}\omega^1_{Z_n}\big) \oplus\cO_{D_n}  \\
 \qquad b @.  \longmapsto @. \big(\dlog(b), p^{-1}\log(b^p\varphi_{n+1}(b)^{-1})\big)
\end{CD} \]
in degree $1$.
Here $\varphi_{n+1}(b)b^{-p}$ belongs to $1+p\cO_{D_{n+1}}$ and
 the logarithm $p^{-1}\log(b^p\varphi_{n+1}(b)^{-1})\in \O_{D_n}$ is
well-defined.
One easily verifies that the maps $s^0$ and $s^1$ yield a map of complexes.
Since there is a natural quasi-isomorphism $C_{n+1} \cong M_{n+1}^\gp[1]$,
 the map $s$ induces a morphism (\cite[(2.2.3)]{Ts1})
\[   M_{n+1}^\gp[1] \lra \cS_n(1)_{(Y,M),(Z^\star,N^\star)} \]
 in the derived category, and hence \eqref{syn-sym-eq1} is obtained.

Now suppose that $M$ is defined by a divisor $D\subset Y$.
Let $U:=Y\setminus D$.
Then we have $M^\gp=j_*\cO_U^\times$, and obtain a map
\stepcounter{equation}
\begin{equation}\label{eq1-2-3} \cO(U_{n+1})^\times \lra H^1_{\syn}((Y,M),\Z/p^n(1)). \end{equation}
This map and the product structure of syntomic cohomology give rise to a map
\begin{equation}\label{eq1-2-4}  
[-]_\syn:K^M_r(\cO(U_{n+1})) \lra H^r_{\syn}((Y,M),\Z/p^n(r)) \end{equation}
for $0 \le r \le p-1$ (cf.\ \cite[Proposition 3.2]{Ka1}), which we call {\it the syntomic symbol map}.

\subsection{Comparison of Symbol maps for $U/S$ with Syntomic symbol maps}
Let $W=W(k)$ be the Witt ring of a perfect field $k$ of characteristic $p>0$,
and $K$ the fractional field. 
Let $F_W$ be the $p$-th Frobenius on $W$ or $K$.
We omit to write ``$/(\Spec W,W^\times,F_W)$'' in the notation of syntomic complexes or syntomic cohomology, 
as long as there is no fear of confusion, e.g.
\[
H^i_\syn((Y,M),\Z/p^n(r))=H^i_\syn((Y,M)/(\Spec W,W^\times,F_W),\Z/p^n(r)),\text{ etc.}
\]
For a flat (log) $W$-scheme $X$, we write $X_K:=X\ot_WK$ and $X_n:=X\ot\Z/p^n\Z$ as before.

\medskip

\def\pS{Q}
Let $\pS$ be a smooth affine scheme over $W$, and $T$ 
a reduced relative simple NCD over $W$.
Let $Y$ be a smooth scheme over $W$ and let
\[\xymatrix{f:Y\ar[r]& \pS}\] be a projective morphism that
is smooth outside $f^{-1}(T)$.
Let
$D\subset Y$ be a reduced relative simple NCD over $W$.
Put $S:=\pS\setminus T$ and $U:=Y\setminus (D\cup f^{-1}(T))$.
Let $L$ be the log structure on $\pS$ defined by $T$, and $M$ the log structure
on $Y$ defined by the reduced part of $D\cup f^{-1}(T)$.
We then suppose that the following conditions hold.
\begin{itemize}
\item $U$ and $S$ are affine.
\item 
The divisor
$D\cup f^{-1}(T)$ is a relative simple NCD over $W$
and 
$D\cap f^{-1}(S)$ 
is a relative simple NCD over $S$.
\item
$f:(Y,M)\to (\pS,L)$ is smooth.
\item
There is a system $\sigma=\{\sigma_m:(\pS_m,L_m)\to(\pS_m,L_m)\}_m$
of $p$-th Frobenius endomorphisms compatible with $F_W$.
\item
Assumption \ref{as:isomorphism} holds
for $U/S$, namely the $i$-th relative rigid cohomology sheaf 
$R^if_{\rig}j^{\dag}_U\sO_{U_K^{\an}}$
is a coherent $j^{\dag}_{S}\sO_{S_K^{\an}}$-module
for each $i\geq0$ (this implies the comparison isomorphism
\eqref{as:isom} by Remark \ref{rem:onassumption} (2)). 
\end{itemize}

\begin{thm}\label{ext-logsyn-thm}
Under the above setting, 
let $\omega^\bullet_{Y/\pS}$ (resp. $\omega^\bullet_{Y_m/\pS_m}$) 
denote the de Rham complex of $(Y,M)/(\pS,L)$ (resp. $(Y_m,M_m)/(\pS_m,L_m)$).
Let $0\leq n<p-1$ be an integer which satisfies
$H^n(\omega^{\geq n+1}_{Y_m/\pS_m})=H^{n+1}(\omega^{\geq n+1}_{Y_m/\pS_m})=0$ for all $m>0$.
Then the following diagram is $(-1)^n$-commutative,
\begin{equation}\label{ext-logsyn-thm-eq1}
\xymatrix{
&K_{n+1}^M(\O(U))\ar[ld]_{[-]_{U/S}}\ar[rd]^{[-]_\syn}
\\
\Ext^1(\O_{S},H^n(U/S)(n+1))\ar[d]_{R_\sigma}
&&
H^{n+1}_\syn((Y,M),\Z_p(n+1))\ar[d]\\
\O(S)^\dag_K\ot_{\O(S_K)}H^n_\dR(U_K/S_K)\ar[d]_\cap&&
H^{n+1}_\syn((Y,M)/(\pS,L,\sigma),\Z_p(n+1))
\\
 \O(S)_K^\wedge\ot_{\O(S_K)}H^n_\dR(U_K/S_K)
&&\O(\pS)^\wedge\ot_{\O(\pS)}H^n_\zar(Y,\omega^\bullet_{Y/\pS}).
\ar[u]^\cong_i\ar[ll]
}
\end{equation}
Here
the extension group is taken in the category $\FilFMIC(S,\sigma)$.
See \eqref{motive-map} and \eqref{eq1-2-4} for the definitions of the symbol
maps $[-]_{U/S}$ and $[-]_\syn$, and see \eqref{phi.sigma} for
$R_\sigma$.
The isomorphism $i$ follows from Proposition \ref{logsyn-exact-prop} and
Lemma \ref{syn-cpx-hom-lem} under the assumption 
$H^{n+1}(\omega^{\geq n+1}_{Y_m/\pS_m})=0$.
\end{thm}

\medskip

To prove Theorem \ref{ext-logsyn-thm}, we prepare for three lemmas.
\begin{lem}\label{fixed-lem}
Let $V=\Spec C$ be a smooth affine $W$-scheme.
Let $\sigma=\{\sigma_n:C_n\to C_n\}_n$ be a compatible system of $p$-th Frobenius.
Let $V(R)=\Hom_W(\Spec R,C)$ denote the set of $R$-valued points for a $W$-algebra $R$.
Put $W_n:=W/p^nW$.
Then for any $a_1\in V(W_1)$,
there is a unique $W$-valued point 
$a=(a_n)\in V(W)=\varprojlim_n T(W_n)$ which makes the diagram
\[
\xymatrix{
\Spec W_n\ar[r]^-{a_n}\ar[d]_{F_W}&V_n\ar[d]^{\sigma_n}\\
\Spec W_n\ar[r]^-{a_n}&V_n.
}
\]
commutative for all $n$.
\end{lem}
\begin{pf}
For $a_n\in V(W_n)$, we define $\phi(a_n)$ to be the morphism
which makes the following diagram commutative
\[
\xymatrix{
\Spec W_n\ar[r]^-{a_n}\ar[d]_{F_W}^\cong&V_n\ar[d]^{\sigma_n}\\
\Spec W_n\ar[r]^-{\phi(a_n)}&V_n.
}
\]
One easily verifies that $\phi(a_n)$ is a $W$-morphism (i.e. $\phi(a_n)\in V(W_n)$),
and commutes with
the reduction map $\rho_n:V(W_n)\to V(W_{n-1})$.
Put
\[
V(W_n)_{b_{n-1}}:=\{a_n\in V(W_n)\mid \rho_n(a_n)=b_{n-1}\}
\]
for $b_{n-1}\in V(W_{n-1})$. 
This is a non-empty set
as the reduction map $\rho_n$ is surjective (formal smoothness property).
We claim that a map \[
\phi:V(W_n)_{b_{n-1}}\to V(W_n)_{\phi(b_{n-1})}
\]
is a constant map. Indeed, let $C_n=W_n[t_1,\ldots,t_m]/I$.
One can write $\sigma_n(t_i)=t^p_i+pf_i(t)$. Let $a_n$ be given by $t_i\mapsto \alpha_i$.
Then $\phi(a_n)$ is given by the ring homomorphism
\[
t_i\longmapsto F_W^{-1}(\alpha_i^p+pf(\alpha_1,\ldots,\alpha_m)) \mod p^nW.
\]
which depends only on $\{\alpha_i$ mod $p^{n-1}W\}_i$, namely on $\rho_n(a_n)=b_{n-1}$.
This means that $\phi(a_n)$ is constant on the set $V(W_n)_{b_{n-1}}$.

We prove Lemma \ref{fixed-lem}. We have shown that if $b_{n-1}\in V(W_{n-1})$ satisfies
$\phi(b_{n-1})=b_{n-1}$, then $\phi^m(a_n)=\phi(a_n)$ for all $a_n\in V(W_n)_{b_{n-1}}$ and
$m\geq 1$.
Let $a_1\in V(W_1)$ be an arbitrary element, and take an arbitrary sequence
\[
(a_1,c_2,\ldots,c_n,\ldots)\in \varprojlim_n V(W_n)=V(W).
\]
Since $\phi$ is the identity on $V(W_1)$, one has $\phi^2(c_2)=\phi(c_2)$ as 
$c_2\in V(W_2)_{a_1}$. Then $\phi^3(c_3)=\phi^2(c_3)$ as $\phi(c_3)\in V(W_3)_{\phi(c_2)}$.
Continuing this, $a_n:=\phi^{n-1}(c_n)$ satisfies that $\phi(a_n)=a_n$ in $V(W_n)$ and 
$\rho_n(a_n)=\phi^{n-1}(c_{n-1})=\phi^{n-2}(c_{n-1})=a_{n-1}\in V(W_{n-1})$.
Hence the sequence $\{a_n\}_n$ defines
a $W$-valued point $a=(a_n)\in V(W)$
which makes a diagram
\[
\xymatrix{
\Spec W_n\ar[r]^-{a_n}\ar[d]_{F_W}&V_n\ar[d]^{\sigma_n}\\
\Spec W_n\ar[r]^-{a_n}&V_n.
}
\]
commutative for all $n\geq 1$.
We show that such $a$ is unique.
Suppose that another $W$-valued point 
$a'=(a_1,a'_2,\ldots,a'_n,\ldots)$ makes the above diagram
commutative. This means $\phi(a'_n)
=a'_n$ for all $n\geq 1$. Since
$\phi$ is a constant map on $V(W_n)_{a'_{n-1}}$,
the point
$a'_n=\phi(a'_n)$ is uniquely determined by $a'_{n-1}$, and hence by $a_1$.
Therefore
$a'$ is uniquely determined by $a_1$, which means $a=a'$.
\end{pf}
\begin{lem}\label{fixed-lem2}
Suppose that $k=W/pW$ is an algebraically closed field.
Let $V=\Spec C$ be a smooth affine scheme over $W$.
Let $\wh C$ be the $p$-adic completion, and $\wh C_K:=\wh C\ot_WK$.
Let $H_K$ be a locally free $\wh C_K$-module of finite rank.
Let $\{P_a\}_{a\in V(k)}$ be a set of $W$-valued points of $V$ such that
each $P_a\in V(W)$ is a lifting of $a\in V(k)$.
For an element $x\in H_K$, let $x|_{P_a}
\in \kappa(P_a)\ot_{C_K}H_K$ denote the reduction at $P_a$ where 
$\kappa(P_a)\cong K$ is the residue field.
If $x$ satisfies $x|_{P_a}=0$ for all $a\in V(k)$, then $x=0$.
\end{lem}
\begin{pf}
Since there is an inclusion $H_K\hra E_K$ into a free $\wh C_K$-module of finite rank
which has a splitting, we may replace $H_K$ with $E_K$ and hence
we may assume that $H_K$ is a free $\wh C_K$-module.
Then one can reduce the proof to the case of rank one, i.e. $H_K=\wh C_K$.
Let $x\in \wh C_K$ satisfy $x|_{P_a}=0$ for all $P_a$.
Suppose that $x\ne0$. 
There is an integer $n$ such that $x\in p^n\wh C\setminus p^{n+1}\wh C$.
Write $x=p^n y$ with $y\in \wh C\setminus p\wh C$ that satisfies
$y|_{P_a}=0$ for all $a\in V(k)$.
Then $y$ is zero in $\wh C/p\wh C\cong C/pC$
by Hilbert nullstellensatz, namely $y\in p\wh C$.
This is a contradiction.
\end{pf}

\begin{lem}\label{syn-rig-lem}
Let $Y$ be a projective smooth scheme over $W$, $D_Y$ a relative NCD in $Y$ over $W$.
Let $M$ be the log structure
on $Y$ defined by $D_Y$, and
put $U=Y\setminus D_Y$.
Then there is a canonical isomorphism
\begin{equation}\label{syn-rig-lem-eq1}
c\colon H^i_\syn((Y,M),\Q_p(j))\longrightarrow H^i_{\rigsyn}(U,\Q_p(j)).
\end{equation}
Moreover, a diagram
\begin{equation}\label{syn-rig-lem-eq2}
\xymatrix{
&K_i^M(\O(U))\ar[rd]^{[-]_\syn}
\ar[d]_{\reg_{\rigsyn}}\\
&H^i_{\rigsyn}(U,\Q_p(i))&
H^i_\syn((Y,M),\Q_p(i))\ar[l]_-c^-\cong
}
\end{equation}
is commutative.
\end{lem}
\begin{pf}
Write $U_K=U\times_WK$ and $(Y_K,M_K)=(Y,M)\times_W K$.
There is a canonical isomorphism 
\[
R\vg_{\text{log-syn}}(U,\Q_p(j))\cong
\text{Cone}[\Fil^jR\vg_{\text{log-dR}}((Y_K,M_K))\xrightarrow{1-p^{-j}\phi_\crys}
R\vg_{\text{log-crys}}((Y_1,M_1)/K)][-1]
\]
arising from \eqref{logsyn-prop1-eq1} (see also Proposition \ref{logsyn-exact-prop})
where $\phi_\crys$ is the $p$-th Frobenius which is defined thanks to the comparison of 
log de Rham and log crystalline cohomology.
There is also a canonical isomorphism (\cite[Remark 8.7, 3]{Be1})
\begin{equation}\label{rigsyn-cone}
R\vg_{\rigsyn}(U,\Q_p(j))\cong
\text{Cone}[\Fil^jR\vg_\dR(U_K/K)\xrightarrow{1-p^{-j}\phi_\rig}R\vg_\rig(U_1/k)][-1]
\end{equation}
where $\phi_\rig$ is the $p$-th Frobenius which is defined thanks to the comparison of the algebraic de Rham cohomology and the rigid cohomology.
Now, the isomorphism $c$ is induced by the comparison morphism between the rigid and the log crystalline cohomology \cite{Shiho}
and the one between the de Rham cohomology and the log de Rham cohomology.
\medskip

By the construction, $c$ is compatible with respect to the cup-product and
the period map to the etale cohomology,
\[
\xymatrix{
H_\syn^i(U,\Q_p(i))\ar[rr]^c\ar[rd]_{\rho_\syn}&&H_{\rigsyn}(U,\Q_p(i))
\ar[ld]^{\rho_{\rigsyn}}\\
&H^i_\et(U_K,\Q_p(i))
}
\]
where $\rho_\syn$ is as in \cite[\S 3.1]{Ts1}, and
$\rho_{\rigsyn}$ is as in \cite[Corollary 9.10]{Be1}.
Moreover both of $\rho_\syn\circ[-]_\syn$ and $\rho_{\rigsyn}\circ
\reg_{\rigsyn}$ agree with the etale symbol map 
(\cite[Proposition 3.2.4]{Ts1}, \cite[Corollary 9.10]{Be1}).
We show the commutativity of the diagram \eqref{syn-rig-lem-eq2}.
Thanks to the compatibility with respect to the cup-product,
one can reduce the assertion to the case $i=1$, namely
it is enough to show that
for $f\in \O(U)^\times$ an element $u:=c[f]_\syn-\reg_\rigsyn(f)$ is zero.
There is an exact sequence
\[
\xymatrix{
0\ar[r]&H^0_\dR(U_K/K)\ar[r]&H^1_\rigsyn(U,\Q_p(1))\ar[r]^-h& \Fil^1H^1_\dR(U_K/K)
}
\]
from \eqref{rigsyn-cone}.
By the construction of $c$ and the definition of $[-]_\syn$ in \eqref{eq1-2-4}
and \cite[Def.6.5 and Prop.10.3]{Be1}, one has $h(u)=df/f-df/f=0$.
Therefore $u$ lies in the image of
$H^0_\dR(U_K/K)$.
We claim that the composition of the maps
\[\xymatrix{
H^0_\dR(U_K/K)\ar[r]& H^1_\rigsyn(U,\Q_p(1))\ar[r]^-{\rho_\rigsyn}&
H^1_\et(U_K,\Q_p(1))
}
\]
is injective. We may assume that $U$ is connected.
Moreover we may replace $W$ with $W(\ol k)$, so that we may further assume that 
$U$ has a $W$-valued point.
Then $H^0_\dR(U_K/K)\cong K$ and the injectivity can be reduced to the 
case $U=\Spec W$, which can be easily verified.
We turn to the proof of $u=0$.
It is enough to show $\rho_\rigsyn(u)=0$.
However 
\begin{align*}
\rho_\rigsyn(u)&=\rho_\rigsyn(\reg_\rigsyn(f))-\rho_\rigsyn(c[f]_\syn)\\
&=\rho_\rigsyn(\reg_\rigsyn(f))-\rho_\syn([f]_\syn)=[f]_\et-[f]_\et=0
\end{align*}
where $[-]_\et$ is the etale symbol map, so we are done.
\end{pf}

\medskip

\noindent({\it Proof of Theorem \ref{ext-logsyn-thm}}).
By replacing $W$ with $W(\ol k)$, we may assume that $k$ is an algebraically closed field.
We fix a $p$-th Frobenius $\varphi$ on $\O(U)^\dag$ compatible with $\sigma$.
Recall the diagram \eqref{ext-logsyn-thm-eq1}.
Let $\xi\in K_{n+1}^M(\O(U))$, and let
\[
\langle\xi\rangle_{U/S},\quad
\langle\xi\rangle_\syn
\in \O(S)_K^\wedge\ot_{\O(S_K)}H^n_\dR(U_K/S_K)
\]
be the elements sent along the diagram in counter-clockwise direction,
clockwise direction respectively.
We want to show $\langle\xi\rangle_{U/S}=(-1)^n\langle\xi\rangle_\syn$.
For a closed point $a\in S_1(k)$, we take the unique lifting 
$P_a\in S(W)$ as in Lemma \ref{fixed-lem}.
By Lemma \ref{fixed-lem2}, it is enough to show
\[
\langle\xi\rangle_{U/S}|_{P_a}=
\langle\xi\rangle_\syn|_{P_a}\in \kappa(P_a)\ot_{B_K}H^n_\dR(U_K/S_K^*)
\]
for every $a$ where $\kappa(P_a)\cong K$ denotes the residue field of $P_a$.
Therefore,
to show the $(-1)^n$-commutativity of the diagram \eqref{ext-logsyn-thm-eq1},
we may specialize the diagram at $P_a$, so that the proof is reduced to the case 
$(\pS,L,\sigma)=(\Spec W,W^\times,F_W)$.
Summing up the above, the proof of Theorem \ref{ext-logsyn-thm} is reduced to showing
 $(-1)^n$-commutativity of the following diagram,
\begin{equation}\label{ext-logsyn-thm-eq2}
\xymatrix{
&K_{n+1}^M(\O(U))\ar[ld]_{[-]_{U/W}}\ar[rd]^{[-]\syn}
\\
\Ext^1(W,H^n(U/W)(n+1))\ar[d]^{\phi_{F_W}}_\cong
&&
H^{n+1}_\syn((Y,M),\Z_p(n+1))\\
H^n_\dR(U_K/K)
&&H^n_\zar(Y,\Omega^\bullet_{Y/W}(\log D)).
\ar[u]^\cong\ar[ll]
}
\end{equation}
Here $Y$ is a projective smooth scheme over $W$, $D$ a relative NCD over $W$,
$U=Y\setminus D$ and $M$ is the log structure
on $Y$ defined by $D$.
The extension group is taken in the category of $\FilFMIC(\Spec W,F_W)$.

To compare $[-]_{U/W}$ and $[-]_\syn$, we use the explicit formula 
for $[-]_{U/W}$
(Theorem \ref{comp-thm}) and the theory of rigid syntomic regulators by Besser \cite{Be1}.
Consider a diagram
\begin{equation}\label{ext-rigsyn-diagram}
\xymatrix{
&K_{n+1}^M(\O(U))\ar@/_20pt/[ldd]_{\reg_{U/W}}\ar[rd]^{[-]_\syn}
\ar[d]^{\reg_{\text{rig-syn}}}\\
&H^{n+1}_{\text{rig-syn}}(U,\Q_p(n+1))&
H^{n+1}_\syn((Y,M),\Z_p(n+1))\ar[l]_-c^-\cong\\
H^n_\dR(U_K/K)\ar[r]^{(-1)^n}
&H^n_\dR(U_K/K)\ar[u]^\cong_i
&H^n_\zar(Y,\Omega^\bullet_{Y/W}(\log D))
\ar[u]^\cong\ar[l]
}
\end{equation}
where $\reg_{U/W}:=\phi_{F_W}\circ [-]_{U/W}$.
Here the isomorphism $i$ follows from \cite[(8.5)]{Be1} and $c$ is the canonical isomorphism
\eqref{syn-rig-lem-eq1} 
in Lemma \ref{syn-rig-lem}.
The commutativity of the right upper triangle is proven in Lemma \ref{syn-rig-lem},
and that of the right lower square is immediate from the construction of $c$.
We show the commutativity of the left.
Let $\xi=\{h_0,\ldots,h_n\}\in K^M_{n+1}(\O(U))$.
Then using \cite[Def.6.5 and Prop.10.3]{Be1}
one can show that
\begin{equation}
\reg_{\text{rig-syn}}(\xi)=
\sum_{i=0}^n(-1)^ip^{-1}\log\left(\frac{h_i^p}{h_i^{\varphi}}\right)
\left(\frac{dh_0}{h_0}\right)^{\varphi_1}\wedge\cdots\wedge\left(\frac{dh_{i-1}}{h_{i-1}}\right)^{\varphi_1}\wedge
\frac{dh_{i+1}}{h_{i+1}}\wedge\cdots\wedge\frac{dh_n}{h_n}.
\end{equation}
under the inclusion $H^n_\dR(U_K/K)\hra
\Omega^n_{A_K/K}/d\Omega^{n-1}_{A_K/K}\cong
\Omega^n_{A^\dag_K/K}/d\Omega^{n-1}_{A^\dag_K/K}$
in a similar way to the proof of \cite[Cor. 2.9]{Ka1}.
This agrees with $(-1)^n\reg_{U/W}(\xi)$ by Theorem \ref{comp-thm},
and hence the commutativity of the left square follows.
This completes the proof of $(-1)^n$-commutativity of 
the diagram \eqref{ext-logsyn-thm-eq2}, and hence Theorem \ref{ext-logsyn-thm}.

\section{$p$-adic regulators of $K_2$ of curves}
For a regular scheme $X$ and a divisor $D$, 
let $(X,D)$ denote the log scheme whose 
log structure is defined by the reduced part of $D$.

\subsection{Symbol map on $K_2$ of a projective smooth family of curves}
\label{curvesymbol.sect}
Let $p>2$ be a prime and 
let $W$ be the Witt ring of a prefect field $k$ of characteristic $p$.
Put $K=\mathrm{Frac}(W)$ the fractional field. 
Let $F_W$ be the $p$-th Frobenius on $W$.
Let $\pS$ be a smooth affine curve over $W$, and $T\subset \pS$ 
a closed set that is finite etale over $W$.
Put $S:=\pS\setminus T$.
Let $Y$ be a smooth quasi-projective surface
over $W$, and let
\[
f:Y\lra \pS
\]
be a projective surjective $W$-morphism such that $f$ is smooth outside
$F:=f^{-1}(T)$, and the fibers are connected. 
Let $D\subset Y$ be a reduced relative simple NCD over $W$.
Put $X:=f^{-1}(S)=Y\setminus F$, $D_X:=D\cap X$ and $U:=X\setminus D_X$.
Suppose that the following conditions hold.
\begin{itemize}
\item[i)]
There is a system $\sigma=\{\sigma_n:(\pS_n,T_n)\to(\pS_n,T_n)\}_n$ of $p$-th 
Frobenius endomorphisms compatible with $F_W$,
where $(X_n,M_n):=(X,M)\ot\Z/p^n\Z$ as before.
\item[ii)]
The divisor $D+F$ is a relative simple NCD over $W$, 
and $D_X$ is finite etale over $S$.
\item[iii)] The multiplicity of each component of $F$ is
prime to $p$.
\item[iv)]
Assumption \ref{as:isomorphism} holds
for $U/S$; the $i$-th relative rigid cohomology sheaf 
$R^if_{\rig}j^{\dag}_U\sO_{U_K^{\an}}$
is a coherent $j^{\dag}_{S}\sO_{S_K^{\an}}$-module
for each $i\geq0$ (this implies the comparison isomorphism
\eqref{as:isom} by Remark \ref{rem:onassumption} (2)). 
\end{itemize}

\begin{lem}\label{lem:v)}
Let $S$ be a smooth affine curve over $W$ and $f:U\to S$ a smooth $W$-morphism.
Suppose that there is a commutative square
\[
\begin{tikzcd}
    U \arrow[r, hook] \arrow[d, "f"'] & \ol X \arrow[d, "\ol f"] \\
    S \arrow[r, hook]  & \ol S 
\end{tikzcd}
\]
where $\hookrightarrow$ are open immersions, that satisfies the following.
\begin{itemize}
\item
$\ol S$ (resp. $\ol X$) is a smooth projective curve (resp. a smooth projective $W$-scheme)
over $W$.
\item
Put $X:=\ol f^{-1}(S)$. Then $X\to S$ 
is projective smooth.
\item
Put $T':=\ol S\setminus S$. Then $T'$ is finite etale over $W$,  
$\ol X\setminus U$ is a relative simple NCD over $W$, and $X\setminus U$ is a
relative simple NCD over $S$.
\item
The multiplicity of an arbitrary component of $\ol f^{-1}(T')$ is prime to $p$.
\end{itemize}
Then $R^if_{\rig}j^{\dag}_U\sO_{U_K^{\an}}$
is a coherent $j^{\dag}_{S}\sO_{S_K^{\an}}$-module
for each $i\geq0$.
\end{lem}
\begin{pf}
The conditions imply that the morphism \[\ol f:(\ol X,\ol X\setminus U)\lra (\ol S, T')\] 
of log schemes is smooth, so that
$U/S$ is adapted to Setting \ref{set:generaldimension}.
Therefore $R^if'_{\rig}j^{\dag}_{U'}\sO_{(U'_K)^{\an}}$
is a coherent $j^{\dag}_{S'}\sO_{(S'_K)^{\an}}$-module
for each $i\geq0$
by Proposition \ref{prop:generaldimension}.
\end{pf}
In \S \ref{curvesymbol.sect}, we work in the above setting.
Here is the summary of notation.
\[
\begin{tikzcd}
 X \arrow[r, "F", hook] \arrow[d] & Y \arrow[d, "f"] 
 &\arrow[l, "D\cup F"',hook']U\arrow[ld]
    \\
    S \arrow[r, "T", hook]                                         & \pS     &                   
\end{tikzcd}
\]
where the notation above $\hookrightarrow$ shows the complement of the subscheme.

\medskip

We have the symbol map
\[
[-]_{U/S}:
K^M_2(\sO(U))
\lra\Ext^1_{\FilFMIC(S)}(\sO_{S}, H^1(U/S)(2))
\]
by Theorem \ref{mot-map-prop}.
We omit to write the subscript ``$\FilFMIC(S)$"
in the extension groups, as long as there is no fear of confusion. 
\begin{lem}\label{lem:tame}
    The following diagram is commutative,
    \[
 \xymatrix{K^M_2(\sO(U))
  \ar[r]^-{\partial} 
  \ar[d]_{[-]_{U/S}} 
  & \sO(D_X)^{\times}\ar[d]^{[-]_{D_X/S}} \\
            \Ext^1(\sO_{S}, H^1(U/S)(2)) \ar[r]^-{\Res} & 
            \Ext^1(\O_{S}, H^0(D_X/S)(1)).
}    \]
Here, the right vertical arrow is the symbol map
defined in Theorem \ref{mot-map-prop}, and
$\partial$ is the tame symbol which is defined by
    \begin{equation}\label{tame.symbol}
\partial:        \{f, g\}\longmapsto (-1)^{\ord_{D_X}(f)\ord_{D_X}(g)}\frac{f^{\ord_{D_X}}(g)}
{g^{\ord_{D_X}}(f)}\bigg|_{D_X}.
    \end{equation}
\end{lem}

\begin{pf}
We first note that $K_2^M(\O(U))$ is generated by
symbols $\{f,g\}$ with $\ord_{D_X}(f)=0$.
Suppose $\ord_{D_X}(f)=0$, namely $f\in \O(U\cup D_X)^\times$.
The natural morphism
\[p:\Cone\left[\sLog f(1)\to\sLog (g)\right]\to \sLog f(1)\]
and the residue map induce the following commutative diagram 
\begin{equation}
\xymatrix{           
 0 \ar[r] & H^1(U/S)(2) \ar[r] \ar@{=}[d] 
 & M_{f,g}(U/S) \ar[d]^{p} \ar[r] & \O_{S_K} \ar[d]^{p'} \ar[r]^{\delta_{f,g}\qquad} 
 & H^2(U/S)(2)\ar@{=}[d]\\
 \O_S \ar[r]^{\delta_f\qquad} & H^1(U/S)(2)\ar[d]^{\Res} 
            \ar[r] & H^1\left(\sLog f(1)\right) \ar[d]^{\Res} \ar[r] 
            & H^1(U/S)(1) \ar[d]^{\Res} \ar[r] & H^2(U/S)(2)\ar[d]^\Res\\
           0 \ar[r] & H^0(D_X/S)(1) \ar[r] 
           & H^0\left(\sLog f|D_X\right) \ar[r] & H^0(D_X/S) \ar[r]^{\delta_{f|D}} & H^1(D_X/S)(1)
}       \label{eq:comm_diagram}
\end{equation}
with exact rows where 
$\delta_{f,g}(1)=\frac{df}{f}\frac{dg}{g}$,
$\delta_f(1)=\frac{df}{f}$ and 
$\delta_{f|D_X}(1)=\frac{df}{f}|_{D_X}$.
    The maps ``Res''s are residue maps appearing in the Gysin exact sequence;
    as for the middle one, we are considering the Gysin exact sequence with coefficient
    in $\sLog(f)(1)$ which is constructed in the same manner as Proposition \ref{prop:Gysin}.
We show that $p'$ agrees with the map $\delta_g:1\mapsto\frac{dg}{g}$. 
Let $[A\to B]$ denote the complex 
\[\cdots\lra 0\lra A\lra B\lra 0\lra\cdots\]
with $A$ placed in degree zero.
Consider a commutative diagram
    \[
 \xymatrix{          \left[\sLog(f)_{\dR}\to\sLog(g)_{\dR}\right] \ar[r] \ar[d]^p & \left[\sLog(f)_{\dR}/\sO_{U_K}e_{-2,f}\to\sLog(g)_{\dR}\right] 
 \ar[d]^{p'} 
 \ar[r]^{\hspace{2cm}\cong} & \left[ 0\to \sO_{U_K}\right] \\
            \left[\sLog(f)_{\dR}\to 0\right] \ar[r] 
            & \left[\sLog(f)_{\dR}/\sO_{U_K}e_{-2,f}\to 0\right] 
            \ar[r]^{\hspace{1cm}\cong} & \left[ \sO_{U_K}\to 0\right],
 }   \]
of complexes in $\FilFMIC(U)$
    where the two horizontal morphisms on the left-hand side are the canonical surjections
and the vertical ones are the projections.
    Under the identification $\sLog(f)_{\dR}/\sO_{U_K}e_{-2,f}\cong\sO_{U_K}$ (which sends $e_{0,f}$ to $1$), the homomorphism $p'$
    is nothing but the connecting morphism arising from the extension $0\to\sO_{U_K}\to \sLog(g)_{\dR}\to \sO_{U_K}\to 0$. This means $p'=\delta_g$.

Since the composition $\Res\circ p'$ is multiplication by $\ord_D(g)$, 
the diagram \eqref{eq:comm_diagram}
induces
\begin{equation}\label{eq:comm_diagram2}
\xymatrix{           
 0 \ar[r] & H^1(U/S)(2) \ar[r] \ar[d]^\Res 
 & M_{f,g}(U/S) \ar[d] \ar[r] & \O_{S_K} \ar[d]^{\ord_{D_X}(g)} \ar[r]^{\delta_{f,g}\qquad} 
 & H^2(U/S)(2)\ar[d]^\Res\\
           0 \ar[r] & H^0(D_X/S)(1) \ar[r] 
           & H^0\left(\sLog f|D_X\right)' \ar[r] & \O_{S_K} \ar[r] & H^1(D_X/S)(1).
}\end{equation}
Now we show the lemma.
Let $\xi\in K^M_2(\O(U))\cap\Ker(\dlog)$ be arbitrary.
Fix $t\in \O(U)^\times$ such that $\ord_{D_X}(t)>0$. Replacing $\xi$ with $m\xi$ for some $m>0$, one can express 
\[
\xi=\{f,t\}+\sum_j\{u_j,v_j\}
\]
with $\ord_{D_X}(f)=\ord_{D_X}(u_j)=\ord_{D_X}(v_j)=0$.
Then \eqref{eq:comm_diagram2} induces a commutative diagram
\begin{equation}\label{eq:comm_diagram3}
\xymatrix{           
 0 \ar[r] & H^1(U/S)(2) \ar[r] \ar[d]^\Res 
 & M_\xi(U/S) \ar[d] \ar[r] & \O_{S_K} \ar[d]^{\ord_{D_X}(t)} \ar[r]
 & 0\\
           0 \ar[r] & H^0(D_X/S)(1) \ar[r] 
           & H^0\left(\sLog f|D_X\right)' \ar[r] & \O_{S_K} \ar[r] & 0.
}\end{equation}
Since the tame symbol of $\xi$ is $f|_{D_X}$, the assertion follows.
\end{pf}
\begin{prop}\label{prop:from_K2X}
Let $K_2^M(\O(U))_{\partial=0}$ be the kernel of the tame symbol
$\partial:K_2^M(\O(U))\to \O(D_X)^\times$.
Then the symbol map $[-]_{U/S}$ uniquely extends to
\[
[-]_{X/S}:K_2^M(\O(U))_{\partial=0}\lra 
\Ext^1
\left(\sO_{S}, H^1(X/S)(2)\right).
\]
\end{prop}
\begin{pf}
Let $N$ be the kernel of the connecting homomorphism $H^0(D_X/S)(-1)\to H^2(X/S)$
in the Gysin exact sequence.
Since $H^0_{\rig}(D_X/S)$ is an overconvergent $F$-isocrystal on $S$ of weight $0$,
the weight of $N_{\rig}$ is $-2$ and therefore
we have $\Hom_{\Fil\hyphen F\hyphen\MIC(S)}\big(\sO_{S}, N(2)\big)=0$.
This shows that the Gysin exact sequence induces an exact sequence
    \[
        0\longrightarrow \Ext^1\left(\sO_{S}, H^1(X/S)(2)\right)\longrightarrow
\Ext^1\left(\sO_{S}, H^1(U/S)(2)\right)\longrightarrow\Ext^1\big(\sO_{S}, N(2)\big)
    \]
 where the extension groups is taken in the category of $\Fil\hyphen F\hyphen\MIC(S)$.  
Now the construction of $[-]_{X/S}$ is immediate from Lemma \ref{lem:tame}.
\end{pf}
\begin{lem}\label{prop:synK2X}
Let
\[
H^\bullet_\syn((Y,F),\Z_p(r))
=\varprojlim_n H^\bullet_\syn((Y,F),\Z/p^n(r))
\]
denote the syntomic cohomology with coefficients in $\Z_p$
where we omit to write ``$/(W,W^\times,F_W)$".
Then the syntomic symbol map $[-]_\syn$ induces a map
\[
K_2^M(\O(U))_{\partial=0}\lra H^2_\syn((Y,F),\Z_p(2)),
\]
which we write by the same notation.
When $\pS=S=\Spec W$ ($F=\emptyset$), this is compatible with the regulator map by 
Besser \cite{Be1}
(or equivalently by Nekov\'a\v{r}-Niziol \cite{NN}), which means that a diagram
\[
\xymatrix{
K_2^M(\O(U))_{\partial=0}\ar[r]^-{[-]_\syn}\ar[d]
&H^2_\syn(Y,\Z_p(2))\ar[d]\\
K_2(Y)^{(2)}\ar[r]^-{\reg_\rigsyn}&H^2_\rigsyn(Y,\Q_p(2))
}
\]
is commutative where 
$K_i(-)^{(j)}\subset K_i(-)\ot\Q$ denotes the Adams weight piece.
\end{lem}
\begin{pf}
To show the former, it is enough to show that a diagram
\begin{equation}\label{prop:synK2X-eq1}
 \xymatrix{K^M_2(\sO(U))
  \ar[r]^-{\partial} 
  \ar[d]_{[-]_\syn} 
  & \sO(D^*)^{\times}\ar[d]^{[-]_\syn} \\
H^2_\syn((Y,D+F),\Z/p^n(2))\ar[r]^-{\Res}&H^1_\syn((D,D\cap F),\Z/p^n(1))
}
\end{equation}
is commutative. 
Then the required symbol map is induced from an exact sequence
\[
0\to
H^2_\syn((Y,F),\Z/p^n(2))\to
H^2_\syn((Y,D+F),\Z/p^n(2))\to
H^1_\syn((D,D\cap F),\Z/p^n(1)).
\]
To show the commutativity of \eqref{prop:synK2X-eq1},
it is enough to check $[\partial(\xi)]_\syn=\Res([\xi]_\syn)$ for an element $\xi=\{f,g\}\in K_2^M(\O(U))$ such that
$\ord_{D_X}(f)=0$. One has $[\partial(\xi)]_\syn=[(f|_{D_X})^{\ord_{D_X}(g)}]_\syn
=\ord_{D_X}(g)[f|_{D_X}]_\syn$.
On the other hand, one has
\[
\Res([f]_\syn\cup[g]_\syn)=[f]_\syn|_{D_X}\cup\Res'([g]_\syn)
=\Res'([g]_\syn)[f|_{D_X}]_\syn
\]
where 
\[
\Res':
H^1_\syn((Y,D+F),\Z/p^n(1))\to H^0_\syn((D,D\cap F),\Z/p^n(0))=\Z/p^n\Z.
\]
By the construction of $[-]_\syn$ in \S \ref{syn-sym-sect},
one can verifies $\Res'([g]_\syn)=\ord_{D_X}(g)$. Hence 
$[\partial(\xi)]_\syn=\Res([\xi]_\syn)$ as required.

When $\pS=S=\Spec W$, the latter assertion can be derived from
Lemma \ref{syn-rig-lem} as 
$H^2_\syn(Y,\Q_p(2))\cong H^1_\dR(Y_K/K)\to H^2_\syn((Y,D),\Q_p(2))
\cong H^1_\dR(U_K/K)$ is injective.
\end{pf}
\begin{thm}\label{ell.diagram-thm}
\[
\xymatrix{
&K_2^M(\O(U))_{\partial=0}\ar[ld]_{[-]_{X/S}}\ar[rd]^{[-]_\syn}
\\
\Ext^1(\O_{S},H^1(X/S)(2))\ar[d]_{R_\sigma}^{\eqref{phi.sigma}}
&&
H^2_\syn((Y,F),\Z_p(2))\ar[d]\\
\O(S)^\dag_K\ot_{\O(S)}H^1_\dR(X/S)\ar[d]_\cap&&
H^2_\syn((Y,F)/(\pS,T,\sigma),\Z_p(2))
\\
 \O(S)_K^\wedge\ot_{\O(S)}H^1_\dR(X/S)
&&\O(\pS)^\wedge\ot_{\O(\pS)}H^1_\zar(Y,\omega^\bullet_{Y/\pS})
\ar[u]_\cong\ar[ll]
}
\]
is $(-1)$-commutative where the extension group is taken in the category 
$\FilFMIC(S,\sigma)$ and
$\omega^\bullet_{Y/\pS}$ is the
log de Rham complex of $(Y,F)/(\pS,T)$.
\end{thm}
\begin{pf}
Noticing that $H^1_\dR(X_K/S_K)\to H^1_\dR(U_K/S_K)$ is injective,
one can derive the $(-1)$-commutativity
from Theorem \ref{ext-logsyn-thm}.
\end{pf}
\subsection{Syntomic regulators of $K_2$ of Elliptic Curves with $3$-torsion points}
Let 
\[
f_\Q:X_\Q\lra S_\Q=\Spec\Q[t,(t-t^2)^{-1}]
\]
be a family of elliptic curves given by a Weierstrass equation
\[
y^2=x^3+(3x+4(1-t))^2.
\]
This is the universal elliptic curve over the modular curve $X_1(3)\cong \P^1_\Q$
with 3-torsion points $(x,y)=(0,\pm4(1-t))$ and $x=\infty$.
The $j$-invariant of the generic fiber $X_t$  is
\[
j(X_t)=\frac{27(1+8t)^3}{t(1-t)^3}.
\]
The family $f_\Q$ extends to a fibration over $\P^1_\Q$ such that 
the Kodaira types of the fibers at $t=0,1,\infty$
are
I$_1$, I$_3$ and IV$^*$ respectively.

Let $p\geq 5$ and let $W=W(\ol\F_p)$ be the Witt ring of the algebraic closure $\ol\F_p$.
Let $K=\Frac W$ be the fractional field.
For $\alpha\in W\cup\{\infty\}$, 
let $P_\alpha$ denote the $W$-valued point of $\P^1_W(t)$ given by $t=\alpha$.
There is a projective flat morphism
\[
\bar f:\ol Y\lra \P^1_W(t)
\]
over $W$
arising from the fibration $f_\Q$ such that the following conditions hold.
\begin{itemize}
\item
$\ol Y$ is a smooth projective surface over $W$.
\item
The singular fiber $\bar f^{-1}(P_0)$ is a non-reduced curve with
two components $Z$ and $E$
such that the multiplicity of $Z$ (resp. $E$) is $1$ (resp. $2$).
The reduced part $\bar f^{-1}(P_0)_{\red}$ is a Neron $2$-gon. In particular, $Z\cup E$
is a relative simple NCD over $W$.
\item
The singular fiber $\bar f^{-1}(P_1)$ is the Neron $3$-gon over $W$.
\item
The singular fiber $\bar f^{-1}(P_\infty)$ is the curve of Kodaira type IV${}^*$ over $W$,
 in particular, a relative simple NCD over $W$.
\end{itemize}
Let 
$\pS:=\P^1_W\setminus\{P_1,P_\infty\}
\supset
S:=\P^1_W\setminus\{P_0,P_1,P_\infty\}$ and put $Y:=\bar f^{-1}(\pS)$,
$X:=\bar f^{-1}(S)$.
Let $c\in 1+pW$, and let 
$\sigma=\{\sigma_n:(\pS_n,P_{0,n})\to(\pS_n,P_{0,n})\}_n$ be the system of $p$-th 
Frobenius endomorphisms given by $\sigma_n(t)=ct^p$.
Let $\ol D\subset \ol Y$ be the closure of sections $(x,y)=(0,\pm 4(1-t))$ and the infinity section $x=\infty$.
It is not hard to 
see that each component of $\ol D$ is a section of $\bar f$ (and hence $\P^1_W$), 
and the three curves are disjoint.
Moreover they
intersect only with reduced and regualr locus of $\bar f^{-1}(P_i)$ for each
$i\in \{0,1,\infty\}$, and the intersections are transversal.
In particular $\ol D\cup \bigcup_i \bar f^{-1}(P_i)$ is a simple relative NCD over $W$.
Put $D:=Y\cap \ol D$, $D_X:=D\cap X$ and $U:=X\setminus D_X$.

\medskip

The above setting $(Y/\pS,S,D,\sigma)$ satisfies the conditions i), \ldots, iv) 
in the beginning of \S \ref{curvesymbol.sect} where 
iv) follows from Lemma \ref{lem:v)}.
Let $\wh B$ (resp. $B^\dag$) denote the $p$-adic completion (resp. the 
weak completion)
and write $\wh B_K:=\wh B\ot_WK$, $B^\dag_K:= B^\dag\ot_WK$ as before.

\medskip

We consider a Milnor symbol
\begin{equation}\label{xi}
\xi:=\left\{
\frac{y-3x-4(1-t)}{-8(1-t)},
\frac{y+3x+4(1-t)}{8(1-t)}
\right\}\in K_2^M(\O(U)).
\end{equation}
It is a simple exercise to show $\partial(\xi)=0$ where $\partial$ is the tame symbol
\eqref{tame.symbol}.
Hence we have a $1$-extension 
\[
[\xi]_{X/S}\in\Ext^1_{\FilFMIC(S)}
\left(\sO_S, H^1(X/S)(2)\right)
\]
by Proposition \ref{prop:from_K2X}.
It follows from Theorem \ref{comp-thm} that one has
\[
D_{X/S}(\xi)=-\dlog(\xi)=-3\frac{dt}{t-1}\ot\frac{dx}{y}\in \Omega^1_{B_K}\ot H^1_\dR(X_K/S_K).
\]
The purpose of this section is to describe
\[
\reg^{(\sigma)}_{X/S}(\xi)
\in  B_K^\dag\ot H^1_\dR(X_K/S_K)
\]
in terms of the hypergeometric functions
\[
{}_2F_1\left({a,b\atop1};t\right):=\sum_{n=0}^\infty\frac{(a)_n}{n!}\frac{(b)_n}{n!}t^n
\]
where $(a)_n:=a(a+1)\cdots(a+n-1)$ denotes the Pochhammer symbol.
Put
\begin{equation}
\omega:=\frac{dx}{y},\quad \eta:=\frac{xdx}{y}
\end{equation}
a $B_K$-basis of $H^1_\dR(X_K/S_K)$.
Let
\[
F(t)=\frac{1}{2\sqrt{-3}}{}_2F_1\left({\frac13,\frac23\atop1};t\right),
\] 
and put
\begin{equation}\label{hat-eq}
\wh\omega:=\frac{1}{F(t)}\frac{dx}{y},\quad \wh\eta:=
4(1-t)(F(t)+3tF'(t))\frac{dx}{y}+F(t)\frac{xdx}{y}.
\end{equation}
a $K((t))$-basis of $K((t))\ot H^1_\dR(X_K/S_K)$.
Let $\Delta:=\Spec W[[t]]\to \P^1_W(t)$ and $O:=\Spec W[[t]]/(t)\hra \Delta$.
The fibration $\cE:=f^{-1}(\Delta)$ over $\Delta$
is a Tate elliptic curve with central fiber
$f^{-1}(O)=F=Z+2E$.
There is the uniformization
\[
\wh{\mathbb G}_m:=\Spec\bigg( \varprojlim_n(W[t]/(t^n)[u,u^{-1}])\bigg)\lra \cE
\]
Let $q\in tW[[t]]$ be the Tate period of $\cE/\Delta$ which is characterized by
\[
j(X_t)=\frac{27(1+8t)^3}{t(1-t)^3}
=\frac{1}{q}+744+196884q+\cdots
\]
\[
\left(\Longrightarrow q=\frac{1}{27}t+\frac{250289}{243}t^2-
\frac{5507717}{243}t^3+ \frac{25287001}{81}t^4+\cdots\right).
\]
Let $\omega^\bullet_{\cE/\Delta}$ be the de Ram complex
of $(\cE,F)/(\Delta,O)$.
Thanks to the fundamental theorem on log crystalline cohomology by Kato,
we have the canonical isomorphism
\[
H^i_\zar(\cE,\omega^\bullet_{\cE/\Delta})\cong H^i_\crys((\cE_1,F_1)/(\Delta,O))
\]
(\cite[Theorem 6.4]{Ka3}, see also 
Proposition \ref{logsyn-prop1})
where $(\cE_n,Z_n)=(\cE,Z)\ot\Z/p^n\Z$ as before.
\begin{prop}\label{ell-prop1}
Let $\nabla:K((t))\ot H^1_\dR(X_K/S_K)\to K((t))dt\ot H^1_\dR(X_K/S_K)$ be
the Gauss-Manin connection.
Then
\begin{equation}\label{ell-prop1-eq1}
\nabla(\omega)=-\frac{dt}{3t}\omega+\frac{dt}{12(t^2-t)}\eta,\quad 
\nabla(\eta)=\frac{4dt}{3t}\omega+\frac{dt}{3t}\eta,
\end{equation}
\begin{equation}\label{ell-prop1-eq2}
\nabla(\wh\omega)
=\frac{dq}{q}\ot\wh\eta,\quad \nabla(\wh\eta)=0.
\end{equation}
Let $L:=\Frac(W[[t]])$.
Then $\{\frac12\wh\omega,\frac12\wh\eta\}$ forms the de Rham symplectic basis
of $L\ot_{W[[t]]}
H^1_\zar(\cE,\omega^\bullet_{\cE/\Delta})$ in the sense of \S \ref{gm-mt-sect}.
Moreover $\wh\omega$ and $\wh\eta$ lie in $H^1_\zar(\cE,\omega^\bullet_{\cE/\Delta})$.
\end{prop}
\begin{pf}
The computation of the Gauss-Manin connection for an elliptic fibration is well-known,
and hence we have
\eqref{ell-prop1-eq1}.
We show that $\{\frac12\wh\omega,\frac12\wh\eta\}$ forms the de Rham symplectic basis.
Then \eqref{ell-prop1-eq2} follows from Proposition \ref{gm-mt-sect-prop1}.
Firstly, thanks to \eqref{ell-prop1-eq1}, one can directly show
\begin{equation}\label{ell-prop1-eq3}
\nabla(\wh\omega)=\frac{dt}{12(t^2-t)F(t)^2}\ot\wh\eta,
\quad \nabla(\wh\eta)=0
\end{equation}
where we note that the hypergeometric function $F(t)$ satisfies
$t(1-t)F^{\prime\prime}(t)+(1-2t)F'(t)-\frac29 F(t)=0$ (e.g. \cite[16.8.5]{NIST}).
Let $\{\wh\omega_\dR,\wh\eta_\dR\}$ be the de Rham symplectic basis.
Then 
\begin{equation}\label{ell-prop1-eq4}
\nabla(\wh\omega_\dR)=\frac{dq}{q}\ot\wh\eta_\dR,
\quad \nabla(\wh\eta_\dR)=0
\end{equation}
by Proposition \ref{gm-mt-sect-prop1}.
Since $\ker(\nabla)\cong K$,
we have $\wh\eta=a\wh\eta_\dR$ for some $a\in K^\times$.
Since $\wh\omega$ and $\wh\omega_\dR$ are regular $1$-forms, 
there is a $h\in L^\times$ such that $\wh\omega=h\wh\omega_\dR$.
We then have
\[
\nabla(\wh\omega)=\nabla(h\wh\omega_\dR)=dh\ot\wh\omega_\dR
+h\nabla(\wh\omega_\dR)
=dh\ot\wh\omega_\dR+h\frac{dq}{q}\ot\wh\eta_\dR.
\]
By \eqref{ell-prop1-eq3}, $h=b$ is a constant.
Let $\bG_m=Z\setminus(Z\cap E)$ be the reduced regular locus of $F$ 
and $i:\bG_m\hra \cE$ the inclusion.
Put $u_0=y/(x+4)$ and $u=(u_0-\sqrt{-3})/(u_0+\sqrt{-3})$. Then 
$\bG_m=\Spec W[u,u^{-1}]$ and
\[
i^*\wh\omega=2\sqrt{-3}i^*\left(\frac{dx}{y}\right)=2\frac{du}{u}=2i^*\wh\omega_\dR.
\]
This shows $\wh\omega=2\wh\omega_\dR$.
By \eqref{ell-prop1-eq3} and \eqref{ell-prop1-eq4}, we have
\[
\frac{dt}{12(t^2-t)F(t)^2}\ot\wh\eta
=2\frac{dq}{q}\ot\wh\eta_\dR=2a^{-1}\frac{dq}{q}\ot\wh\eta.
\]
Take the residue at $t=0$, and then we see $a=2$.

There remains to show
$\wh\omega,\wh\eta\in H^1_\zar(\cE,\omega^\bullet_{\cE/\Delta})$.
It is straightforward to see that $dx/y
\in 
\vg(\cE,\omega^\bullet_{\cE/\Delta})$, so that
one has $\wh\omega\in 
\vg(\cE,\omega^\bullet_{\cE/\Delta})$.
The Gauss-Manin connection induces a connection
\[
\nabla:H^1_\zar(\cE,\omega^\bullet_{\cE/\Delta})\lra\frac{dt}{t}\ot
H^1_\zar(\cE,\omega^\bullet_{\cE/\Delta}),
\]
and one has from \eqref{ell-prop1-eq2}
\[
\wh\eta=q\frac{d}{dq}\wh\omega=u(t)\cdot t\frac{d}{dt}\wh\omega,
\quad u(t):=\frac{q}{t}\frac{dt}{dq}\in W[[t]]^\times
\]
(note that $W[[q]]=W[[t]]$).
This shows $\wh\eta\in H^1_\zar(\cE,\omega^\bullet_{\cE/\Delta})$.
This completes the proof.
\end{pf}
\begin{prop}\label{ell-prop2}
Let $c\in 1+pW$ and 
let $\sigma$ be the $p$-th Frobenius on $\wh B_K$ defined by $\sigma(t)=ct^p$.
Let $\Phi$ be the $\sigma$-linear Frobenius on $H^1_\zar(\cE,\omega^\bullet_{\cE/\Delta})$.
Put
\[
\tau^{(\sigma)}(t)=p^{-1}\log\left(\frac{q^p}{q^\sigma}\right)
=-p^{-1}\log (27^{p-1}c)+p^{-1}\log(q_0^p/q_0^\sigma),\quad \left(q_0:=\frac{27q}t\right)
\]
where $\log:1+pW[[t]]\to W[[t]]$ is defined by the customary Taylor expansion.
Then
\[
\begin{pmatrix}
\Phi(\wh\omega)&\Phi(\wh\eta)
\end{pmatrix}
=\begin{pmatrix}
\wh\omega&\wh\eta
\end{pmatrix}
\begin{pmatrix}
p&0\\
-p\tau^{(\sigma)}(t)&1
\end{pmatrix}
\]
\end{prop}
\begin{pf}
Note $W[[t]]=W[[q]]$. This is a special case of Theorem \ref{frob-sect-thm1} in Appendix.
\end{pf}

Let $\xi$ be the symbol \eqref{xi}.
Let $\sigma:B^\dag\to B^\dag$ be the $p$-th Frobenius defined by $\sigma(t)=ct^p$.
Let
\[
0\lra H^1_\dR(X/S)(2)\lra M_\xi(X/S)\lra\O_{S}\lra0
\]
be the 1-extension in $\FilFMIC(S,\sigma)$,
 associated to $[\xi]_{X/S}$. Let $e_\xi\in \Fil^0M_\xi(X_K/S_K)_\dR$
be the unique lifting of $1\in \O(S_K)$.
Define $\ve_i(t),E_i(t)$ by
\begin{align}
\reg^{(\sigma)}_{X/S}(\xi)=e_\xi-\Phi(e_\xi)&=\ve_1(t)\omega+\ve_2(t)\eta
\label{e-Phi-eq1}\\
&=E_1(t)\wh\omega+E_2(t)\wh\eta.\label{e-Phi-eq2}
\end{align}
It follows that $\ve_i(t)\in B^\dag_K$. 
By \eqref{hat-eq}, one immediately has
\begin{equation}\label{ell-thm-eq3}
\ve_1(t)=\frac{E_1(t)}{F(t)}+4(1-t)(F(t)+3tF'(t))E_2(t),
\end{equation}
\begin{equation}\label{ell-thm-eq4}
\ve_2(t)=F(t)E_2(t).
\end{equation}
The power series $E_i(t)$ are explicitly described as follows.
\begin{thm}\label{ell-thm}
Let $\xi$ be as in \eqref{xi}. Put $\nu=(-1+\sqrt{-3})/2$.
Then $E_1(t),E_2(t)\in K[[t]]$ and they are characterized by 
\begin{equation}\label{ell-thm-eq1}
\frac{d}{dt}E_1(t)=-3\left(F(t)\frac{dt}{t-1}-F(t)^\sigma\frac{p^{-1}dt^\sigma}{t^\sigma-1}\right),\quad E_1(0)=0,
\end{equation}
\begin{equation}\label{ell-thm-eq2}
\frac{d}{dt}E_2(t)=-E_1(t)\frac{q'}{q}-3F(t)^\sigma
\tau^{(\sigma)}(t)\frac{p^{-1}dt^\sigma}{t^\sigma-1},\quad
E_2(0)=-9\ln^{(p)}_2(-\nu).
\end{equation}
\end{thm}
\begin{pf}
Apply $\nabla$ on \eqref{e-Phi-eq2}.
By Proposition \ref{ell-prop1} together with the fact that $\Phi\nabla=\nabla\Phi$, one has
\begin{equation}\label{ell-thm-pf-eq1}
\nabla(e_\xi)-\Phi(\nabla(e_\xi))=dE_1(t)\ot\wh\omega+
\left(E_1(t)\frac{dq}{q}+dE_2(t)\right)\ot\wh\eta.
\end{equation}
Since
\[
\nabla(e_\xi)=-\dlog(\xi)=-3\frac{dt}{t-1}\ot\frac{dx}{y}=-3F(t)\frac{dt}{t-1}\ot\wh\omega
\]
(cf. \eqref{comp-thm-pf-eq3}), the left hand side of \eqref{ell-thm-pf-eq1} is
\begin{align*}
&-3F(t)\frac{dt}{t-1}\ot\wh\omega+3p^{-1}\sigma\left(F(t)\frac{dt}{t-1}\right)
\ot p^{-1}\Phi(\wh\omega)\\
=&-3F(t)\frac{dt}{t-1}\ot\wh\omega+3F(t)^\sigma
\frac{p^{-1}dt^\sigma}{t^\sigma-1}\ot(\wh\omega-\tau^{(\sigma)}(t)\wh\eta)
\quad(\text{Proposition \ref{ell-prop2}})
\\
=&-3\left(F(t)\frac{dt}{t-1}-F(t)^\sigma\frac{p^{-1}dt^\sigma}{t^\sigma-1}\right)\ot\wh\omega
-3F(t)^\sigma\tau^{(\sigma)}(t)\frac{p^{-1}dt^\sigma}{t^\sigma-1}\ot\wh\eta.
\end{align*}
Therefore one has
\begin{equation}\label{ell-thm-pf-eq2}
\frac{d}{dt}E_1(t)=-3\left(F(t)\frac{dt}{t-1}-F(t)^\sigma\frac{p^{-1}dt^\sigma}{t^\sigma-1}\right),
\end{equation}
and
\begin{equation}\label{ell-thm-pf-eq3}
\frac{d}{dt}E_2(t)=-E_1(t)\frac{q'}{q}-3F(t)^\sigma
\tau^{(\sigma)}(t)\frac{p^{-1}dt^\sigma}{t^\sigma-1}.
\end{equation}
The differential equation \eqref{ell-thm-pf-eq2} implies that $E_1(t)\in K[[t]]\cap\wh B_K$.
It determines all coefficients of $E_1(t)$ except the constant term.
Note $\tau^{(\sigma)}(t)\in W[[t]]$. Taking the residue at $t=0$ of the both sides
of \eqref{ell-thm-pf-eq3}, one concludes $E_1(0)=0$.
We thus have the full description of $E_1(t)$.
Since $E_1(t)\in tK[[t]]$,
the differential equation \eqref{ell-thm-pf-eq3} implies that 
$E_2(t)\in K[[t]]\cap\wh B_K$, and
it determines all coefficients of $E_2(t)$ except the constant term.

The rest of the proof is to show $E_2(0)=-9\ln_2^{(p)}(-\nu)$.
To do this, we look at the syntomic cohomology of the singular fiber $F$.
Recall that $F$ has two components $Z,E$ and the multiplicity of $Z$ (resp. $E$) is 
$1$ (resp. $2$)
The reduced part $F_{\red}=Z\cup E$ is a Neron $2$-gon,
and the divisor $D$ intersects only with $Z$.
Let $F_\star$ be the simplicial nerve of the normalization of $F_{\red}$.
Let $i_{F_{\red}}: F_{\red}\hra \cE$ and let $i_{F_\star}:F_\star\to \cE$.
There is a commutative diagram
\[\xymatrix{
H^2_\syn(\cE,\Z_p(2))\ar[d]\ar[r]^{i^*_{F_\star}}\ar@/_80pt/[dd]_{\rho_\sigma}
&H^2_\syn(F_\star,\Z_p(2))&
H^1_\dR(F_\star/W)\ar[l]_-{\cong}\ar[d]^{u'}_\cong\\
H^2_\syn((\cE,F)/(\Delta,O,\sigma),\Z_p(2))&&H^1_\zar(F_\star,\O_{F_\star})\\
H^1_\zar(\cE,\omega^\bullet_{\cE/\Delta})\ar[u]^\cong\ar[r]^{i_{F_{\red}}^*}&
H^1_\zar(F_{\red},\omega^\bullet_{\cE/\Delta}|_{F_{\red}})\ar[ru]_{u}
}\]
where $u$ and $u'$ are the maps induced from the canonical maps
$\omega^\bullet_{\cE/\Delta}|_{F_{\red}}\to\O_{F_\star}$
and $\Omega^\bullet_{F_\star}\to\O_{F_\star}$.
By Theorem \ref{ell.diagram-thm}, 
$
\reg_{X/S}^{(\sigma)}(\xi)
$
is related to $[\xi]_\syn\in H^2_\syn(\cE,\Z_p(2))$ as follows,
\begin{equation}\label{ell-thm-pf-eq4}
\rho_\sigma([\xi]_\syn)=-E_1(t)\wh\omega-E_2(t)\wh\eta\in H^1_\zar(\cE,\omega^\bullet_{\cE/\Delta}).
\end{equation}
To compute $E_2(0)$, we shall compute the both side of
\begin{equation}\label{ell-thm-pf-eq5}
(u'\circ i^*_{F_\star})([\xi]_\syn)=-E_2(0)(u\circ i^*_{F_{\red}})
(\wh\eta)\in H^1_\zar(F_\star,\O_{F_\star}).
\end{equation}

Let
$z_0=y/(x+4)|_Z$ be a coordinate
of $Z\cong \P^1_W$. 
Let $R_1,R_2$ be the intersection points of
$Z$ and $E$ given by $z_0=\sqrt{-3},-\sqrt{-3}$
respectively.
We use another coordinate $z:=\nu^2(1-z_0)/2$ of $Z$, so that the rational functions
\[
h_1:=\frac{y-3x-4(1-t)}{-8(1-t)},\quad
h_2:=\frac{y+3x+4(1-t)}{8(1-t)}\in\O(U)^\times
\]
in the Milnor
symbol $\xi$ in \eqref{xi} satisfy 
\begin{equation}\label{h1h2-z}
h_1|_Z=z^3,\quad h_2|_Z=(1-\nu z)^3,\quad h_1|_E=h_2|_E-1.
\end{equation}
The points $R_1,R_2$ are given by $z=-1,-\nu$ respectively.
Let $\cS_n(r)_{V}=\cS_n(r)_{V/(W,F_W)}$ denote the syntomic complex
for $V=Z, E$ or $Z\cap E$.
We fix isomorphisms
\begin{align*}
H^j_\syn(F_\star,\Z/p^n(r))
&\cong H^j_\et(F_\red,\cS_n(r)_Z\op\cS_n(r)_E\os{i}{\to}\cS_n(r)_{Z\cap E})\\
H^j_\dR(F_\star/W)
&\cong H^j_\zar(F_\red,\Omega^\bullet_Z\op\Omega^\bullet_E
\os{i}{\to}\O_{Z\cap E})\\
H^j_\zar(F_\star,\O_{F_\star})
&\cong H^j_\zar(F_\red,\O_Z\op\O_E
\os{i}{\to}\O_{Z\cap E}),
\end{align*}
where $i:(f,g)
\mapsto f|_{Z\cap E}-g|_{Z\cap E}$.
We also fix an isomorphism
\begin{align*}
\alpha:
H^1_\zar(F_\star,\O_{F_\star})
\cong\Coker[H^0(\O_Z)\op H^0(\O_E)\os{i}{\to} H^0(\O_{R_1})\op 
H^0(\O_{R_2})]\os{\cong}{\to}W,
\end{align*}
where the last isomorphism
is given by $(c_1,c_2)\mapsto c_1-c_2$.

\medskip

We compute the right hand side of \eqref{ell-thm-pf-eq5}.
The map $u\circ i^*_{F_\red}$ agrees with the composition of the
maps,
\[
H^1_\zar(\cE,\omega^\bullet_{\cE/\Delta})\os{u^{\prime\prime}}{\lra} 
H^1_\zar(\cE,\O_\cE)\os{ i^*_{F_\star}}{\lra} H^1_\zar(F_\star,\O_{F_\star}).
\]
Let $y_1=2y$ and $x_1=x+3$ so that we have 
$y^2=x^3+(3x+4-4t)^2$ $\Leftrightarrow$ $y_1^2=4x_1^3-g_2x_1-g_3$.
Let $\cE=U_0\cup U_\infty$ be an affine open covering given by
$U_0=\{y_1^2=4x_1^3-g_2x_1-g_3\}$ and 
$U_\infty=\{w^2=4v-g_2v^3-g_3v^4\}$ with $v=1/x_1$, $w=y_1/x_1^2$.
Then one can compute the cohomology $H^i_\zar(\cE,{\mathscr F}^\bullet)$ for a bounded complex
 ${\mathscr F}^\bullet$ of quasi-coherent sheaves by the total complex of Cech complex
\[
\vg(U_0,{\mathscr F}^\bullet)\op\vg(U_\infty,
{\mathscr F}^\bullet)\lra \vg(U_0\cap U_\infty,{\mathscr F}^\bullet),
\quad (f_0,f_\infty)\longmapsto f_0|_{U_0\cap U_\infty}-f_\infty|_{U_0\cap U_\infty}.
\]
In particular, $H^1_\zar(\cE,{\mathscr F}^\bullet)$ is the kernel of the map
\[
\vg(U_0\cap U_\infty,{\mathscr F}^0)\times
(\vg(U_\infty,{\mathscr F}^1)\op \vg(U_\infty,{\mathscr F}^1))
\to \vg(U_0\cap U_\infty,{\mathscr F}^1)\times
(\vg(U_\infty,{\mathscr F}^2)\op \vg(U_\infty,{\mathscr F}^2))
\]
\[
(f)\times(g_0,g_\infty)\longmapsto 
(df-(g_0|_{U_0\cap U_\infty}-g_\infty|_{U_0\cap U_\infty}))\times(dg_0,dg_\infty)
\]
modulo the image of the map
\[
\vg(U_\infty,{\mathscr F}^0)\op \vg(U_\infty,{\mathscr F}^0)\lra
\vg(U_0\cap U_\infty,{\mathscr F}^0)\times
(\vg(U_\infty,{\mathscr F}^1)\op \vg(U_\infty,{\mathscr F}^1))
\]
\[
(g_0,g_\infty)\longmapsto(g_0|_{U_0\cap U_\infty}-g_\infty|_{U_0\cap U_\infty})\times(dg_0,dg_\infty)
\]
where $d:{\mathscr F}^i\to{\mathscr F}^{i+1}$ is the differentail map.
We compute
\begin{equation}\label{cech-1}
u\circ i^*_{F_\red}(\wh\eta)= i^*_{F_\star}(u^{\prime\prime}(\wh\eta))
= i^*_{F_\star}\circ u^{\prime\prime}\left(\frac{1}{\sqrt{-3}}\wh\eta_1\right),\quad
\wh\eta_1:=\frac{x_1dx_1}{y_1}
\end{equation}
where the last equality follows from the fact that $u^{\prime\prime}(dx/y)=0$.
The cohomology class $\wh\eta_1\in H^1_\zar(\omega^\bullet_{\cE/\Delta})$ 
is represented by a Cech cocycle
\[
\left(\frac{y_1}{2x_1}\right)\times
\left(\frac{x_1dx_1}{y_1},\frac{(g_2v+2g_3v^2)dv}{4w}\right)
\in \vg(U_0\cap U_\infty,\O_\cE)\times
(\vg(U_0,\omega^1_{\cE/\Delta})\op\vg(U_\infty,\omega^1_{\cE/\Delta})).
\]
Then $u^{\prime\prime}(\wh\eta)$ is represented by a Cech cocycle
$(y_1/2x_1)=(y/(x+3))\in \vg(U_0\cap U_\infty,\O_\cE)$, and hence
\eqref{cech-1} is represented by a Cech cocycle
\begin{align*}
&\left(\frac{1}{\sqrt{-3}}\frac{y}{x+3}\bigg|_Z,0\right)
\times(0,0)=
\left(\frac{1}{\sqrt{-3}}\frac{z_0(z_0^2+3)}{z_0^2+2},0 \right)\times(0,0)\\
&=
\left(\frac{1}{\sqrt{-3}}z_0,0 \right)\times(0,0)
-\left(\frac{1}{\sqrt{-3}}\frac{-z_0}{z_0^2+2},0 \right)\times(0,0)
\\
&\in \vg(U_0\cap U_\infty,\O_Z\op\O_E)\times(\vg(U_0,\O_{Z\cap E})\op 
\vg(U_\infty,\O_{Z\cap E})).
\end{align*}
Since $z_0\in \vg(U_0,\O_Z)$ and $z_0/(z_0^2+2)\in \vg(U_\infty,\O_Z)$,
this is equivalent to
\begin{align*}
&(0,0)\times
\left(\left(\frac{1}{\sqrt{-3}} \cdot\sqrt{-3},\frac{1}{\sqrt{-3}}\cdot(-\sqrt{-3})\right),
\left(\frac{1}{\sqrt{-3}} \cdot\frac{-\sqrt{-3}}{-1},\frac{1}{\sqrt{-3}}\cdot\frac{
\sqrt{-3}}{-1}\right)
\right)\\
=&(0,0)\times((1,-1),(1,-1)).
\end{align*}
This shows
\begin{equation}\label{ell-thm-pf-eq5-right}
\alpha\bigg((h\circ i^*_{F_{\red}})(\wh\eta)\bigg)=1-(-1)=2
\end{equation}
giving
the right hand side of \eqref{ell-thm-pf-eq5}.

\medskip

We write down the left hand side of \eqref{ell-thm-pf-eq5} explicitly.
Put $D_F:=D\cap F=D\cap Z$ ($=$ three points $z=0,\nu^2,\infty$).
Let $F_U:=F_{\red}\cap U=(Z\setminus D_F)\cup E$ and 
its simplicial nerve is denoted by $F_{U,\star}$. 
One has
$
h_i|_{F_U}\in \O(F_U)^\times
$
and 
\[
\xi|_{F_U}=
\{h_1|_{F_U},h_2|_{F_U}\}
\in K_2^M(\O(F_U)).
\]
We think $h_i|_{F_U}$ to be an element of $K_1(F_{U,\star})$ and
$\xi_{F_U}$ 
of being an element of $K_2(F_{U,\star})$ via the canonical maps
$K^M_i(\O(F_U))\to K_i(F_{U,\star})$.
To compute the left hand side of \eqref{ell-thm-pf-eq5}, it is enough to compute
\[
[\xi|_{F_U}]_\syn=
[h_1|_{F_U}]_\syn\cup[h_2|_{F_U}]_\syn
\in H^2_\syn((F_\star,D_F),\Z_p(2))
\]
as there is a commutative diagram
\[
\xymatrix{
H^2_\syn(F_\star,\Z_p(2))\ar[d]&
H^1_\dR(F_\star/W)\ar[l]_-{\cong}\ar[r]^-\cong_-{u'}\ar[d]&
H^1_\zar(F_\star,\O_{F_\star})\ar@{=}[d]\\
H^2_\syn((F_\star,D_F),\Z_p(2))&
H^1_\dR((F_\star,D_F)/W)\ar[l]_-{\cong}\ar[r]&
H^1_\zar(F_\star,\O_{F_\star}).
}
\]
We further replace the log syntomic cohomology with the rigid syntomic cohomology.
Take an (arbitrary) 
affine open set $E'\subset E$ such that $E'\supset Z\cap E$.
Put $Z':=Z\setminus D_F$ and $F'_U:=Z'\cup E'$ and let $F'_{U,\star}$ be the simplicial nerve.
Then
there is a commutative diagram
\[
\xymatrix{
H^2_\syn((F_\star,D_F),\Z_p(2))\ar[d]&
H^1_\dR((F_\star,D_F)/W)\ar[l]_-{\cong}\ar[r]\ar[d]&
H^1_\zar(F_\star,\O_{F_\star})\ar[d]\\
H^2_\rigsyn(F'_{U,\star},\Q_p(2))&
H^1_\dR(F'_{U,\star}/K)\ar[l]_-{\cong}\ar[r]&
H^1_\zar(F_\star,\O_{F_\star})\ot\Q.
}
\]
Thanks to the compatibility with the rigid regulator maps (Lemma \ref{syn-rig-lem}),
it is enough to compute
\begin{equation}\label{ell-thm-pf-eq6}
\reg_\rigsyn(h_1|_{F'_U})\cup\reg_\rigsyn(h_2|_{F'_U})
\in H^2_\rigsyn(F'_{U,\star},\Q_p(2)).
\end{equation} 
We take a $p$-th Frobenius $\varphi_{Z'}$ on $\O(Z')^\dag$ 
given by $\varphi_{Z'}(z)=z^p$.
Note $\varphi:Z'\to Z'$ fixes $R_1$ and $R_2$.  We also take a
$p$-th Frobenius $\varphi_{E'}$ on $\O(E')^\dag$ that fixes $R_1$ and $R_2$.
Then for $(V,\ol V)=(Z',Z)$ or $(E',E)$, 
the rigid syntomic complex $S_\rigsyn(r)_{V}:=R\varGamma_\rigsyn(V,\Q_p(r))$ 
is given as follows,
\[
S_\rigsyn(r)_{V}=
\text{Cone}[\vG(\ol V_K,\Omega^{\bullet\geq r}_{\ol V_K}(\log \partial V_K))\os{1-p^{-r}\varphi_V}\lra\Omega^\bullet_{\O(V)^\dag_K} ][-1]
\]
where $\partial V=\ol V\setminus V$ and $\Omega^1_{\O(V)^\dag_K}
=\Omega^1_{\O(V)^\dag_K/K}$ denotes the module of continuous differentials.
Moreover the rigid syntomic cohomology of $F'_{U,\star}$ is described as follows,
\[
H^j_\rigsyn(F'_{U,\star},\Q_p(r))\cong
H^j(S_\rigsyn(r)_{Z'}\op S_\rigsyn(r)_{E'}\os{i}{\to}S_\rigsyn(r)_{Z\cap E}),
\]
where $i:(f,g)\mapsto f|_{Z\cap E}-g|_{Z\cap E}$.
Under this identification,
it follows from \cite[Prop.10.3]{Be1} that we have cocyles 
\begin{align*}
\reg_\rigsyn(h_1|_{F'_U})&=\left(3\frac{dz}{ z},0\right)\times(0,0)\times(0,0)\\
\reg_\rigsyn(h_2|_{F'_U})&=\left(\frac{-3\nu dz}{1-\nu z},
3p^{-1}\log\left(\frac{(1-\nu z)^p}{1-\nu^p z^p}\right)\right)\times(0,0)\times(0,0)\\
\in&(\Omega^1_{\O(Z')^\dag_K}\op\O(Z')^\dag_K)\times
(\Omega^1_{\O(E')^\dag_K}\op\O(E')^\dag_K)\times
(\O(R_1)\op\O(R_2))
\end{align*}
for $h_i$ in \eqref{h1h2-z}, and then
\begin{align*}
\eqref{ell-thm-pf-eq6}
&=\left(-9p^{-1}\log\left(\frac{(1-\nu z)^p}{1-\nu^p z^p}\right)\frac{dz}{z},0\right)
\times(0,0)\times(0,0)\\
&=\left(9d(\ln^{(p)}_2(\nu z)),0\right)
\times(0,0)\times(0,0)\\
&\equiv
(0,0)\times(0,0)\times(9\ln^{(p)}_2(-\nu ),9\ln^{(p)}_2(-\nu^2))
\end{align*}
in $H^2_\rigsyn(F'_{U,\star},\Q_p(2))$.
Here we use the fact that $\ln^{(p)}_2(z)$ 
is an overconvergent function on $\P^1\setminus
\{1,\infty\}$
(Proposition \ref{polylog-oc}) .
This shows
\[
\alpha\left((u'\circ i^*_{F_\star})([\xi]_\syn)\right)=
9\ln^{(p)}_2(-\nu )-9\ln^{(p)}_2(-\nu^2)=18\ln^{(p)}_2(-\nu),
\]
which provides
the left hand side of \eqref{ell-thm-pf-eq5}.
Combining this with \eqref{ell-thm-pf-eq5-right}, one finally has
\[
18\ln^{(p)}_2(-\nu)=-2E_2(0).
\]
This completes the proof.
\end{pf}
\begin{cor}\label{ell-cor}
Let $a\in W$ satisfy $a\not\equiv 0,1$ mod $p$.
Let $c=F_W(a)a^{-p}$ so that $\sigma(t)=ct^p$ satisfy $\sigma(t)|_{t=a}=F_W(a)$.
Let $E_a$ be the fiber of $f$ at $t=a$, and put $U_a:=E_a\cap U$.
Let $\xi|_{E_a}\in K^M_2(\O(U_a))_{\partial=0}$ be 
the restriction of the Milnor symbol $\xi$ in \eqref{xi}, and
we think it to be an element of the Adams weight piece $K_2(E_a)^{(2)}$
under the natural map $K^M_2(\O(U_a))_{\partial=0}\to K_2(E_a)^{(2)}$.
Let \[
\reg_\rigsyn:K_2(E_a)^{(2)}\lra H^2_\rigsyn(E_a,\Q_p(2))\cong H^1_\dR(E_a/K)
\]
be the regulator map by Besser \cite{Be1}
or equivalently by Nekov\'a\v{r}-Niziol \cite{NN}.
Then\[
\reg_\rigsyn(\xi|_{E_a})=-\ve_1(a)\frac{dx}{y}-\ve_2(a)\frac{xdx}{y}
\in H^1_\dR(E_a/K).
\]
\end{cor}
\begin{pf}
Recall the diagram in Theorem \ref{ell.diagram-thm}, and take the pull-back of it
at the point $t=a$ of $S^*$.
Then $[-]_\syn$ turns out to be the regulator map $\reg_\rigsyn$
(Lemma \ref{prop:synK2X}).
Now the assertion is immediate from Theorem \ref{ell-thm}.
\end{pf}

\section{Appendix : Frobenius on totally degenerating abelian schemes}
\label{degenerate-sect}
\subsection{De Rham symplectic basis for totally degenerating abelian varieties}
\label{gm-mt-sect}
Let $R$ be a regular noetherian domain, and $I$
a reduced ideal of $R$. Let $L:=\Frac(R)$ be the fractional field.
Let $J/R$ be a $g$-dimensional 
commutative group scheme such that
the generic fiber $J_\eta$ is a principally polarized abelian variety over $L$.
If the fiber $T$ over $\Spec R/I$ is an algebraic torus,
we call $J$ a {\it totally degenerating abelian scheme over $(R,I)$}
(cf. \cite{FC} Chapter II, 4).
Assume that the algebraic torus $T$ is split.
Assume further that $R$ is complete with respect to $I$.
Then there is the uniformization $\rho:\bG_m^g\to J$ in the rigid analytic sense.
We fix $\rho$ and the coordinates $(u_1,\ldots,u_g)$ of $\bG_m^g$.
Then a matrix
\begin{equation}\label{gm-sect-3-eq1}
\underline{q}=\begin{pmatrix}
q_{11}&\cdots&q_{1g}\\
\vdots&&\vdots\\
q_{g1}&\cdots&q_{gg}
\end{pmatrix},\quad q_{ij}=q_{ji}\in L
\end{equation}
of multiplicative periods is determined up to $\mathrm{GL}_g(\Z)$, and this 
yields an isomorphism
\[
J\cong \bG_m^g/\underline{q}^\Z
\]
of abelian schemes over $R$ where $\bG_m^g/\underline{q}^\Z$ denotes 
Mumford's construction of the quotient scheme (\cite{FC} Chapter III, 4.4).

\medskip

In what follows, we suppose that the characteristic of $L$ is zero.
The morphism $\rho$ induces 
\[
\rho^*:\Omega^1_{J/R}\lra \bigoplus_{i=1}^g\wh\Omega^1_{\bG_m,i},\quad
\wh\Omega^1_{\bG_m,i}:=\varprojlim_n\Omega^1_{R/I^n[u_i,u_i^{-1}]/R}.
\]
Let 
\[
\Res_i:\wh\Omega^1_{\bG_m,i}\lra R,\quad 
\Res_i\left(\sum_{m\in\Z} a_mu^m_i\frac{du_i}{u_i}\right)=a_0
\]
be the residue map. The composition of $\rho^*$ 
and the residue map induces 
a morphism
$\Omega^\bullet_{J/R}\lra R^g[-1]$
of complexes, and hence a map 
\begin{equation}\label{tau-app}
\tau:H^1_\dR(J/R):=H^1_\zar(J,\Omega^\bullet_{J/R})\lra R^g.
\end{equation}
Let $U$ be defined by
\[
0\lra U\lra H^1_\dR(J_\eta/L)\os{\tau}{\lra} L^g\lra0.
\]
Note that the composition
$\vg(J_\eta,\Omega^1_{J_\eta/L})\os{\subset}{\to} H^1_\dR(J_\eta/L)\os{\tau}{\to} L^g
$
is bijective.
Let $\langle x,y\rangle$ denotes the symplectic pairing on $H^1_\dR(J_\eta/L)$
with respect to the principal polarization on $J_\eta$.
We call an $L$-basis
\[
\wh\omega_i,\, \wh\eta_j\in H^1_\dR(J_\eta/L),\quad 1\leq i,\,j\leq g
\]
a {\it de Rham symplectic basis} if the following conditions are satisfied. 
\begin{description}
\item[(DS1)]
$\wh\omega_i\in \vg(J_\eta,\Omega^1_{J_\eta/L})$ and
$\tau(\wh\omega_i)\in (0,\ldots,1,\ldots,0)$ where ``$1$'' is placed in the $i$-th component.
Equivalently, \[
\rho^*(\wh\omega_i)=\frac{du_i}{u_i}.
\]
\item[(DS2)]
$\wh\eta_j\in U$ and
$\langle\wh\omega_i ,\wh\eta_j\rangle=\delta_{ij}$ where
$\delta_{ij}$ is the Kronecker delta.
\end{description}
If we fix the coordinates $(u_1,\ldots,u_g)$ of $\bG_m^g$, then 
$\wh\omega_i$ are uniquely determined by {\bf(DS1)}.
Since the symplectic pairing $\langle x,y\rangle$ is annihilated on $U\ot U$ and
$\vg(J_\eta,\Omega^1_{J_\eta/L})\ot \vg(J_\eta,\Omega^1_{J_\eta/L})$,
the basis $\wh\eta_j$ are uniquely determined as well by {\bf(DS2)}.
\begin{prop}\label{gm-mt-sect-prop1}
Let $V$ be a subring of $R$.
Suppose that $(R,I)$ and $V$ satisfy the following.
\begin{description}
\item[(C)]
There is a regular integral noetherian 
$\C$-algebra $\wt R$ complete with respect to a reduced ideal $\wt I$ and an
injective homomorphism $i:R\to\wt R$ such that $i(V)\subset \C$ and
$i(I)\subset \wt I$ and 
\[
\wh\Omega^1_{L/V}\lra \wh\Omega^1_{\wt L/\C}
\]
is injective where we put $\wt L:=\Frac(\wt R)$ and 
\[
\wh\Omega^1_{L/V}:=L\ot_R\left(\varprojlim_n \Omega^1_{R_n/V}\right),\quad \wh\Omega^1_{\wt L/\C}:=\wt L\ot_{\wt R}\left(\varprojlim_n \Omega^1_{\wt R_n/\C}\right),\quad R_n:=R/I^n,\, \wt R_n:=\wt R/\wt I^n.
\]
\end{description}
Let
\[
\nabla:H^1_\dR(J_\eta/L)\lra \wh\Omega^1_{L/V}\ot_L
H^1_\dR(J_\eta/L)
\]
be the Gauss-Manin connection. Then
\begin{equation}\label{gm-mt-sect-eq3}
\nabla(\wh\omega_i)=\sum_{j=1}^g\frac{dq_{ij}}{q_{ij}}\ot \wh\eta_j,
\quad \nabla(\wh\eta_i)=0.
\end{equation}
\end{prop}
\begin{pf}
By the assumption {\bf(C)}, we may replace $J_\eta/L$ with $(J_\eta\ot_L\wt L)/\wt L$.
We may assume $R=\wt R$, $V=\C$ and $J=J_\eta\ot_R\wt R$.
There is a smooth scheme $J_S/S$ over a
connected smooth affine variety $S=\Spec A$ over $\C$
with a Cartesian square
\[
\xymatrix{
J\ar[d]\ar[r]\ar@{}[rd]|{\square}&J_S\ar[d]\\
\Spec R\ar[r]&S
}
\]
such that $\Spec R\to S$ is dominant.
Let $D\subset S$ be a closed subset such that $J_S$ is proper over $U:=S\setminus D$.
Then the image of $\Spec R/I$ is contained in $D$ since
$J$ has a totally degeneration over $\Spec R/I$. 
Thus we may replace $R$ with the completion $\wh A_D$ of $A$ by the ideal
of $D$. Let $L_D=\Frac \wh A_D$ then
$\wh \Omega^1_{L_D/\C}\cong L_D\ot_A\Omega^1_{A/\C}$.
Let ${\frak m}\subset A$ be a maximal ideal containing $I$, and $\wh A_{\frak m}$ the completion
by $\frak m$. Then $\wh A_D\subset \wh A_{\frak m}$ and $\wh\Omega_{\wh A_D/\C}
\subset \wh\Omega_{\wh A_{\frak m}/\C}$.
Therefore we may further replace $\wh A_D$ with $\wh A_{\frak m}$.
Summing up the above, it is enough work in the following situation.
\[
R=\wh A_{\frak m}\cong \C[[x_1,\ldots, x_n]]\supset I=(x_1,\ldots,x_n),
\quad V=\C,\quad L=\Frac R
\]
\[
J=\Spec R\times_AJ_S\lra \Spec R.
\]
Note $\wh \Omega^1_{L/\C}\cong L\ot_A\Omega^1_{A/\C}$.
Let $h:\cJ\to S^{\an}$ be the analytic fibration associated to $J_S/S$.
Write $J_\l=h^{-1}(\l)$ a smooth fiber over $\l\in U^{\an}:=(S\setminus D)^{\an}$.
Let
\[
\nabla:\O_{U^{\an}}\ot R^1h_*\C\lra 
\Omega^1_{U^{\an}}\ot R^1h_*\C.
\] 
be the flat connection compatible with
the Gauss-Manin connection on 
$H^1_\dR(J_S/S)|_U$ under the comparison
\[
\O_{U^{\an}}\ot R^1h_*\C\cong \O_{U^{\an}}\ot_AH^1_\dR(J_S/S).
\]
We describe a de Rham symplectic basis in $\wh\omega^{\an}_i,\wh\eta^{\an}_j
\in \O_{U^{\an}}\ot_AH^1_\dR(J_S/S)$
and prove \eqref{gm-mt-sect-eq3} for the above flat connection.
Write $J_\l=(\C^\times)^g/\underline{q}^\Z$ for $\l\in U^{\an}$.
Let $(u_1,\ldots,u_g)$ denotes the coordinates of $(\C^\times)^g$.
Let $\delta_i\in H_1(J_\l,\Z)$ be the homology cycle defined by the circle 
$|u_i|=\ve$ with $0<\ve\ll1$.
Let $\gamma_j\in H_1(J_\l,\Z)$ be the homology cycle defined by
the path from $(1,\ldots,1)$ to $(q_{j1},\ldots,q_{jg})$.
As is well-known, the dual basis $\delta_i^*,\gamma^*_j\in H^1(J_\l,\Z)$
is a symplectic basis, namely
\[
\langle\delta^*_i,\delta^*_{i'}\rangle=
\langle\gamma^*_j,\gamma^*_{j'}\rangle=0,\quad
\langle\delta^*_i,\gamma^*_j\rangle=\frac{1}{2\pi\sqrt{-1}}\delta_{ij}
\]
where $\delta_{ij}$ denotes the Kronecker delta.
We have $\wh\omega^{\an}_i=du_i/u_i$ by {\bf (DS1)},  and then 
\begin{equation}\label{gm-mt-sect-eq5}
\wh\omega_i^{\an}=2\pi\sqrt{-1}\delta_i^*+\sum_{j=1}^g\log q_{ij}\gamma^*_j.
\end{equation}
Let $\tau^B:R^1h_*\Z\to \Z(-1)^g$ be the associated map to $\tau$.
An alternative description of $\tau^B$ is
\[\tau^B(x)=\frac{1}{2\pi\sqrt{-1}}((x,\delta_1),\ldots,(x,\delta_g))\]
where $(x,\delta)$ denotes the natural pairing
on $H^1(J_\l,\Z)\ot H_1(J_\l,\Z)$.
Obviously $\tau^B(\gamma_j^*)=0$.
This implies that
$\wh\eta_j$ is a linear combination
of $\gamma_1^*,\ldots,\gamma_g^*$ by {\bf(DS2)}.
Since $\langle \wh\omega^{\an}_i,\wh\eta_j\rangle=\delta_{ij}=\langle \wh\omega^{\an}_i,
\gamma^*_j\rangle$,
one concludes 
\[
\wh\eta_j^{\an}=\gamma^*_j.
\]
Now \eqref{gm-mt-sect-eq3} is immediate from this and \eqref{gm-mt-sect-eq5}.

Let $\wh\omega_i,\wh\eta_j\in H^1_\dR(J_\eta/L)$ be the de Rham symplectic basis.
Let $x\in D^{\an}$ be the point associated to $\frak m$.
Let $V^{\an}$ be a small neighborhood of $x$ and $j:V^{\an}\setminus D^{\an}\hra S^{\an}$ an open immersion .
Obviously 
$\wh\omega^{\an}_i=du_i/u_i\in\vg(V^{\an},j_*\O^{\an})\ot_A H^1_\dR(J_S/S)$.
Thanks to the uniquness property, this implies 
$\wh\eta^{\an}_j=\gamma_j^*\in\vg(V^{\an},j_*\O^{\an})\ot_A H^1_\dR(J_S/S)$,
in other words, $\gamma_j^*\in \vg(V^{\an}\setminus D^{\an},j^{-1}R^1h_*\Q)$.
Let $\wh S_{\frak m}$ be the ring of power series over $\C$ containing
$\wh A_m$ and $\vg(V^{\an},j_*\O^{\an})$. 
There is a commutative diagram
\[
\xymatrix{
\wh A_{\frak m}\ot_A\Omega^1_{A/\C}\ot_AH^1_\dR(J_S/S)\ar[r]&
\wh S_{\frak m}\ot_A\Omega^1_{A/\C}\ot_AH^1_\dR(J_S/S)\\
\Omega^1_{A/\C}\ot_AH^1_\dR(J_S/S)\ar[r]\ar[u]
&\vg(V^{\an},j_*\O^{\an})\ot_A\Omega^1_{A/\C}\ot_A H^1_\dR(J_S/S)\ar[u]
}
\]
with all arrows injective.
Hence the desired assertion for
$\wh\omega_i,\wh\eta_j$ can be reduced
to that of $\wh\omega_i^{\an},\wh\eta_j^{\an}$.
This completes the proof.
\end{pf}

\subsection{Frobenius on De Rham symplectic basis}
Let $V$ be a complete discrete valuation ring such that
the residue field $k$ is perfect and of characteristic $p$, and the fractional field
$K:=\Frac V$ is of characteristic zero.
Let $F_V$ be
a $p$-th Frobenius endomorphism on $V$.

Let $A$ be an integral flat noetherian $V$-algebra
equipped with
a $p$-th Frobenius endomorphism $\sigma$ on $A$ which is
compatible with $F_V$.
Assume that $A$ is $p$-adically complete and separated
and that there is a family $(t_i)_{i\in I}$ of finitely many elements of $A$
such that it forms a $p$-basis of $A_n:=A/p^{n}A$ over 
$V_n:=V/p^{n}V$ for all $n\geq 1$
in the sense of \cite[Definition 1.3]{Ka1}. 
The latter assumption is equivalent to that
$(t_i)_{i\in I}$ forms a $p$-basis of $A_1$ over $V_1$ since $A$ is flat over $V$
(loc.cit. Lemma 1.6).
Then $\Omega^1_{A_n/V_n}$ is a free $A_n$-module with basis $(dt_i)_{i\in I}$
(\cite{Ka1} Lemma (1.8)).
Write $A_K:=A\ot_VK$.
\begin{defn}
We define the category $\FMIC(A_K,\sigma)$ as follows.
An object is a triplet $(M,\nabla,\Phi)$ where
\begin{itemize}
\item
$M$ is a locally free $A_K$-module of finite rank, 
\item
$\nabla:M\lra \wh\Omega^1_{A/V}\ot_A M$ is an integrable connection where 
$\wh\Omega^1_{A/V}:=\varprojlim_n \Omega^1_{A_n/V_n}$ a free $A$-module
with basis $\{dt_i\}_{i\in I}$,
\item
$\Phi:\sigma^*M\to M$ is a horizontal $A_K$-linear map.
\end{itemize}
A morphism in $\FMIC(A_K,\sigma)$ is a $A$-linear map of the underlying $A$-modules
which is commutative with $\nabla$ and
$\Phi$.
Let $L:=\Frac(A)$ be the fractional field.
The category $\FMIC(L,\sigma)$ is defined in the same way
by replacing $A$ with $L$, and $\wh\Omega^1_{A/V}$ with $L\ot_A\wh\Omega^1_{A/V}$.
\end{defn}
\begin{lem}[{\cite[6.1]{EK}}]\label{EK-lem}
Let $\sigma'$ be another $F_V$-linear $p$-th Frobenius on $A$.
Then there is the natural equivalence
\[
\FMIC(A,\sigma)\os{\cong}{\lra}
\FMIC(A,\sigma'),\quad(M,\nabla,\Phi)\longmapsto(M,\nabla,\Phi')
\]
of categories, where $\Phi'$ is defined in the following way.
Let $(\partial_i)_{i\in I}$ be the basis of $\wh T_{A/V}:=\varprojlim_n 
(\Omega^1_{A_n/V_n})^*$
which is the dual basis of $(dt_i)_{i\in I}$.
Then
\[
\Phi'=\sum_{n_i\geq 0}\left(\prod_{i\in I}(\sigma'(t_i)-\sigma(t_i))^{n_i}\right)\Phi
\prod_{i\in I}\partial_i^{n_i}.
\]
\end{lem}

Let $f:X\to\Spec A$ be a projective smooth morphism.
Write $X_n:=X\times_VV/p^{n+1}V$.
Then one has an object
\[
H^i(X/A):=(H^i_\dR(X/A)\ot_V K,\nabla,\Phi)\in \FMIC(A_K)
\]
where $\Phi$ is induced from the Frobenius on crystalline cohomology
via the comparison (\cite{BO} 7.4)
\[H^\bullet_\crys(X_0/A)\cong \varprojlim_n H^\bullet_\dR(X_n/A_n)\cong
H^\bullet_\dR(X/A).\]
Assume that there is a smooth $V$-algebra $A^a$ and a smooth projective morphism
$f^a:X^a\to \Spec A^a$ with a Cartesian diagram
\[
\xymatrix{
X\ar[r]\ar[d]\ar@{}[rd]|{\square}&\Spec A\ar[d]\\
X^a\ar[r]&\Spec A^a
}
\]
such that $\Spec A\to \Spec A^a$ is flat.
One has the overconvergent $F$-isocrystal
$R^if^a_{\rig,*}\O_{X^a}$ on $\Spec A^a_0$ (\cite{Et} 3.4.8.2).
Let
\[
H^i(X^a/A^a):=(H^i_\rig(X^a_0/A^a_0),\nabla,\Phi)\in \FMIC^\dag(A^a_K)
\] 
denote the associated object via the natural equivalence
$F\text{-}\mathrm{Isoc}^\dag(A^a_0)\cong\FMIC^\dag(A^a_K)$ (\cite{LS} 8.3.10).
The comparison
$H^i_\rig(X^a_0/A^a_0)\cong H_\dR^i(X^a/A^a)\ot_{A^a}(A^a)^\dag_K$ 
in \eqref{eq:comparisonisom2}
induces an isomorphism
 $H^i(X/A)\cong A_K\ot_{(A^a)^{\dag}_K}H^i(X^a/A^a)$ in $\FMIC(A_K)$.
\begin{defn}[Tate objects]\label{frob-sect-def1}
For an integer $r$, a Tate object $A_K(r)$ is defined to be the triplet $(A_K,\nabla,\Phi)$ 
scuh that $\nabla=d$
is the usual differential operator, and $\Phi$ is a multiplication by $p^{-r}$.
\end{defn}
We define for $f\in A\setminus\{0\}$
\[
\log^{(\sigma)}(f):=p^{-1}\log\left(\frac{f^p}{f^\sigma}\right)=-\sum_{n=1}^\infty
\frac{p^{n-1}g^n}{n},
\quad
\frac{f^p}{f^\sigma}=1-pg
\]
which belongs to the $p$-adic completion of the subring $A[g]\subset L$.
In particular, if $f\in A^\times$, then $\log^{(\sigma)}(f)\in A$.
\begin{defn}[Log objects]\label{frob-sect-def2}
Let $\uq=(q_{ij})$ be a $g\times g$-symmetric matrix with $q_{ij}\in A^\times$.
We define a log object $\Log(\uq)=(M,\nabla,\Phi)$ in $\FMIC(A,\sigma)$ to be the following.
Let 
\[
M=\bigoplus_{i=1}^gA_Ke_i\op\bigoplus_{i=1}^gA_Kf_i
\]
be a free $A_K$-module with a basis $e_i,f_j$.
The connection is defined by
\[
\nabla(e_i)=\sum_{j=1}^g \frac{dq_{ij}}{q_{ij}}\ot f_j,\quad \nabla(f_j)=0
\]
and the Frobenius $\Phi$ is defined by
\[
\Phi(e_i)=e_i-\sum_{j=1}^g \log^{(\sigma)}(q_{ij})f_j,\quad
\Phi(f_j)=p^{-1}f_j.
\]
\end{defn}
It is immediate to check that the log objects are compatible under the natural
equivalence in Lemma \ref{EK-lem}.
In this sense, our $\Log(\uq)$ does not depend on $\sigma$.
By definition there is an exact sequence
\[
0\lra \bigoplus_{j=1}^gA_K(1)f_j\lra \Log(\uq)\lra \bigoplus_{i=1}^gA_K(0)e_i\lra0.
\]

\begin{thm}\label{frob-sect-thm1}
Let $R$ be a flat $V$-algebra which is 
a regular noetherian domain complete with respect to a reduced ideal $I$.
Suppose that $R$ has a $p$-th Frobenius $\sigma$.
Let $J$ be a totally degenerating abelian scheme with a principal polarization over $(R,I)$
in the sense of \S \ref{gm-mt-sect}.
Let $\Spec R[h^{-1}]\hra \Spec R$ be an affine open set such that $J$ is proper
over $\Spec R[h^{-1}]$ and 
$q_{ij}\in R[h^{-1}]^\times$ where 
$\uq=(q_{ij})$ is the multiplicative periods as in \eqref{gm-sect-3-eq1}.
Suppose that $R/pR[h^{-1}]$ has a $p$-basis over $V/pV$. 
Let $A=R[h^{-1}]^\wedge$ be the $p$-adic completion of $R[h^{-1}]$.
Put $L:=\Frac(A)$ and $J_A:=J\ot_RA$.
Let $J_\eta$ be the generic fiber of $J_A$. 
Then there is an isomorphism
\begin{equation}\label{frob-sect-eq1}
(H^1_\dR(J_\eta/L),\nabla,\Phi)\ot_{A_K} A_K(1)\cong \Log(\underline{q})\in 
\FMIC(L)
\end{equation}
which sends the de Rham symplectic basis $\wh\omega_i,\wh\eta_j\in H^1_\dR(J_\eta/L)$ 
to $e_i$, $f_j$ respectively.
\end{thm}
\begin{pf}
Let $q_{ij}$ be indeterminates with $q_{ij}=q_{ji}$, and 
$t_1,\ldots,t_r$ ($r=g(g+1)/2$) are products $\prod q_{ij}^{n_{ij}}$ such that 
$\sum n_{ij}x_ix_j$ is positive semi-definite and they give a $\Z$-basis of 
the group of the symmetric pairings.
Let $J_q=\bG_m^g/\uq^\Z$ be Mumford's construction of tbe quotient group scheme
over a ring $\Z_p[[t_1,\ldots,t_r]]$ (\cite{FC} Chapter III, 4.4).    
Then there is a Cartesian square
\[
\xymatrix{
J\ar@{}[rd]|{\square}\ar[r]\ar[d]&J_q\ar[d]\\
\Spec R\ar[r]&\Spec \Z_p[[t_1,\ldots,t_r]]}
\]
such that the bottom arrow sends $t_i$ to an element of $I$
by the functoriality of Mumford's construction (\cite{FC} Chapter III, 5.5).
Thus we may reduce the assertion to the case of 
$J=J_q$, $R=\Z_p[[t_1,\ldots,t_r]]$, $I=(t_1,\ldots,t_r)$ and $h=\prod q_{ij}$.
Since $\sLog(\ul q)$ and $H^1_\dR(J_\eta/L)$ are compatible
under the natural equivalence in Lemma \ref{EK-lem},
we may replace the Frobenius $\sigma$ on $R$ with a suitable one.
Thus we may assume that it is given as $\sigma(q_{ij})=q_{ij}^p$
and $\sigma(a)=a$ for $a\in \Z_p$. 
Under this assumption $\log^{(\sigma)}(q_{ij})=0$ by definition.
Therefore our goal is to show
\begin{equation}\label{frob-sect-eq2}
\nabla(\wh\omega_i)=\sum_{j=1}^g \frac{dq_{ij}}{q_{ij}}\ot \wh\eta_j,\quad 
\nabla(\wh\eta_j)=0,
\end{equation}
\begin{equation}\label{frob-sect-eq3}
\Phi(\wh\omega_i)=p\wh\omega_i,\quad
\Phi(\wh\eta_j)=\wh\eta_j.
\end{equation}
Since the condition {\bf (C)} in Proposition \ref{gm-mt-sect-prop1} is satisfied,
\eqref{frob-sect-eq2} is nothing other than \eqref{gm-mt-sect-eq3}.
We show \eqref{frob-sect-eq3}.
Let $\uq^{(p)}:=(q_{ij}^p)$ and $J_{q^p}:=\bG_m^g/(\uq^{(p)})^\Z$.
Then there is the natural morphism $\sigma_J:J_{q^p}\to J_q$ such that the following diagram
is commutative.
\[
\xymatrix{
\bG_m^g\ar[d]\ar@{=}[r]&\bG_m^g\ar[d]\\
J_{q^p}\ar[d]\ar[r]^{\sigma_J}&J_q\ar[d]\\
\Spec R\ar[r]^\sigma&\Spec R}
\]
Let $[p]:J_{q^p}\to J_{q^p}$ denotes the multiplication by $p$ with respect to the
commutative group scheme structure of $J_q$.
It factors through the canonical surjective morphism $J_{q^p}\to J_q$ so that we have
$[p]':J_q\to J_{q^p}$. Define $\varphi:=\sigma_J \circ[p]'$. 
Under the uniformization $\rho:\bG_m^g\to J_q$,
this is compatible with
a morphism $\bG_m^g\to \bG_m^g$ given by $u_i\to u^p_i$ and $a\to \sigma(a)$ for
$a\in R$, 
which we also write $\Phi$. 
Therefore
\[
\Phi=\varphi^*:H^1_\dR(J_\eta/L)\lra H^1_\dR(J_\eta/L).
\]
In particular $\Phi$ preserves the Hodge filtration, so that 
$\Phi(\wh\omega_i)$ is again a linear combination of $\wh\omega_i$'s.
Since 
\[
\rho^*\Phi(\wh\omega_i)=
\Phi\rho^*(\wh\omega_i)=\Phi\left(\frac{du_i}{u_i}\right)=p\frac{du_i}{u_i}
\]
one concludes $\Phi(\wh\omega_i)=p\wh\omega_i$.
On the other hand, since $\Phi(\ker\nabla)\subset \ker\nabla$ and
$\ker\nabla$ is generated by $\wh\eta_j$'s by \eqref{frob-sect-eq2},
$\Phi(\wh\eta_j)$ is again
a linear combination of $\wh\eta_i$'s.
Note
\[
\langle \Phi(x),\Phi(y)\rangle=p\langle x,y\rangle.
\]
Therefore 
\[
\langle \Phi(\wh\eta_j),p\wh\omega_j\rangle=
\langle \Phi(\wh\eta_j),\Phi(\wh\omega_j)\rangle=
p\langle \wh\eta_j,\wh\omega_j\rangle.
\]
This implies $\Phi(\wh\eta_j)=\wh\eta_j$, so we are done.
\end{pf}


\begin{thebibliography}{AA-AA}
    \bibitem[A-Bal]{AndreBaldassarri} Y.\ Andr\'e and F.\ Baldassarri, \emph{De Rham cohomology of differential modules over algebraic varieties}. Progress in Mathematics 189, Birkh\"auser, 2000.


 \bibitem[A]{New}
Asakura, M.,
{\it New $p$-adic hypergeometric functions and syntomic regulators}.
arXiv.1811.03770.

 \bibitem[A-C]{AC}
 Asakura, M., Chida, M.,
{\it A numerical approach toward the 
$p$-adic Beilinson conjecture for elliptic curves over $\Q$}.
arXiv:2003.08888
\bibitem[A-M]{atiyah-mac}
Atiyah, M., Macdonald, I.:
{\it Introduction to Commutative Algebra}. 
Addison-Wesley Series in Math. Westview Press, Boulder, CO, 2016.

    \bibitem[Bal-Ch]{BaldassarriChiarellotto} F.\ Baldassarri and B.\ Chiarellotto,
        \emph{Algebraic versus rigid cohomology with logarithmic coefficients.}
        In \emph{Barsotti Symposium in Algebraic Geometry}, Perspectives in Math.
        \textbf{15} (1994), 11--50.
    \bibitem[Ban1]{Bannai00} Bannai, K.,
\emph{Rigid syntomic cohomology and $p$-adic polylogarithms}.
        J.\ reine angew.\ Math., \textbf{529} (2000), 205--237.
    \bibitem[Ban2]{Bannai02} K.\ Bannai, \emph{``On the $p$-adic realization of elliptic polylogarithms for CM-elliptic curves''}, Duke Math.\ Journ. \textbf{113}, No. 2 (2002), 193--236.


\bibitem[BBDG]{BBDG}
A. Beilinson, J. Bernstein, P. Deligne: {\it Faisceaux pervers.} In: Analyse et topologie
sur les espaces singuliers I, Ast\'erisque \textbf{100} (1982)




\bibitem[Ber1]{Berthelot81} Berthelot, P.,
\emph{G\'eom\'etrie rigide et cohomologie des vari\'et\'es alg\'ebriques
        de caract\'eristique $p$}. Groupe de travail d'analyse ultram\'etrique, \textbf{9}, No. 3 (1981).

\bibitem[Ber2]{Berthelot97} Berthelot, P.,
\emph{Finitude et puret\'e cohomologique en cohomologie rigide avec un appendice par Aise Johan de Jong}.
Inventiones mathematicae \textbf{128} (1997), 329--377.

        \bibitem[Ber3]{Berthelot96} Berthelot, P.,
        \emph{Cohomologie rigide et cohomologie rigide \'a supports propres,
        Premi\`ere partie.},
        preprint available at https://perso.univ-rennes1.fr/pierre.berthelot/.


\bibitem[B-O]{BO}
Berthelot, P. and Ogus, A.,
{\it Notes on crystalline cohomology.} 
Princeton University Press, Princeton, N.J.; University of Tokyo Press, Tokyo, 1978. vi+243

\bibitem[Be1]{Be1}
Besser, A.,
{\it Syntomic regulators and $p$-adic integration. I. Rigid syntomic regulators.} 
Proceedings of the Conference on $p$-adic Aspects of the Theory of Automorphic Representations (Jerusalem, 1998). 
Israel J. Math.  {\bf120}  (2000),  part B, 291--334. 
\bibitem[Be2]{Be2}
\bysame,
{\it Syntomic regulators and p -adic integration. II. $K_2$ of curves.} 
Proceedings of the Conference on $p$-adic Aspects of the Theory of Automorphic Representations (Jerusalem, 1998). 
Israel J. Math.  {\bf120}  (2000),  part B, 335--359. 
\bibitem[B-K]{BK}
   Bloch, S., Kato, K.,
   $L$-functions and Tamagawa numbers of motives.
   In: Cartier, P., Illusie, L., Katz, N. M., Laumon, G.,
   Manin, Yu. I., Ribet, K. A. (eds.)
   {\it The Grothendieck Festscherift I},
   (Progr.\ Math.\ 86), pp.\ 333--400,
   Boston, Birkh\"{a}user, 1990


\bibitem[Co]{Colmez}
 Colmez, P.,
 {\it Fonctions $L$ $p$-adiques}.
 S\'eminaire Bourbaki, Vol. 1998/99, Ast\'erisque No. 266 (2000), Exp. No. 851, 3, 21--58.
\bibitem[Dw]{Dwork-p-cycle}
Dwork, B.,
{\it $p$-adic cyles}.
Publ. Math. IHES, tome 37 (1969), 27--115.
\bibitem[E-K]{EK}
Emerton, M. and  Kisin, M.: {\it
An introduction to the Riemann-Hilbert correspondence for unit $F$-crystals}.  
Geometric aspects of Dwork theory. Vol. I, II,  677--700, Walter de Gruyter, Berlin, 2004.


\bibitem[Et]{Et}
J-Y. \'{E}tesse, {\it Images directes I: Espaces rigides analytiques et images directes}. 
J. Th\'{e}or. Nombres Bordeaux 24 (1)
(2012) 101--151.

\bibitem[F-C]{FC}
Faltings, G. and  Chai, C-L,.: {\it Degeneration of Abelian Varieties.}
Ergebnisse der Mathematik und ihrer Grenzgebiete (3), 22. Springer-Verlag, Berlin, 1990. 
\bibitem[F-M]{FM}
   Fontaine, J.-M., Messing, W.:
   $p$-adic periods and $p$-adic etale cohomology.
   In: Ribet, K. A. (ed.) {\it Current Trends in Arithmetical Algebraic Geometry},
   (Contemp.\ Math.\ 67), pp.\ 179--207, Providence, Amer.\ Math.\ Soc., 1987.

    \bibitem[G]{Gerkmann08} R.\ Gerkmann, \emph{Relative rigid cohomology and point counting on families of elliptic curves}, J.\ Ramanujan Math.\ Soc.\ \textbf{23}, No. 1 (2008), 1--31.
\bibitem[Ka1]{Ka1} Kato, K.,
{\it The explicit reciprocity law and the cohomology of Fontaine-Messing}. 
Bull. Soc. Math. France  {\bf119}  (1991),  no. 4, 397--441.  
\bibitem[Ka2]{Ka2} \bysame,
{\it On $p$-adic vanishing cycles (application of ideas of Fontaine-Messing)}. In:
   {Algebraic geometry, Sendai, 1985} (Adv.\ Stud.\ Pure Math.\ 10),
   pp.\ 207--251, Amsterdam, North-Holland, 1987.
   
\bibitem[Ka3]{Ka3}
\bysame, {\it 
Logarithmic structures of Fontaine-Illusie.} Algebraic analysis, geometry,
and number theory. Johns Hopkins University Press, Baltimore 1989, 191--224.
\bibitem[P]{Perrin-Riou}
Perrin-Riou, B., {\it Fonctions $L$ $p$-adiques des repr\'esentations $p$-adiques.} 
Ast\'erisque \textbf{229} (1995).

\bibitem[Lau]{Lauder}
 Lauder, A.: {\it Rigid cohomology and $p$-adic point counting}. 
J. Th\'eor. Nombres Bordeaux {\bf17} (2005), no. 1, 169--180.


\bibitem[N-N]{NN}
Nekov\'a\v{r}, J. and Niziol, W.,
{\it Syntomic cohomology and $p$-adic regulators for varieties over 
$p$-adic fields}.
With appendices by Laurent Berger and Fr\'ed\'eric D\'eglise.
Algebra Number Theory \textbf{10} (2016), no. 8, 1695--1790. 


\bibitem[Sa1]{msaito1}
Saito, M.,
{\it Modules de Hodge polarisables.} 
Publ. Res. Inst. Math. Sci.  24 (1988) no. 6, 849--995.
\bibitem[Sa2]{msaito2}
\bysame,
{\it Mixed Hodge modules}. Publ. Res. Inst. Math. Sci.  26  (1990),  no. 2, 221--333. 


\bibitem[Sh1]{Shiho}
Shiho, A.,
{\it Crystalline Fundamental Groups II --- Log Convergent Cohomology and Rigid Cohomology}.
J. Math. Sci. Univ. Tokyo
Vol. 9 (2002), No. 1, Page 1--163.

\bibitem[Sh2]{Shiho1}
A.~Shiho,
\textit{``Relative log convergent cohomology and relative rigid cohomology I''}, arXiv:0707.1742.

 \bibitem[Sh3]{Shiho3}
A.~Shiho,
\textit{``Relative log convergent cohomology and relative rigid cohomology III''}, arXiv:0805.3229.


    \bibitem[So]{Solomon} N.\ Solomon, \emph{$p$-adic elliptic polylogarithms and arithmetic applications}, thesis, 2008.

\bibitem[LS]{LS}
Le Stum, B.:{\it Rigid cohomology.} Cambridge Tracts in Mathematics, 172. Cambridge University Press, Cambridge, 2007. xvi+319 pp. 

\bibitem[Ts1]{Ts1}
   Tsuji, T.,
   $p$-adic \'{e}tale cohomology and crystalline cohomology
   in the semi-stable reduction case.
   Invent.\ Math.\ {\bf 137}, 233--411 (1999)
\bibitem[Ts2]{T}
\bysame, {\it On $p$-adic nearby cycles of log smooth families.} Bull.Soc.Math.France.
(2000), 529--575.

 \bibitem[Tz1]{TsuzukiGysin}
N.~Tsuzuki,
\textit{``On the Gysin isomorphism of rigid cohomology.''}
Hiroshima Math J.\ 29(3), 479--527, 1999.

    \bibitem[Tz2]{Tsuzuki03} 
Tsuzuki, N.,\emph{On base change theorem and coherence in rigid cohomology}, Doc.\ Math.\ Extra, 2003, 891--918.
    
\bibitem[NIST]{NIST} {\it NIST Handbook of Mathematical Functions. }
Edited by Frank W. J. Olver, Daniel W. Lozier, Ronald F. Boisvert and Charles W. Clark. 
Cambridge Univ. Press, 2010.

\end{thebibliography}
\end{document}